\documentclass[11pt, a4paper]{amsart}
\usepackage{a4wide}
\usepackage{graphicx,enumitem}
\usepackage{amsmath,amssymb}
\usepackage[english]{babel}
\usepackage{subfigure}
\usepackage{units}
\usepackage{diagbox}
\usepackage{fontenc}
\usepackage[dvipsnames]{xcolor}
\usepackage{caption}
\usepackage{ulem}
\usepackage[numbers]{natbib}

\usepackage{cancel}
\usepackage{tikz}
\usepackage{amssymb}
\DeclareSymbolFontAlphabet{\amsmathbb}{AMSb}%
\usepackage{mathbbol}
\usepackage{latexsym}
\usepackage{xspace,bm}

\usepackage[pdftex,
bookmarksopen,
colorlinks,
linkcolor=black,
urlcolor=black,
citecolor=black
]{hyperref}
\hypersetup{pdfauthor={E. Jansson, A. Lang, and M. Pereira}}

\usepackage{cleveref}

\usepackage{dsfont}

\usepackage{cases}

\usepackage{mathtools}

\newcommand{\IP}{\amsmathbb{P}}
\newcommand{\R}{\amsmathbb{R}}

\newcommand{\C}{\amsmathbb{C}}
\newcommand{\N}{\amsmathbb{N}}

\newcommand{\IS}{\amsmathbb{S}}

\newcommand{\cD}{\mathcal{D}} 
\newcommand{\cE}{\mathcal{E}}
\newcommand{\cF}{\mathcal{F}}

\newcommand{\cH}{\mathcal{H}}

\newcommand{\cL}{\mathcal{L}}
\newcommand{\cM}{\mathcal{M}}

\newcommand{\cO}{\mathcal{O}}

\newcommand{\cQ}{\mathcal{Q}}

\newcommand{\cS}{\mathcal{S}}
\newcommand{\cT}{\mathcal{T}}

\newcommand{\cW}{\mathcal{W}}

\newcommand{\cZ}{\mathcal{Z}}

\DeclareMathOperator{\E}{\amsmathbb{E}} 
\DeclareMathOperator{\Cov}{\mathsf{Cov}}

\DeclareMathOperator{\Diag}{Diag}
\newcommand{\dd}{\,\mathrm{d}}

\newcommand{\sL}{\mathsf{L}}
\newcommand{\sZ}{\mathsf{Z}}
\newcommand{\sZb}{\widehat{\mathsf{Z}}}
\newcommand{\sW}{\mathsf{W}}
\newcommand{\sWb}{\widehat{\mathsf{W}}}
\newcommand{\sP}{\mathsf{P}}

\newcommand{\cj}[1]{%
	\overline{#1}%
}

\newtheorem{lemma}{Lemma}[section]
\newtheorem{proposition}[lemma]{Proposition}
\newtheorem{theorem}[lemma]{Theorem}

\theoremstyle{remark}
\newtheorem{remark}[lemma]{Remark}

\theoremstyle{plain}

\theoremstyle{definition}

\newcommand{\psds}{amplitude spectral densities\xspace}
\newcommand{\psd}{amplitude spectral density\xspace}

\usepackage{color}
\definecolor{darkgreen}{rgb}{0,.6,0}

\newcommand{\revision}[1]{}

\usepackage[outline]{contour} 
\usetikzlibrary{angles,quotes} 
\usetikzlibrary{bending} 
\contourlength{1.0pt}
\usetikzlibrary{3d}
\tikzset{>=latex}
\colorlet{myblue}{blue!65!black}
\colorlet{mydarkblue}{blue!50!black}
\colorlet{myred}{red!65!black}
\colorlet{mydarkred}{red!40!black}
\colorlet{veccol}{green!70!black}
\colorlet{vcol}{green!70!black}
\colorlet{xcol}{blue!85!black}
\tikzstyle{vector}=[->,very thick,xcol,line cap=round]
\tikzstyle{xline}=[myblue,very thick]
\tikzstyle{yzp}=[canvas is zy plane at x=0]
\tikzstyle{xzp}=[canvas is xz plane at y=0]
\tikzstyle{xyp}=[canvas is xy plane at z=0]
\def\tick#1#2{\draw[thick] (#1) ++ (#2:0.12) --++ (#2-180:0.24)}

\usepackage{pgfplots}
\pgfplotsset{compat=newest}
\usetikzlibrary{intersections, pgfplots.fillbetween}
\usetikzlibrary{patterns}
\usetikzlibrary{arrows}
\usetikzlibrary{calc}

\newcommand{\q}[1]{``#1''}

\renewcommand{\Re}{\operatorname{Re}}
\renewcommand{\Im}{\operatorname{Im}}

\begin{document}
	
	\title[Galerkin--Chebyshev Gaussian random fields using SFEM]{Non-stationary Gaussian random fields on hypersurfaces: Sampling and strong error analysis}

\author[E.~Jansson]{Erik Jansson} \address[Erik Jansson]{\newline Department of Mathematical Sciences
	\newline Chalmers University of Technology \& University of Gothenburg
	\newline S--412 96 G\"oteborg, Sweden.} 
\email[]{erikjans@chalmers.se}

\author[A.~Lang]{Annika Lang} \address[Annika Lang]{\newline Department of Mathematical Sciences
	\newline Chalmers University of Technology \& University of Gothenburg
	\newline S--412 96 G\"oteborg, Sweden.} \email[]{annika.lang@chalmers.se}
\author[M.~Pereira]{Mike Pereira} \address[Mike Pereira]{\newline Department of Geosciences and Geoengineering
	\newline Mines Paris \newline  PSL University
	\newline F--77 305  Fontainebleau, France.} \email[]{mike.pereira@minesparis.psl.eu}

\thanks{ Acknowledgment: The authors thank Mihály Kovács, Julie Rowlett and Iosif Polterovich for helpful discussions.
	This work was supported in part by the European Union (ERC, StochMan, 101088589), by the Mines Paris - PSL / INRAE  “Geolearning” chair, by the Swedish Research Council (VR) through grant no.\ 2020-04170, by the Wallenberg AI, Autonomous Systems and Software Program (WASP) funded by the Knut and Alice Wallenberg Foundation, and by the Chalmers AI Research Centre (CHAIR). Funded by the European Union.
	Views and opinions expressed are however those of the author(s) only and do not necessarily reflect those of the European Union or the European Research Council Executive Agency. Neither the European Union nor the granting authority can be held responsible for them.
}

\subjclass{60G60, 60H35, 60H15,65C30, 60G15, 58J05, 41A10, 65N30,65M60}

\keywords{Gaussian random fields, non-stationary random fields, stochastic partial differential equations, surface finite element method, Chebyshev approximation, Gaussian processes.}

	\begin{abstract}
		A flexible model for non-stationary Gaussian random fields on hypersurfaces is introduced.
		The class of random fields on curves and surfaces is characterized by an	\psd of a second order elliptic differential operator.
		Sampling is done by a Galerkin--Chebyshev approximation based on the surface finite element method and Chebyshev polynomials.
		Strong error bounds are shown with convergence rates depending on the smoothness of the approximated random field.
		Numerical experiments that confirm the convergence rates are presented.
	\end{abstract}

	\maketitle

\section{Introduction}
\label{sec:intro}

Random fields are  powerful tools for modeling spatially dependent data.
They have found uses in a wide range of applications, for instance in geostatistics, cosmological data analysis, climate modeling, and biomedical imaging \cite{MP11, Farag2014}.
One challenge in the modeling of spatial data is non-stationary behavior, i.e., different behaviors in different parts of the domain.  
Another challenge is that the domain may be a non-Euclidean space, for instance, a surface such as the sphere or on the cortical surface of the brain.
In this paper, we present a surface finite element-based method to sample a flexible class of non-stationary random fields on \revision{smooth} curves and surfaces and show its strong convergence. 
The method, building on the foundational work for stationary fields introduced in \cite{Lang2023}, is an extension of the stochastic partial differential equation (SPDE) approach pioneered by \cite{W63} and popularized by \cite{Lindgren2011}. 
The idea behind our method is to color white noise by applying a function of an elliptic differential operator~$\cL$. 
Formally, we study Gaussian random fields on curves and surfaces of the form 
\begin{align}
	\label{eq:fielddef}
	\cZ = \gamma(\cL)\cW,	
\end{align}
where $\cL$ is an elliptic differential operator, $\cW$ denotes white noise, $\gamma$ is a function, called \emph{\psd} in analogy to the spectral analysis of time signals \cite{heinzel2002spectrum}.
\revision{In this paper, \emph{non-stationarity} means that the coefficients of the differential operator $\cL$ vary over the domain. This allows for a range of local behaviors that cannot be achieved with constant coefficients.}
If $1/\gamma$ is well-defined over $\R_+$,  one may formally view $\cZ$  as the solution to the stochastic partial differential equation 
\begin{align*}
	(1/\gamma)(\cL)\cZ = \cW.	
\end{align*}

\begin{figure}[tb]
	\subfigure[A random field where the choice of coefficients \revision{is inspired by the continents}.]{\includegraphics[width=0.3\textwidth, trim=110mm 5mm 110mm 5mm, clip]{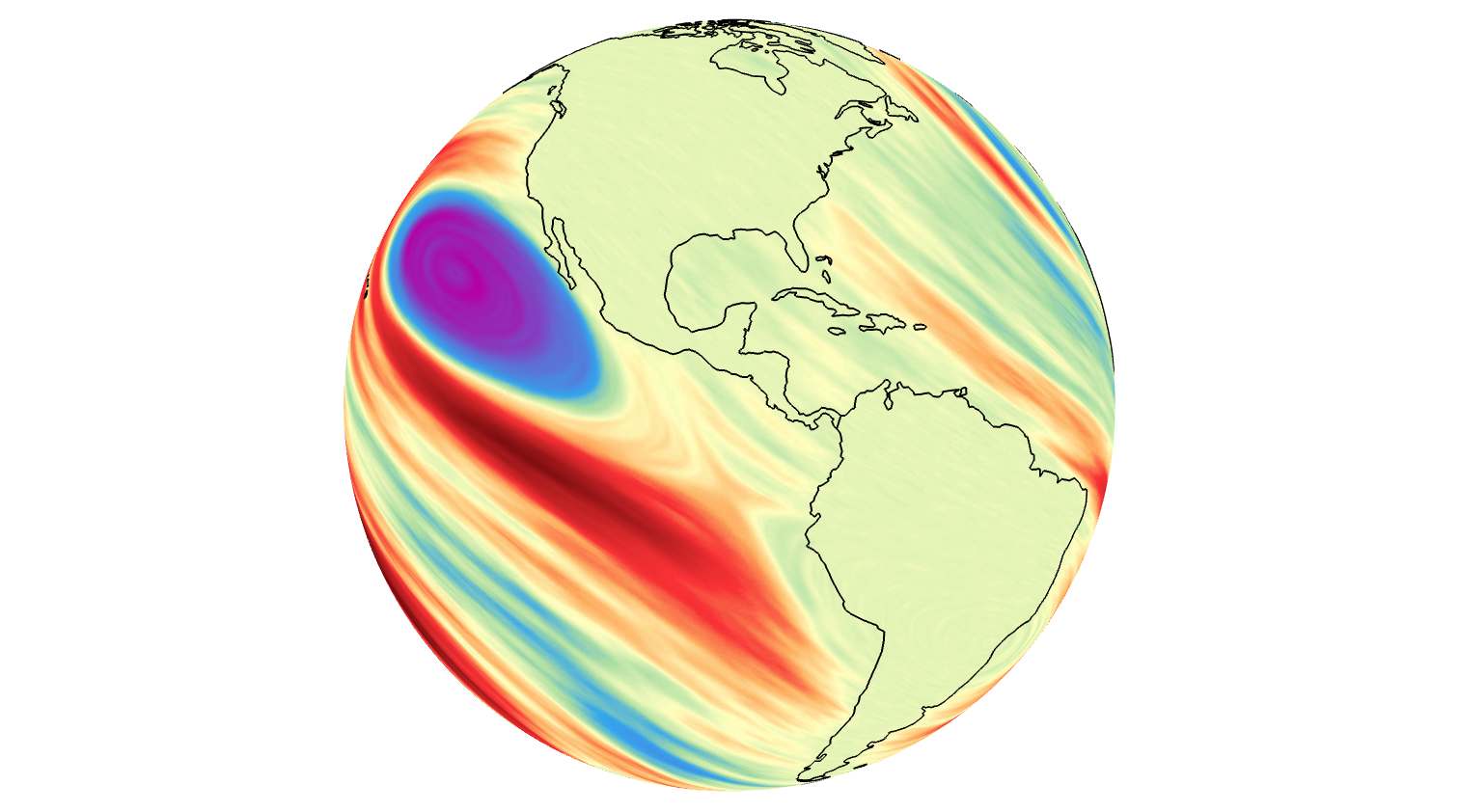}
		\label{fig:earth}}\hspace*{1em}
	\subfigure[A random field on \revision{the cortical surface} that is only locally activated.]{\includegraphics[width=0.3\textwidth, trim=150mm 30mm 150mm 30mm, clip]{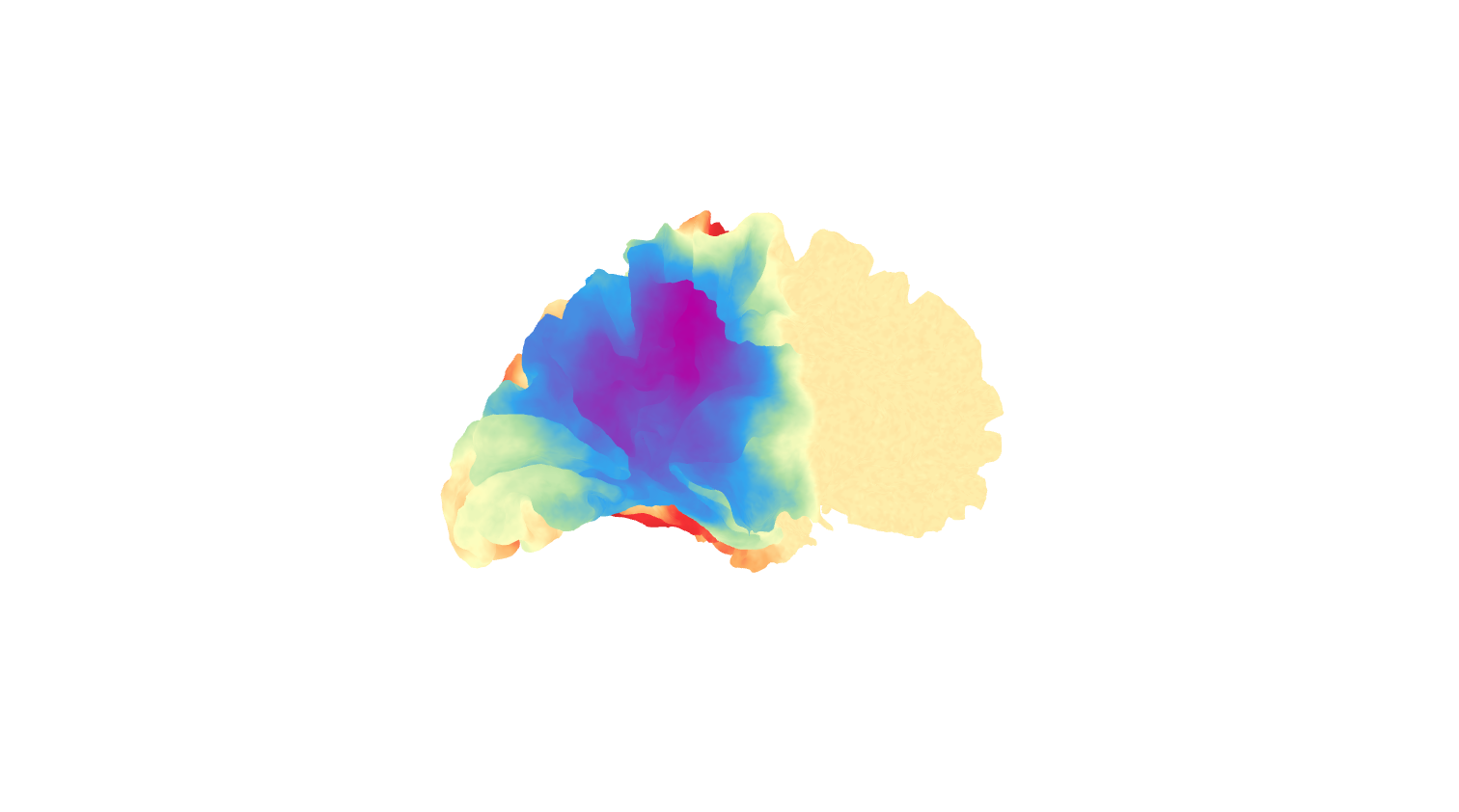}
		\label{fig:brain}}\hspace*{1em}
	\subfigure[A random field on a \revision{smooth one-dimensional} domain.]{\includegraphics[width=0.3\textwidth]{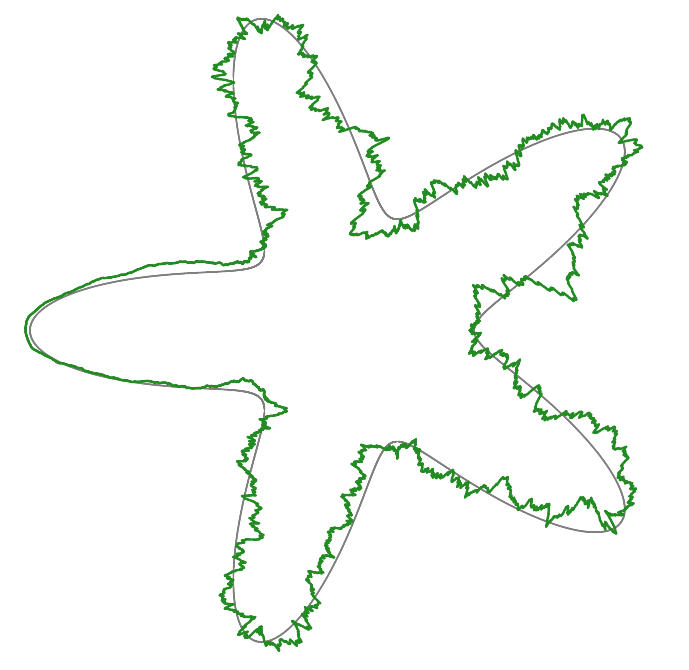}
		\label{fig:third_example}}
	\caption{Examples of random field samples generated with our method.\label{fig:nice_ex}}
\end{figure}

For instance, consider the three examples depicted in \Cref{fig:nice_ex}. To generate random field samples, there are two components of $\cL$ that we can vary: the diffusion matrix and the potential. 
In \Cref{fig:earth}, we use a diffusion matrix to give the field preferred directions, more specifically it elongates field in the northwest-southeast direction. The potential is large over the continents and small over the oceans, effectively ``turning off'' the random field over land.
In \Cref{fig:brain}, the potential is small in the front of the brain and large elsewhere, so that the field is only large in the front of the brain. 
Finally, \Cref{fig:third_example} shows the method used to generate a non-stationary random field in the one-dimensional case. To illustrate the value of the field at a point, we move it in the normal direction for a distance proportional to the value of the field. 
In all three cases, we see that the field behaves locally varying over the domain. With the suggested model, we can achieve preferred directions, local activation, and local deactivation. 

The computational method we use to solve \Cref{eq:fielddef}, i.e., sample the random fields, is based on the surface finite element method (SFEM), a computational method pioneered by \cite{Dziuk1988,Dziuk2013} and that has been used in the context of the generation of Whittle--Mat\'ern random fields with Laplace--Beltrami operators \revision{in for instance} \cite{Bonito2024,JKL22,Lang2023}.
Our main mathematical contribution is a strong convergence result for equations with more general elliptic operators and \psds than fractional powers.
Using a functional calculus approach to the finite element discretization error, we obtain a strong rate of convergence of order $\cO(C_\alpha(h)h^{2\min\{\alpha-d/4;1\}})$, where $d = 2$ for surfaces and $d=1$ for curves, and $C_\alpha(h)$ is a dimension-dependent logarithmic factor. 

The SPDE approach to random fields and their approximation have been studied previously for both surfaces and Euclidean geometries, examples include \cite{Bonito2024,BorovitskiyTMD20,BKK20,Cox2020,JKL22,Lang2023,Lindgren2011,LBR22}. 
However, to the best of our knowledge, we are the first to present a strong error analysis for the non-stationary case on hypersurfaces that are only given by mesh points. 
In contrast to the method presented in \cite{HKS20}, we do not need to explicitly know the true hypersurface and perform computations on it. Instead, we only require a polyhedral approximation of the hypersurface.
Moreover,  like in \cite{Cox2020}, we achieve this without requiring explicit approximation  bounds on the eigenfunctions of the elliptic operator, and contrary to for instance \cite{BorovitskiyTMD20}, computing the eigenfunctions.
In fact, the computation of eigenfunctions, with theoretical guarantees, is a notoriously difficult problem, see e.g., \cite{Boffi2010}.
Our method circumvents this issue by introducing a Chebyshev approximation.
\revision{Further, it should be remarked that in contrast to \cite{bonito2016}, the numerical method does not rely on a quadrature approximation of an integral representation of $\gamma(\cL)$.}

Our main contribution is a powerful, efficient, and flexible tool for the modeling and  sampling of non-stationary random fields on curves and surfaces with proven accuracy. In particular, using tools from complex analysis and operator theory, we derive strong error bounds for approximations of arbitrary sufficiently smooth transformations of elliptic operators, where we do not require assumptions on the approximability of individual eigenfunctions.

The paper is structured as follows:
In \Cref{sec:det}, we introduce the relevant deterministic framework. 
We provide the necessary background on geometry and functional analysis in \Cref{sec:det_theory}. 
This is followed by a description of the main computational tool, surface finite elements in \Cref{sec:sfem}. 
Finally, \Cref{sec:det_est} provides the relevant error estimates in the deterministic setting. 
In \Cref{sec:stoch}, we collect all material in the stochastic setting. 
We introduce first the class of considered random fields in \Cref{sec:rf} and their  Galerkin--Chebyshev approximation in \Cref{sec:method}. The proof of its strong convergence is split into the SFEM error in \Cref{sec:convergence} and the Chebyshev approximation error in \Cref{sec:conv_cheb_ap}. 
In \Cref{sec:num}, we present numerical experiments that confirm the strong error bounds.  
The source code used to generate the figures is available at this address: \url{https://github.com/mike-pereira/SFEMsim}.

\section{Deterministic theory: geometry, functional analysis and finite elements}
\label{sec:det}

Before we are able to approximate random fields on hypersurfaces, we need to introduce and partially extend the existing literature on surface finite element approximations due to so far unconsidered error bounds required in our stochastic setting. We introduce the functional analytic setting in \Cref{sec:det_theory}, discuss surface finite element methods in \Cref{sec:sfem} and show error bounds in the deterministic setting in \Cref{sec:det_est}.

\subsection{Geometric and functional analytic setting}
\label{sec:det_theory}
Let $\cM\subset \R^{d+1}$ be a $d$-dimensional ($d \le 2$) compact oriented smooth hypersurface ($k\ge 2$) without boundary, i.e., for any $x_0\in\cM$, there exists an open set $U_{x_0}\subset \R^{d+1}$ containing $x_0$ and a function $\phi_{x_0}\in C^\infty(U_{x_0})$ such that $\nabla \phi_{x_0} \neq 0$ on $\cM \cap U_{x_0}$ and 
\begin{equation*}
	\cM \cap U_{x_0}=\lbrace x\in U_{x_0}, \phi_{x_0}(x)=0 \rbrace.
\end{equation*}
The tangent space of $\cM$ at $x \in \cM$ is the $d$-dimensional subspace of $\R^{d+1}$ given by $T_{x}\cM=[\nabla\phi_{x}]^\perp$ (where $\nabla$ denotes the usual gradient of functions of $C^1(\R^{d+1})$ and $\perp$ denotes the orthogonal complement in $\R^{d+1}$ with respect to the standard Euclidean inner product.).
Since $\cM$ is oriented, there exists a smooth map $\nu : \cM \rightarrow \R^{d+1}$ assigning to each point $x \in \cM$ a unit vector $\nu(x) = \pm \nabla \phi_x / \Vert \nabla\phi_x\Vert$ perpendicular to the tangent space $T_{x}\cM$.
Our hypersurface~$\cM$ is a Riemannian manifold equipped with the metric $g$ that is the pullback of the Euclidean metric on $\R^{d+1}$. 
For instance, if $\cM = \IS^2$, this results in the standard round metric.
\revision{We remark that the restriction on dimensionality arises due to the use of surface finite elements in \Cref{sec:sfem}, see \cite{Dziuk2013} for more details.}

Let  $\nabla_\cM$ be the gradient operator acting on differentiable functions of $\cM$, and let $\Delta_\cM$ denote the Laplace--Beltrami operator on $(\cM,g)$. We denote by $\dd A$ the surface measure on $\cM$, and by $L^2(\cM)$ the Hilbert space of $\dd A$-measurable square integrable complex-valued functions, equipped with the inner product $(\cdot,\cdot)_{L^2(\cM)}$ defined by
\begin{equation*}
	(u,v)_{L^2(\cM)} = \int_{\cM} u \cj{v} \dd A, \quad u,v \in  L^2(\cM).
\end{equation*}
\revision{Note that throughout this work, we consider complex-valued function spaces as we rely on complex analysis to derive our results.}
The \emph{Sobolev spaces} with smoothness index $s \in \R^+$ are then defined via Bessel potentials by 	
\begin{equation*}
	H^s(\cM)=\left(I -\Delta_{\cM}\right)^{-s/2}L^2(\cM),
\end{equation*}
with corresponding norm $\|\cdot\|_{H^s(\cM)}=\|\left(I-\Delta_{\cM}\right)^{s/2} \cdot \|_{L^2(\cM)}$.
For $s<0$,  $H^s(\cM)$ is defined as the space of distributions generated by
\begin{equation}
	\label{eq:sobolev_spaces}
	H^s(\cM)= \left\{u=\left(I-\Delta_{\cM}\right)^{k}v, ~ v \in H^{2k+s}(\cM) \right\},
\end{equation}
where $k \in \N$ is the smallest integer such that $2k+s>0$. In this case, the corresponding norm is given by $\|u\|_{H^s(\cM)} = \|v\|_{H^{2k+s}(\cM)}$.
We set $H^0(\cM)=L^2(\cM)$. 
The reader is referred to~\cite{HLS18}, \cite{strichartz1983analysis}, \cite{Triebel1985}, and references therein for more details on Sobolev spaces defined using Bessel potentials. 

In this work, we consider elliptic differential operators associated to bilinear forms  $\mathsf{A}_{\cM}$ given by 
\begin{align}
	\label{eq:biform}
	\mathsf{A}_{\cM}(u,v)=\int_{\cM} (\cD \nabla_{\cM} u) \cdot ({ \nabla_{\cM} \cj v}) \dd A + \int_\cM (V u) \cj{v} \dd A, \quad u,v\in H^1(\cM),
\end{align}
where for any $x_0\in\cM$ the diffusion matrix $\cD(x_0) = [D_{ij}(x_0)]_{i,j=1}^{d+1}$ is a real-valued, symmetric matrix such that for any $w\in T_{x_0}\cM$, $\cD(x_0) w \in T_{x_0}\cM$ and $(\cD(x_0) w)\cdot \cj{w} >0$ if $w\neq 0$.
In particular, since $T_{x_0}\cM = [\nu(x_0)^\perp]$, $\cD(x_0)$ is simply a matrix admitting $\nu(x_0)$ as an eigenvector with some eigenvalue $\mu_1(x_0) \in\R$, and such that the eigenvalues $\mu_i(x_0)$, $2\le i\le d+1$ associated with its other  eigenvectors are positive. 
Without loss of generality, we may assume that $\mu_1(x_0)=0$, meaning that $\cD\nu =0$,  that the eigenvalues $\mu_i(x_0)$, $2\le i\le d+1$ are uniformly lower-bounded and upper-bounded on $\cM$  by positive constants.
\revision{Further, we impose the condition that $D_{ij} \in C^\infty(\cM)$ for any $1\le i,j\le d+1$.}
Finally, we assume that $V \in L^\infty(\cM)$ is a real-valued function that satisfies $V_-\le V \le V_+$ for some $0 < V_- \le V_+ <+ \infty$.

Throughout this paper, let $\mathsf{A}_{\cM}$ be coercive and continuous, i.e., there exist positive constants $\delta$ and $M$ such that for all $u,v \in H^1(\cM)$,
\begin{align}
	&\mathsf{A}_{\cM}(u,u) \geq \delta \|u\|_{H^1(\cM)}^2,\label{eq:coerc} \\ 
	&\left|\mathsf{A}_{\cM}(u,v) \right | \leq M \|u\|_{H^1(\cM)} \, \|v\|_{H^1(\cM)} \label{eq:cont}.
\end{align}
Following \cite[Equation (1.33)]{Yagi2010}, $\mathsf{A}_\cM$ gives rise to an associated elliptic differential operator $\cL:H^1(\cM) \rightarrow H^{-1}(\cM)$ defined weakly by 
\begin{align*}
	\mathsf{A}_{\cM}(u,v) = \int_{\cM} (\cL u) \overline{v} \dd A, \quad u,v \in H^1(\cM).
\end{align*}
The spectral properties of this operator are detailed in the next proposition, which is proven in \Cref{app:detproofs}.

\begin{proposition}\label{prop:spectral}
	Let $\delta>0$ be the coercivity constant defined in \Cref{eq:coerc}. There exists a set of eigenpairs $\lbrace (\lambda_i, e_i)\rbrace_{i\in\N}$ of $\cL$ consisting of a sequence of increasing real-valued eigenvalues $0<\delta \le \lambda_1 \le \lambda_2 \le \cdots $ with $\lambda_i \rightarrow +\infty$ as $i\rightarrow +\infty$, and $\lbrace e_i\rbrace_{i\in\N}$ forms an orthonormal basis of~$L^2(\cM)$ where each $e_i$ is real-valued.
\end{proposition}

Since the operator $\cL$ differs from the Laplace--Beltrami operator only by a zeroth-order potential term and a diffusion function in the second order term, switching between the two operators corresponds to a change of metric on $\cM$. 
Therefore, the eigenvalue problem for $\cL$  is equivalent to that for the Laplace--Beltrami operator on $\cM$. The results in \cite{Bandara2021} imply growth rates on the eigenvalues in accordance with Weyl's law, and more specifically that there exist $ c_\cM, C_\cM >0$ such that for any $i\in\N$
\begin{equation}
	\label{eq:gbounds_ev}
	c_\cM i^{2/d} \le \lambda_i \le C_\cM i^{2/d}.
\end{equation}

As a last step in this subsection, we introduce nonlinear functions of~$\cL$ which allow later in \Cref{sec:stoch} for the definition of a variety of Gaussian random fields.
For that, we call a function $\gamma : \R_+\rightarrow \R$ an $\alpha$-\emph{\psd} if
\begin{enumerate}
	\item[1)]\label{assum:gamma1} $\gamma$ is extendable to a holomorphic function on $H_{\pi/2} \coloneqq \{z \in \C: |\arg z| < \pi/2\}$.
	\item[2)]\label{assum:gamma2} There exist constants $ C_{\gamma}>0$ and $\alpha >0$ such that for all $z \in H_{\pi/2}$,
	\begin{equation}
		\label{eq:fundamental_ineq}
		\left| \gamma(z) \right| \le C_\gamma |z|^{-\alpha}.
	\end{equation}
\end{enumerate}
Applying \revision{an \psd} to $\cL$ results in a linear operator $\gamma(\cL)$ whose action on functions $f \in L^2(\cM)$ is defined by
\begin{align}
	\label{eq:def_func_of_op}
	\gamma(\cL) f = \sum_{i=1}^\infty \gamma(\lambda_i) (f,e_i)_{L^2(\cM)} e_i,	
\end{align}
where $\lbrace (\lambda_i, e_i)\rbrace_{i\in\N}$ are the eigenpairs of $\cL$ defined in \Cref{prop:spectral}.
A typical example is the function $\gamma(\lambda) = (\kappa^2+\lambda)^{-\alpha}$ for $\alpha > d/4$ and $\kappa>0$, which can be used to obtain Whittle--Mat\'ern random fields~\cite{Lang2023}. 

\begin{remark}\label{rem:regf}
	For any $\alpha$-\psd $\gamma$ and any $f\in L^2(\cM)$, $\gamma(\cL)f \in L^2(\cM)$. In fact, we have for any $s\in [0,\alpha]$, $\Vert \cL^{s}(\gamma(\cL)f)\Vert_{L^2(\cM)}<\infty$. Indeed,
	\begin{align*}
		\Vert\cL^{s}\gamma(\cL) f\Vert_{L^2(\cM)}^2= \sum_{i=1}^\infty \vert\lambda_i^s\gamma(\lambda_i)\vert^2 \vert(f,e_i)_{L^2(\cM)}\vert^2 ,	
	\end{align*}
	where, since $\gamma$ is an $\alpha$-\psd and $\lambda_{i}\in H_{\pi/2}$ for any $i
	\in\N$, $\vert\lambda_i^s\gamma(\lambda_i)\vert \lesssim \vert\lambda_i\vert^{-(\alpha-s)}$. Using then the fact that for any $i\in\N$, $\lambda_{i}\ge \lambda_1>0$, and that $\alpha-s \ge 0$, we conclude that $\vert\lambda_i^s\gamma(\lambda_i)\vert \lesssim \vert\lambda_1\vert^{-(\alpha-s)}\lesssim 1$ and therefore $	\Vert\cL^{s}\gamma(\cL) f\Vert_{L^2(\cM)}\lesssim \sum_{i=1}^\infty \vert(f,e_i)_{L^2(\cM)}\vert^2 = \Vert f\Vert_{L^2(\cM)}<\infty$.
\end{remark}

The goal of the remainder of this section is to study the approximation of functions of the form $u = \gamma(\cL) f$, where $f \in L^2(\cM)$. Formally, if $1/\gamma$ is well-defined on the spectrum of~$\cL$, then this is the solution to the partial differential equation $(1/\gamma)(\cL)u =  f$.

\subsection{SFEM--Galerkin approximation}
\label{sec:sfem}

The idea behind the surface finite element method, as introduced by \cite{Dziuk1988}, is to work on a polyhedral approximation of the surface that is in some sense close to the true surface $\cM$. 
More precisely, fix $h>0$ and let $\cM_h$ be a piecewise polygonal surface consisting of non-degenerate simplices (for $d=2$, triangles and for $d=1$, line segments) with vertices on~$\cM$, and such that $h$ is the size of the largest simplex defined as the in-ball radius. The set of simplices making up the discretized surface is denoted by $\mathcal{T}_h$, thus meaning that
\begin{equation*}
	\cM_h=\bigcup_{T_j \in \mathcal{T}_h} T_j,
\end{equation*}
and we assume that for any two simplices in $\cT_h$, it holds that their intersection is either empty, or a common edge or vertex.    

Following \cite[Section 1.4.1]{Dziuk2013}, we assume that the triangulation $\cT_h$ is quasi-uniform, shape-regular, and that the number of simplices sharing the same vertex can be upper-bounded by a constant independent of $h$. In turn, these three assumptions imply that $N_h \propto h^{-d}$ where $N_h\in\N$ denotes the number of vertices of $\cM_{h}$. 
\begin{figure}
	\centering
	\begin{tikzpicture}[scale = 1.0,vert/.style args={of #1 at #2}{insert path={%
				#2 -- (intersection cs:first
				line={#1}, second line={#2--($#2+(0,10)$)}) }},
		vert outwards/.style args={from #1 by #2 on line to #3}{insert path={
				#1 -- ($#1!#2!90:#3$)
		}}]
		
		\coordinate (c1) at (3.5,3.5);
		\coordinate (c2) at (4.5,5);
		\coordinate (c3) at (8,4);
		\coordinate (c4) at (7.5,3);
		\coordinate (c5) at (6,1);
		\coordinate (midpoint) at (6.25,4.5);
		\coordinate (liftpoint) at (7.0,5.05);
		\draw [ultra thick] (c1) to[bend left = 20]
		(c2) to[bend left =50] (c3) to[bend left] (c4) node[shift = {(-2,2.7)}]{$\mathcal{M}$};
		\draw [very thick, gray] (c1) to
		(c2) to (c3) to (c4) node[shift = {(-2.5,1.4)}]{$\mathcal{M}_h$};
		\draw[ thick,->,blue,vert outwards={from {($(c2)!0.5!(c3)$)} by 0.5cm on line to {(c3)}}] node[shift ={(-0.3,0)}]{$\nu_h$};
		
		\draw[ thick,->,red,vert outwards={from {($(liftpoint)!0.0!(c3)$)} by 0.7cm on line to {(10.8,0)}}] node[shift ={(-0.3,0)}]{$\nu$};
		\foreach \p in {c1,c2,c3,c4}
		\fill[gray] (\p) circle(3pt);
		\fill[black](midpoint) circle(3pt) node[below left]{x};
		\draw[dashed] (midpoint) to (liftpoint);
		\fill[black](liftpoint) circle(3pt) node[shift={(0.7,0.0)}]{a(x)};
		
	\end{tikzpicture}
	\caption{One-dimensional illustration of the lift. 
		The lift is along the normal vector $\nu$ to the surface $\cM$.\label{fig:lift}}
\end{figure}

The discrete surface $\cM_h$ is close to the true surface $\cM$ in the sense that $\cM_h$ is contained in a small neighborhood around $\cM$ defined as follows. 
First, note that $\cM$ can be seen as the boundary of some bounded open set $G\subset \R^{d+1}$ with exterior normal $\nu$. Then, following \cite[Section 2.3]{Dziuk2013}, we consider that there exists some (small) $\varpi >0$ such that $\cM_h$ is contained in a so-called \emph{tubular neighborhood} $U_\varpi$ of $\cM$ defined by  
\begin{equation*}
	U_{\varpi}=\lbrace x\in\R^{d+1} : | d_s(x) | < \varpi\rbrace,
\end{equation*} 
where $d_s : \R^{d+1} \rightarrow \R$ denotes the oriented distance function given by
\begin{align*}
	d_s(x)=\begin{cases}
		\inf_{y \in \cM} |x-y|,~~ x	 \in \R^{d+1} \setminus {G},\\
		\inf_{y \in \cM} -|x-y|,~~x \in G. 
	\end{cases}
\end{align*}

We denote by $\dd A_h$ the surface measure on $\cM_h$, and by $L^2(\cM_h)$ the Hilbert space of  $\dd A_h$-measurable square integrable functions, equipped with the inner product $(\cdot,\cdot)_{L^2(\cM_h)}$ defined by
\begin{align*}
	(u_h,v_h)_{L^2(\cM_h)} = \int_{\cM_h} u_h \cj{v}_h \dd A_h, \quad u_h,v_h\in L^2(\cM_h).
\end{align*}
Following \cite[Section 1.2.1]{Bonito2020}, we denote by $\sigma :\cM_h \rightarrow \R^+$ the \textit{area element} given by $\sigma =\mathrm{d}A/\mathrm{d}A_h$, such that for all $v \in L^2(\cM)$, 
\begin{align}\label{eq:sigma}
	\int_{\cM} v \dd A = \int_{\cM_h} \sigma v^{-\ell}	\dd A_h,
\end{align}
where we next introduce the lift and its inverse denoted by $\ell$ and $-\ell$, respectively.

A key element of SFEM is that we can move between $\cM$ and $\cM_h$ using the so-called \emph{lift operator}. 
To construct the lift operator, we note that $d_s \in C^k(U_\varpi)$ for $k \geq 2$, and that for any $x\in U_\varpi$, there exists a unique $a(x)\in \cM$  such that
\begin{equation*}
	x=a(x)+d_s(x)\nu(a(x)),
\end{equation*}
where $\nu$ denotes the normal at $a(x)$ to $\cM$.
In particular, this implies that any point $x\in U_\varpi$ can be uniquely described by the pair $(a(x), d_s(x)) \in \cM \times \R$,
and this procedure defines an isomorphism $p\colon\cM_h \rightarrow\cM$ given by  
\begin{equation*}
	p(x)=x-d_s(x) \nu(a(x)),  \quad x \in \cM_h. 
\end{equation*} 
Therefore, any function $\eta\colon\cM_h \rightarrow \C$ may be lifted to~$\cM$ by $\eta^\ell = \eta \circ p^{-1} : \cM \rightarrow \C$. 
Likewise, the inverse lift of any function $\zeta: \cM \rightarrow \C$ is given by 
$\zeta^{-\ell} = \zeta \circ p : \cM_h \rightarrow \C$.
The procedure is illustrated in Figure~\ref{fig:lift} in the one-dimensional setting. 
Note that the points on the discretized surface are \emph{lifted} along the normal $\nu$ to the surface~$\cM$.

The mapping $a$ is used to define the gradient of functions on $\cM_h$ \cite{Dziuk2013}. 
More specifically, for a differentiable $\eta:\cM_h \to \C$, the gradient is given by
\begin{equation}
	\label{eq:tan_grad_def_h}
	\nabla_{\cM_h}\eta(x)=\nabla \check{\eta}(x) - (\nabla \check{\eta}(x) \cdot\nu_h(x))\nu_h(x)\in T_{x}\cM_h,
\end{equation}
where $\nu_h$ is the normal of $\cM_h$ and $\check\eta$ is the continuous extension of $\eta$ in the normal direction, defined uniquely by $\check\eta : x\in U_\varpi \mapsto \check\eta(x)=\eta^\ell(a(x))$. 
\revision{This choice of extension follows that of \cite{Dziuk2013}.}
With this definition \revision{of $\nabla_{\cM_h}$}, the Laplace--Beltrami operator on $\cM_h$ can be defined by $\Delta_{\cM_h} = \nabla_{\cM_h} \cdot \nabla_{\cM_h} $, and Sobolev spaces on $\cM_h$ are defined in complete analogy to those on~$\cM$.

The analogue of the bilinear form $\mathsf{A}_{\cM}$  on $\cM_h$ is given by
\begin{equation}\label{eq:AMh}
	\mathsf{A}_{\cM_h}(u_h,v_h)=\int_{\cM_h} (\cD^{-\ell} \nabla_{\cM_h} u_h) \cdot ( \nabla_{\cM_h} \cj v_h) \dd A_h + \int_{\cM_h} (V^{-\ell} u_h) \cj{v}_h \dd A_h, 
\end{equation}
where $ u_h,v_h\in H^1(\cM_h)$ and $\cD^{-\ell} =  [D_{ij}^{-\ell}]_{i,j=1}^{d+1}$. As in \cite{Dziuk2013}, we assume that there exists $h_0~\in~(0,1)$ small enough, such that $\mathsf{A}_{\cM_h}$ is coercive (and continuous) whenever $h \leq h_0$. 
Unless stated otherwise, we now assume that this last condition on $h$ is fulfilled.

To conclude this subsection, we introduce the (linear) finite element space $S_h$ on $\cM_h$.
The finite element space $S_h$ is defined as the complex span of the standard real-valued \textit{nodal basis} 
\begin{align*}
	\psi_1, \dots,\psi_{N_h}\colon \cM \rightarrow \R,
\end{align*} where for any $i\in \lbrace 1,\dots,N_{h}\rbrace$, $\psi_i|_T$ is a polynomial of at most degree one taking the value $1$ at the $i$-th vertex of $\cM_{h}$ and $0$ at all the other vertices, i.e.,
\begin{equation*}
	S_h = \operatorname{span}\left(\psi_1, \dots, \psi_{N_h}\right) \subset H^1(\cM_h).
\end{equation*}
By construction, $S_h$ is a vector space of dimension $N_h$. 
Its counterpart on $\cM$ is the lifted finite element space $S_h^\ell$ given by 
\begin{equation*}
	S_h^\ell=\left\{\phi_h^\ell, \phi_h \in S_h \right\} \subset H^1(\cM). 
\end{equation*}
On $S_h$, we can, as $\cL$ on~$\cM$, associate to the bilinear form $\mathsf{A}_{\cM_h}$ in \Cref{eq:AMh} a linear operator $\sL_h: S_h \rightarrow S_h$ which maps any $u_h\in S_h$ to the unique $\sL_h u_h\in S_h$ satisfying, for any $v_h \in S_h$, the equality
\begin{align*}
	\mathsf{A}_{\cM_h}(u_h,v_h) = ( \sL_h u_h ,v_h )_{L^2(\cM_h)}. 
\end{align*}
Similarly, if the bilinear form $\mathsf{A}_{\cM}$ introduced in \Cref{eq:biform} is restricted to $S_h^\ell$, we can associate it to a linear operator $\cL_h: S_h^\ell \rightarrow S_h^\ell$ that maps any $u_h^\ell\in S_h$ to the unique $\cL_h u_h^\ell\in S_h^\ell$ that satisfies, for any $v_h^\ell \in S_h^\ell$, the equality
\begin{align*}
	\mathsf{A}_{\cM}(u_h^\ell,v_h^\ell) = ( \cL_h u_h^\ell,v_h^\ell)_{L^2(\cM)}. 
\end{align*}

Since the bilinear forms $\mathsf{A}_{\cM}$ and $\mathsf{A}_{\cM_h}$ are coercive, positive definite, Hermitian and have real coefficients, 
these two operators are diagonalizable in the sense that they each give rise to a set of $N_h$ eigenpairs \cite{Hall2013}. 
On the one hand, there exists a sequence $0\le \Lambda_1^h   \leq \dots \leq \Lambda_{N_h}^h$ and an $L^2(\cM_h)$-orthonormal basis  $E_1^h,\ldots, E_{N_h}^h$ of $S_h$ such that 
\begin{equation*}
	\sL_h E_i^h = \Lambda_{i}^h E_i^h, \quad 1\le i\le N_h,
\end{equation*}
and similarly there exists a sequence $0\le \lambda_1^h   \leq \dots \leq \lambda_{N_h}^h$ and an $L^2(\cM)$-orthonormal basis  $e_1^h,\ldots, e_{N_h}^h$ of $S_h^\ell$ such that 
\begin{equation*}
	\cL_h e_i^h = \lambda_{i}^h e_i^h, \quad 1\le i\le N_h.
\end{equation*}
In particular, using the same approach as in \Cref{prop:spectral}, we can assume that the eigenfunctions $\lbrace E_i^h\rbrace_{1\le i\le N_h}$ and  $\lbrace e_i^h\rbrace_{1\le i\le N_h}$ are all real-valued.
The eigenvalues of the operators $\cL$, $\cL_h$ and $\sL_h$ are linked to one another through the following lemma, due to \cite[Lemma 3.1]{Bonito2018}, \cite[Lemma 4.1]{Bonito2024} and \cite[Theorem 6.1]{Strang2008-pr}. 

In the following, $A \lesssim B$ is shorthand for that there is a constant $C > 0 $ such that $A \leq C B$. 
\begin{lemma}[Eigenvalue error bound]
	\label{lem:ev_bound}
	Let $\lbrace\lambda_i\rbrace_{i\in\N}$ and $\lbrace\lambda_i^h\rbrace_{1\le i\le N_h}$  denote the eigenvalues of the operators $\cL$ and $\cL_h$, respectively. Then,
	\begin{align}\label{eq:SF_growth}
		\lambda_i\le \lambda_i^h,  \quad 1\le i\le N_h.
	\end{align}
\end{lemma}

Finally, we remark that the eigenvalues $\lbrace\Lambda_i^h\rbrace_{1\le i\le N_h}$ of $\sL_h$ can be linked to the eigenvalues of some classical finite element matrices. Let $\bm C$ and $\bm R$ be the so-called \emph{mass matrix} and \emph{stiffness matrix}, respectively, and defined from the nodal basis by
\begin{equation}
	\begin{aligned}
		\bm C &= \left[(\psi_k, \psi_l)_{L^2(\cM_h)}\right]_{1\le k,l\le N_h}, \quad
		\bm R & =\left[\mathsf{A}_{\cM_h}(\psi_k, \psi_l)\right]_{1\le k,l\le N_h}.
	\end{aligned}
	\label{eq:def_CG}
\end{equation}
As defined, $\bm C$ is a symmetric positive definite  matrix and $\bm R$ is a symmetric positive semi-definite matrix. Let then  $\sqrt{\bm C}\in\R^{N_h\times N_h}$ be an invertible matrix satisfying $\sqrt{\bm C}(\sqrt{\bm C})^T=\bm C$.
Then, by \cite[Corollary 3.2]{Lang2023}, the eigenvalues $\lbrace\Lambda_i^h\rbrace_{1\le i\le N_h}$ are also the eigenvalues of the matrix $\bm S\in\R^{N_h\times N_h}$ defined by
\begin{equation}
	\bm S = \big(\sqrt{\bm C}\big)^{-1}\bm R\big(\sqrt{\bm C}\big)^{-T}.
	\label{eq:def_S}
\end{equation}
Besides, if $\bm\psi$ denotes the vector-valued function given by  $\bm\psi=(\psi_1,\dots,\psi_{N_h})^T$, then the mapping $F: \R^{N_h} \rightarrow S_h$, defined by
\begin{equation*}
	F(\bm v) = F(\bm v)=\bm\psi^T\big(\sqrt{\bm C}\big)^{-T}\bm v, 
	\quad \bm v\in\R^{N_h},
	\label{eq:def_isom_eig}
\end{equation*}
is an isomorphism whose inverse maps the eigenfunctions $\lbrace E_{i}^h\rbrace_{1\le i\le N_h}$ to (orthonormal) eigenvectors of $\bm S$. This means in particular that $\bm S$ can also be written as
\begin{equation}\label{eq:ortho_S}
	\bm S
	=
	\bm V
	\Diag\big(\Lambda_{1}^h , \dots,\Lambda_{1}^{N_h }\big)
	\bm V^T,
\end{equation}
where $\bm V=\big(F^{-1}(E_1^h)| \cdots | F^{-1}(E_{N_h}^h)\big)\in\R^{N_h\times N_h}$.

\subsection{Deterministic error analysis}
\label{sec:det_est}

From now on, let us make the following assumption on the mesh size $h$. Let $\delta_0 \in (0, \delta/2)$ be fixed and arbitrary and let $h_0 \in(0,1)$ such that $h_0^{-2} > \delta_0$ and $\vert \log h_0\vert >1$.                                                                                                                                                                                                                
We  assume for the remainder of the paper that the mesh size satisfies $h\in(0,h_0)$.

Based on the introduced framework, we are now in place to quantify the error between functions of the operators $\cL$, $\cL_h$ and $\sL_h$. 
Let $P_h:L^2(\cM) \rightarrow S_h^\ell$ \revision{be} the $L^2$-projection onto $S_h^\ell$ and let $\sP_h : L^2(\cM_{h}) \rightarrow S_h$ be the $L^2$-projection onto~$S_h$. We note that the operators  $\cL$, $\cL_h$, and $\sL_h$ define norms that are equivalent to the standard Sobolev norms, i.e.,
\begin{align}\label{eq:norm_equiv}
	\begin{gathered}
		\|\cL^{1/2} v\|_{L^2(\cM)} \sim \|v\|_{H^1(\cM)},~\|\sL_h^{1/2} V_h\|_{L^2(\cM)} \sim \|V_h\|_{H^1(\cM_h)}, \\
		\|\cL_h^{1/2} v_h\|_{L^2(\cM)} \sim \|v_h\|_{H^1(\cM)},
	\end{gathered}
\end{align}
for all $v \in H^1(\cM)$, and all $v_h \in S_h^\ell$ and $V_h \in S_h$.

With that at hand we are ready to state our main result in this section.
\begin{proposition}
	\label{prop:det}
	Let $\gamma$ be an $\alpha$-\psd. There exists some constant such that for any $h\in(0,h_0)$,  for all $f \in L^2(\cM)$ and 
	for any $p\in [0,1]$ such that $\Vert \cL^pf\Vert_{L^2(\cM)}<\infty$, we have
	\begin{align*}
		\|\gamma(\cL_h)P_h f-\gamma(\cL)f &\|_{L^2(\cM)}	 \lesssim C_{\alpha+p}(h) h^{2\min\lbrace\alpha+p;1\rbrace}\|\cL^p f\|_{L^2(\cM)},
	\end{align*}
	where $C_{\alpha+p}(h)=\vert \log h\vert$ if $\alpha+p\le1$, and $C_{\alpha+p}(h)=1$ otherwise.
\end{proposition}

To prove this proposition we rely on a representation of functions of operators based on Cauchy--Stieltjes integrals, which are constructed as follows. Since the sesquilinear form defined by $\mathsf{A}_{\cM}$ is continuous and coercive, the associated operators $\cL$ and $\cL_h$ are sectorial with some (common) angle $\theta \in (0,\pi/2)$ \cite[Theorem 2.1]{Yagi2010}. Therefore, the spectra of $\cL$ and $\cL_h$ are contained in the complement of the set $G_{\theta} = \{z \in \C, \theta \leq \arg(z) \leq \pi\}$, as illustrated in \Cref{fig:Gtheta}, and the following inequalities are satisfied for any $z\in G_\theta$ (cf. \cite[Equation (2.2)]{Yagi2010}):
\begin{align}
	\|(z-\cL)^{-1}v\|_{L^2(\cM)} & \leq C_\theta |z|^{-1}\|v\|_{L^2(\cM)}, \quad v\in L^2(\cM),\label{eq:A4}\\
	\|(z-\cL_h)^{-1}v_h^\ell\|_{L^2(\cM)} & \leq C_\theta |z|^{-1}\|v_h^\ell\|_{L^2(\cM)}, \quad v_h^\ell\in S_h^\ell,\label{eq:A5}
\end{align}
where $C_\theta>0$ is a generic constant \revision{which is uniform in $h\in(0,h_0)$ and can be expressed using the constants $M$ and $\delta$ in \Cref{eq:coerc,eq:cont} (cf. \cite[Theroem 2.1 and Equantion (2.5)]{Yagi2010})}. Note in particular that by definition, $G_\theta$ is contained in the resolvent sets of $\cL$ and $\cL_h$, and that any \psd $\gamma$ is bounded, holomorphic and satisfies the inequality $| z |^{\alpha}| \gamma(z)| \lesssim 1$ for any $z\in G_\theta$. Hence, the operators $\gamma(\cL)$ and $\gamma(\cL_h)$ can be defined as functional calculi of the operators $\cL$ and $\cL_h$ as \cite[Chapter 16, Section 1.2]{Yagi2010} by 
\begin{align}\label{eq:int_rep_operator}
	\gamma(\cL) = \frac{1}{2\pi i} \int_{\Gamma} \gamma(z) (z-\cL)^{-1} \dd z
	\quad \text{and} \quad
	\gamma(\cL_h) = \frac{1}{2\pi i} \int_{\Gamma} \gamma(z)(z-\cL_h)^{-1} \dd z,
\end{align}
where $\Gamma \subset \C$ is any integral contour surrounding the spectra of $\cL$ and $\cL_h$ and contained in $G_\theta$.
In particular, these new definitions of functions of operators are independent of the choice of $\Gamma$, and coincide with the spectral definitions previously introduced in \Cref{eq:def_func_of_op} (cf. e.g. \cite[Remark 2.7]{Yagi2010}).

\revision{Let $\delta_0 \in (0, \delta/2)$ where $\delta$ denotes the coercivity constant in \Cref{eq:coerc}}. In the remainder, we \revision{define} the contour $\Gamma$ as
\begin{align}\label{def:Gamma}
	\Gamma = \Gamma_+\cup\Gamma_0\cup\Gamma_-,
\end{align}
where $\Gamma_+$ is parametrized by $g_+(t) = -te^{i\theta}$ for $t \in (-\infty, -\delta_0]$, $\Gamma_0$ by $g_0(t) = \delta_0 e^{-it}$ for $t \in [-\theta,\theta]$ and $\Gamma_-$ by $g_-(t) = te^{-i\theta}$ for $t \in [\delta_0, +\infty)$ (see \Cref{fig:nicecontour} for an illustration). \revision{Note in particular that, by definition $\delta_0$, the spectra of $\cL$ and $\cL_h$ and indeed surrounded by in $\Gamma$.}  We use this contour to prove \Cref{prop:det}, while relying on the following bounds for the resolvent error along $\Gamma$. The proof is included in \Cref{app:detproofs_err} and is an adaption of results from \cite{Fujita1991,Bonito2024,bonito2016}.

\begin{lemma}	\label{lem:lemma_error1}
	There exists some constant such that for any $h\in(0,h_0)$,  for any $z \in \Gamma$, any $f \in L^2(\cM)$, 
	for any $p\in [0,1]$ such that $\Vert \cL^pf\Vert_{L^2(\cM)}<\infty$, 
	for any $\beta\in[0,1]$ such that $p\in [0,(1+\beta)/2]$, 
	\begin{align}\label{eq:lp}
		\|(z-\cL_h)^{-1}P_h f-P_h(z-\cL)^{-1} f\|_{L^2(\cM)} \lesssim  h^{2\beta}|z|^{-(1+p-\beta)} \Vert \cL^pf\Vert_{L^2(\cM)}.
	\end{align}
\end{lemma}

We now provide a proof for \Cref{prop:det}. 

\begin{figure}
	\centering
	
	\begin{minipage}{0.5\textwidth}
		\centering
		\subfigure[Illustration of~$G_{\theta}$ being the complement of the shaded blue slice.]{
			\begin{tikzpicture}
				\def\xmax{4.0}
				\def\ymax{2.6}
				\def\R{2.5}
				\def\RR{4.0}
				\def\ang{18}
				\coordinate (O) at (0,0);
				\coordinate (R) at (\ang:\R);
				\coordinate (RR) at (\ang:\RR);
				\coordinate (-R) at (-\ang:\R);
				\coordinate (-RR) at (-\ang:\RR);
				\coordinate (X) at ({2.0*\R*cos(\ang)},0);
				\coordinate (lambda1) at ({0.7*\R*cos(\ang)},0);
				\coordinate (lambda2) at ({1.0*\R*cos(\ang)},0);
				\coordinate (lambda3) at ({1.5*\R*cos(\ang)},0);
				\coordinate (Y) at (0,{2.0*\R*sin(\ang)});
				\coordinate (-Y) at (0,{-2.0*\R*sin(\ang)});
				\node (R') at (R) {};
				\node (RR') at (RR) {};
				\node (-R') at (-R) {};
				\draw[->,line width=0.9] (-0.3*\xmax,0) -- (\xmax+0.05,0) node[right] {Re};
				\draw[->, line width=0.9] (0,-0.3*\ymax) -- (0,0.4*\ymax+0.05) node[left] {Im};
				\draw[line width = 1.2] (O) -- (R');
				\draw[name path=B,dotted, line width = 1.2] (O) -- (RR)  node[pos=0.55,above left=-2] {$\partial G_{\theta}$};
				\draw[line width = 1.2] (O) -- (-R') ;
				\draw[name path=A, dotted, line width = 1.2] (O) -- (-RR) node[pos=0.55,below left=-2] {$\partial G_{\theta}$};
				\tikzfillbetween[of=A and B]{blue, opacity=0.1};
				
				\draw pic[->,"$~\theta$",black,draw=black,angle radius=23,angle eccentricity=1.24]
				{angle = X--O--R};
				\draw pic[<-,"$-\theta$"{right=-2},black,draw=black,angle radius=20,angle eccentricity=1]
				{angle = -R--O--X};
				\tick{lambda1}{90} node[scale=1,above=3] {$\lambda_1$};
				\tick{lambda2}{90} node[scale=1,above=3] {$\lambda_2$};
				\tick{lambda3}{90} node[scale=1,above=3] {$\lambda_3$};
			\end{tikzpicture}
			\label{fig:Gtheta} }
	\end{minipage}%
	\begin{minipage}{0.5\textwidth}
		\centering
		\subfigure[Illustration of the contour $\Gamma = \Gamma_+ \cup \Gamma_- \cup \Gamma_0$ used in the proof of \Cref{prop:det}.]{
			\begin{tikzpicture}
				\def\xmax{4.0}
				\def\ymax{2.6}
				\def\delt{0.5}
				\def\R{2.5}
				\def\RR{4.0}
				\def\ang{18}
				\coordinate (O) at (0,0);
				\coordinate (del) at (\ang:\delt);
				\coordinate (-del) at (-\ang:\delt);
				\coordinate (R) at (\ang:\R);
				\coordinate (RR) at (\ang:\RR);
				\coordinate (-R) at (-\ang:\R);
				\coordinate (-RR) at (-\ang:\RR);
				\coordinate (X) at ({2.0*\R*cos(\ang)},0);
				\coordinate (lambda1) at ({0.7*\R*cos(\ang)},0);
				\coordinate (lambda2) at ({1.0*\R*cos(\ang)},0);
				\coordinate (lambda3) at ({1.5*\R*cos(\ang)},0);
				\coordinate (Y) at (0,{2.0*\R*sin(\ang)});
				\coordinate (-Y) at (0,{-2.0*\R*sin(\ang)});
				\node (R') at (R) {};
				\node (RR') at (RR) {};
				\node (-R') at (-R) {};
				\draw[->,line width=0.9] (-0.3*\xmax,0) -- (\xmax+0.05,0) node[right] {Re};
				\draw[->, line width=0.9] (0,-0.3*\ymax) -- (0,0.4*\ymax+0.05) node[left] {Im};
				\draw[line width = 1.2] (del) -- (R');
				\draw[name path=B,dotted, line width = 1.2] (O) -- (RR)  node[pos=0.55,above left=-2] {$\Gamma_+$};
				\draw[line width = 1.2] (-del) -- (-R') ;
				\draw[name path=A, dotted, line width = 1.2] (O) -- (-RR) node[pos=0.55,below left=-2] {$\Gamma_-$};
				
				\draw[name path=C,->,line width = 1.2](\ang:\delt) arc[radius=\delt, start angle=\ang, end angle=-\ang] node[pos=0.55,above = 7] {$\Gamma_0$};
				\draw[dotted,line width = 1.2] (0,0) circle (\delt);
				\draw[->,line width = 1.2] (0,0) -- (-120:\delt) node[scale = 0.75, left] {$\delta_0 $};
				\tikzfillbetween[of=A and B]{blue, opacity=0.1};
				
				\draw pic[->,"$~\theta$",black,draw=black,angle radius=23,angle eccentricity=1.24]
				{angle = X--O--R};
				\draw pic[<-,"$-\theta$"{right=2},black,draw=black,angle radius=20,angle eccentricity=1]
				{angle = -R--O--X};
				\tick{lambda1}{90} node[scale=1,above=3] {$\lambda_1$};
				\tick{lambda2}{90} node[scale=1,above=3] {$\lambda_2$};
				\tick{lambda3}{90} node[scale=1,above=3] {$\lambda_3$};
			\end{tikzpicture}
			\label{fig:nicecontour}}
	\end{minipage}%
	\caption{Contours used to define the Cauchy--Stieltjes integral representation of operators.} 
	\label{fig:contours}
\end{figure}	

\begin{proof}[Proof of \Cref{prop:det}]
	Let  $h\in(0,h_0)$, $f\in L^2(\cM)$ and let $p\in [0,1]$ such that  $\Vert \cL^pf\Vert_{L^2(\cM)}<\infty$. We  can therefore introduce $u=\cL^pf\in L^2(\cM)$. First, note using the triangle inequality, we have
	\begin{align*}
		\Vert\gamma(\cL_h)P_h f-\gamma(\cL)f\Vert_{L^2(\cM)} &\le \Vert\gamma(\cL_h)P_h f-P_h\gamma(\cL)f\Vert_{L^2(\cM)}\\
		&\quad +\Vert P_h\gamma(\cL)f-\gamma(\cL)f\Vert_{L^2(\cM)} = S_1+S_2
	\end{align*}
	where we take $S_1= \Vert\gamma(\cL_h)P_h f-P_h\gamma(\cL)f\Vert_{L^2(\cM)}$ and $S_2=\Vert P_h\gamma(\cL)f-\gamma(\cL)f\Vert_{L^2(\cM)}$. We now bound these two terms. 
	
	For the term $S_2$, let us introduce the function $\widehat{\gamma}$ defined as $\widehat{\gamma}(\lambda)=\gamma(\lambda)\lambda^{-p}$. In particular, note that since $\gamma$ is an $\alpha$-\psd, $\widehat{\gamma}$ is an $(\alpha+p)$-\psd, and we have, by definition of $u$, 
	\begin{equation*}
		S_2 = \Vert (I-P_h)\gamma(\cL)f\Vert_{L^2(\cM)}=\Vert (I-P_h)\widehat{\gamma}(\cL)u\Vert_{L^2(\cM)}
	\end{equation*}
	We then use the Bramble--Hilbert lemma (cf. \Cref{lem:bh}) with $t=2\min\lbrace\alpha+p;1\rbrace \in (0,2]$, while noting that, following \Cref{rem:regf}, $\Vert \cL^{t/2} (\widehat{\gamma}(\cL)u)\Vert_{L^2(\cM)}\lesssim \Vert u\Vert_{L^2(\cM)}<\infty$ (since $t/2 \le \alpha+p$):
	\begin{equation*}
		S_2 \lesssim h^t \Vert \cL^{t/2} (\widehat{\gamma}(\cL)u)\Vert_{L^2(\cM)} \lesssim h^t \Vert u\Vert_{L^2(\cM)} = h^{2\min\lbrace\alpha+p;1\rbrace} \Vert \cL^pf\Vert_{L^2(\cM)}.
	\end{equation*}

	Let us now bound the term $S_1$. 
	For any $z\in\Gamma$, set $\cF_h(z) = (z-\cL_h)^{-1}P_hf -P_h(z-\cL)^{-1}f$, which yields, using the integral representations of $\gamma(\cL)$ and $\gamma(\cL_h)$,
	\begin{align*}
		\gamma(\cL_h)P_h f-P_h\gamma(\cL)f &= \frac{1}{2\pi i} \int_{\Gamma} \gamma(z) \cF_h(z)\dd z.
	\end{align*}
	The definition of $\Gamma$ and its parametrization allow to decompose the integral as
	\begin{align*}
		\gamma(\cL_h)P_h f-P_h\gamma(\cL)f &= \frac{1}{2\pi i} \int_{-\infty}^{-\delta_0} \gamma(g_+(t)) \cF_h(g_+(t))g'_+(t) \dd t \\
		&\quad + \frac{1}{2\pi i} \int_{-\theta}^{\theta} \gamma(g_0(t)) \cF_h(g_0(t))g_0'(t) \dd t
		\\
		&\quad+ \frac{1}{2\pi i} \int_{\delta_0}^{+\infty} \gamma(g_-(t)) \cF_h(g_-(t))g'_-(t)\dd t,
	\end{align*}
	where we recall that $g_+'(t) = -e^{i\theta}$, $g_-'(t)=e^{-i\theta}$ and $g_0'(t) = -i \delta_0 e^{-it}$. Making a change of variable $t \mapsto -t$ in the first integral, and taking norms on both sides of the equality above then gives by  the triangle inequality
	\begin{align*}
		& S_1=\|\gamma(\cL_h)P_h f-P_h\gamma(\cL)f \|_{L^2(\cM)}\\	 
		& \quad \leq  \frac{1}{2\pi } \int_{\delta_0}^\infty |\gamma(g_+(-t))| \left\|\cF_h(g_+(-t))\right\|_{L^2(\cM)} \dd t
		+ \frac{\delta_0}{2\pi} \int_{-\theta}^{\theta} |\gamma(g_0(t))| \|\cF_h(g_0(t))\|_{L^2(\cM)} \dd t\\
		& \qquad + \frac{1}{2\pi}\int_{\delta_0}^\infty |\gamma(g_-(t))| \left\|\cF_h(g_-(t))\right\|_{L^2(\cM)} \dd t\\
		& \quad = I_+ + I_0 + I_-.
	\end{align*}
	
	Let us start by bounding $I_0$. 
	We apply \Cref{eq:fundamental_ineq} and \Cref{eq:lp} (with $\beta=1$) to obtain
	\begin{align*}
		I_0 
		&\lesssim   \int_{-\theta}^{\theta} |g_0(t)|^{-\alpha}~h^2  |g_0(t)|^{-p} \Vert \cL^pf\Vert_{L^2(\cM)} \dd t 
		= h^2 \|\cL^pf\|_{L^2(\cM)}\int_{-\theta}^{\theta} \delta_0^{-(\alpha+p)} \dd t \\
		&\lesssim h^2 \|\cL^pf\|_{L^2(\cM)}. 
	\end{align*}
	
	To bound $I_+$, we distinguish between the three cases $\alpha +p>1$, $\alpha+p <1$ and $\alpha+p =1$. 
	When $\alpha+p > 1$, we apply \Cref{eq:fundamental_ineq} and \Cref{eq:lp} (with $\beta=1$) to get
	\begin{align*}
		I_+	
		&\lesssim h^2 \|\cL^pf\|_{L^2(\cM)}  \int_{\delta_0}^{\infty}|g_+(-t)|^{-(\alpha+p)} \dd t 
		= h^2  \|\cL^pf\|_{L^2(\cM)} \int_{\delta_0}^{+\infty} |t|^{-(\alpha+p)} \dd t\\ 
		&= h^2 \|\cL^pf\|_{L^2(\cM)} \frac{\delta_0^{-(\alpha+p-1)}}{\alpha+p-1}.
	\end{align*}
	Hence, we can conclude that if $\alpha + p>1$, we have $I_+ \lesssim  h^{2} \|\cL^pf\|_{L^2(\cM)}$.
	
	For $\alpha+p<1$, recall that $0<h <h_0 < 1$ for some fixed $h_0$ such that $h_0^{-2} > \delta_0$. We then introduce $\tilde{\alpha}=(\alpha+p)/\alpha$ and  split $I_+=I_{+1}+I_{+2}$ with
	\begin{align*}
		&I_{+1}	 =  \int_{\delta_0}^{h^{-2\tilde{\alpha}}}   |\gamma(g_+(-t))|\,  \left\|\cF_h(g_+(-t))\right\|_{L^2(\cM)} \dd t, \\
		&I_{+2} =   \int_{h^{-2\tilde{\alpha}}}^\infty   |\gamma(g_+(-t))|\,  \left\|\cF_h(g_+(-t))\right\|_{L^2(\cM)} \dd t.
	\end{align*}
	Let then $e_h=(1-(\alpha+p))/\vert\log h\vert \in (0,1-(\alpha+p))$ and $\beta = \alpha + p + e_h \in (0,1)$. Note in particular that $p\in[0,(1+\beta)/2]$ since $(1+\beta)/2-p=(1-p+\alpha+e_h)/2>0$. 
	We can therefore use \Cref{eq:fundamental_ineq} and \Cref{eq:lp} to bound $I_{+1}$ by
	\begin{align*}
		I_{+1}	
		&\lesssim h^{2\beta} \|\cL^pf\|_{L^2(\cM)} \int_{\delta_0}^{h^{-2\tilde{\alpha}}} |g_+(-t)|^{-(\alpha+1+p-\beta)} \dd t 
		=h^{2\beta} \|\cL^pf\|_{L^2(\cM)} \int_{\delta_0}^{h^{-2\tilde{\alpha}}} |t|^{-(1-e_h)} \dd t \\
		&\lesssim h^{2\beta} \|\cL^pf\|_{L^2(\cM)} \int_{0}^{h^{-2\tilde{\alpha}}} |t|^{-(1-e_h)} \dd t 
		= h^{2\beta} \|\cL^pf\|_{L^2(\cM)} \frac{h^{-2\tilde{\alpha}e_h}}{e_h} \\
		&= h^{2(\alpha+p)} \|\cL^pf\|_{L^2(\cM)} \frac{h^{-2(\tilde{\alpha}-1)e_h}}{e_h}
		\lesssim  h^{2(\alpha+p)} \vert \log h\vert \|\cL^pf\|_{L^2(\cM)},
	\end{align*}
	since $h^{-2\tilde{\alpha}e_h}=\exp(2\tilde{\alpha}(1-(\alpha+p)))\lesssim 1$.
	We proceed in the same manner to bound $I_{+2}$, using this time \Cref{eq:lp} (with $\beta =p$):  
	\begin{align*}
		I_{+2} 
		&\lesssim h^{2p} \|\cL^pf\|_{L^2(\cM)}  \int_{h^{-2\tilde{\alpha}}}^\infty   |g_+(-t)|^{-(\alpha+1)} \dd t 
		= h^{2p}\|\cL^pf\|_{L^2(\cM)}  \int_{h^{-2\tilde{\alpha}}}^\infty | t|^{-(\alpha+1)} \dd t \\
		&\lesssim h^p h^{2\tilde{\alpha}\alpha} \|\cL^pf\|_{L^2(\cM)}
		\lesssim h^{2(\alpha+p)} \|\cL^pf\|_{L^2(\cM)}.
	\end{align*}
	Hence, we  conclude that if $\alpha + p<1$, we have $I_+ \lesssim  \vert\log h\vert~h^{2(\alpha+p)} \|\cL^pf\|_{L^2(\cM)}$.
	
	Finally, for $\alpha +p=1$ we repeat the same approach as for the case $\alpha+p<1$. On the one hand we use  \Cref{eq:lp} (with $\beta =1$) to get for the term $I_{+1}$ the bound
	\begin{align*}
		I_{+1}	
		&\lesssim h^2 \|\cL^pf\|_{L^2(\cM)} \int_{\delta_0}^{h^{-2\tilde{\alpha}}} |t|^{-1} \dd t 
		= h^2\|\cL^p f\|_{L^2(\cM)} (\log(h^{-2\tilde{\alpha}})-\log\delta_0 )\\
		& \lesssim h^2\|\cL^p f\|_{L^2(\cM)} \vert\log h\vert,
	\end{align*}
	where we used the fact that $h$ is upper-bounded by a fixed constant $h_0$ satisfying $h_0^{-2} > \delta_0$ to derive the last inequality. On the other hand, we once again use \Cref{eq:lp} with $\beta =p$  to bound the term $I_{+2}$ 
	\begin{align*}
		I_{+2} 
		&\lesssim
		h^{2p}\|\cL^pf\|_{L^2(\cM)}  \int_{h^{-2\tilde{\alpha}}}^\infty | t|^{-(\alpha+1)} \dd t
		\lesssim h^p h^{2\tilde{\alpha}\alpha} \|\cL^pf\|_{L^2(\cM)}
		\\
		&\lesssim h^{2(\alpha+p)} \|\cL^pf\|_{L^2(\cM)}=h^{2} \|\cL^pf\|_{L^2(\cM)}.
	\end{align*}
	Hence, we have $I_{+} \lesssim h^2 |\log h|~\|\cL^pf\|_{L^2(\cM)}$ when $\alpha+p=1$. 
	
	Putting together the three cases, we can therefore conclude that
	$$I_+\lesssim  C_{\alpha+p}(h) h^{2\min\lbrace\alpha+p;1\rbrace} \|\cL^pf\|_{L^2(\cM)},$$
	where $ C_{\alpha+p}(h)=\vert\log h\vert$ if $\alpha+p\le 1$, and  $ C_{\alpha+p}(h)>1$ if $\alpha+p\neq 1$.
	Finally, note that by symmetry, $I_-$ satisfies the same bounds as $I_+$, meaning that we also get
	$I_-\lesssim  C_{\alpha+p}(h) h^{2\min\lbrace\alpha+p;1\rbrace} \|\cL^pf\|_{L^2(\cM)}$.
	
	In conclusion, putting together the estimates for $I_0$, $I_+$ and $I_-$ we get 
	$$S_1 \lesssim  C_{\alpha+p}(h) h^{2\min\lbrace\alpha+p;1\rbrace} \|\cL^pf\|_{L^2(\cM)},$$
	and putting together the estimates pf $S_1$ and $S_2$ yields
	$$\Vert\gamma(\cL_h)P_h f-\gamma(\cL)f\Vert_{L^2(\cM)} \lesssim  C_{\alpha+p}(h) h^{2\min\lbrace\alpha+p;1\rbrace} \|\cL^pf\|_{L^2(\cM)},$$
	which concludes the proof.
\end{proof}

A result similar to \Cref{prop:det} can be derived to quantify the error between functions of the operators $\cL_h$ and $\sL_h$. To do so, we note that since $\mathsf{A}_{\cM_h}$ is continuous and coercive, the associated operator $\sL_h$ is also sectorial with some angle in $(0,\pi/2)$. Hence, without loss of generality, the angle $\theta \in (0,\pi/2)$ can be assumed to be large enough to ensure  that \Cref{eq:A5} also holds for $\sL_h$, i.e. that for any $z\in G_\theta$,
\begin{align}
	\|(z-\sL_h)^{-1}v_h\|_{L^2(\cM_h)} & \leq C_\theta |z|^{-1}\|v_h\|_{L^2(\cM_h)}, \quad v_h\in S_h,\label{eq:A5bis}
\end{align}
and that an integral representation similar to \Cref{eq:int_rep_operator} also holds for $\gamma(\sL_h)$, namely:
\begin{align*}
	\gamma(\sL_h)=\frac{1}{2\pi i } \int_{\Gamma} \gamma(z)(z-\sL_h)^{-1} \dd z.
\end{align*}

\begin{proposition}
	Let $\gamma$ be an $\alpha$-\psd with $\alpha>d/4$. Then, there exists some constant such that   for any $h\in(0,h_0)$, for any $\tilde{f}\in S_h^\ell$, 
	\begin{equation}
		\label{eq:err_pol_ld}
		\left\| \left(\gamma(\cL_h)\tilde{f}\right)^{-\ell} - \gamma(\sL_h)\sP_h(\sigma\tilde{f}^{-\ell})\right\Vert_{L^2(\cM_{h})}
		\lesssim  h^2 \, \| \cL_h^{-\min\lbrace \alpha+d/4;1\rbrace/2}\tilde{f}\Vert_{L^2(\cM)}.
	\end{equation}
	\label{lem:error_polyhedral}
\end{proposition}

This proposition can be seen as an extension of \cite[Lemma 4.4]{Bonito2024} relying on extensions of the estimates proven in \cite[Lemma A.1]{Bonito2024}. Its proof is similar and therefore postponed to \Cref{app:error_discr_op}.

\section{Stochastic theory: random fields on surfaces and convergence of SFEM approximation}
\label{sec:stoch}

In this section, we introduce random fields and white noise on surfaces, thus allowing us to make sense of \Cref{eq:fielddef} in \Cref{sec:intro}. 
Further, we give approximation methods based on SFEM and prove strong error bounds.  
\subsection{Random fields on surfaces}
\label{sec:rf}

Let $(\Omega, \cS, \IP)$ be a complete probability space. 
We are interested in Gaussian random fields on $\cM$ defined as $L^2(\cM)$-valued random variables through expansions of the form
\begin{align}
	\label{eq:KL}
	\cZ = \sum_{i\in\N} Z_i e_i,
\end{align}
where $\lbrace e_i\rbrace_{i\in\N}$ denotes a real-valued orthonormal basis of $L^2(\cM)$ composed of eigenfunctions of the operator $\cL$ (cf. \Cref{prop:spectral}), and $\lbrace Z_i\rbrace_{i\in\N}$ is a sequence of real Gaussian random variables such that $\E[Z_i]=0$ for any $i\in\N$ and  $\sum_{i \in \N} \E[| Z_i|^2] < \infty$. As such, $\cZ$ can be seen as en element of the Hilbert space  $L^2(\Omega;L^2(\cM))$ of  $L^2(\cM)$-valued random variables, to which we associate the inner product $(\cdot,\cdot)_{L^2(\Omega;L^2(\cM))}$ (and  norm $\| \cdot\Vert_{L^2(\Omega;L^2(\cM))}$) defined by
\begin{align*}
	(\cZ,\cZ')_{L^2(\Omega;L^2(\cM))} = \E\left[(\cZ,\cZ')_{L^2(\cM)}\right], \quad \cZ,\cZ'\in L^2(\Omega;L^2(\cM)).
\end{align*}
Finally, in analogy to \Cref{eq:KL}, we formally define the Gaussian white noise on $\cM$ by the expansion
\begin{align}
	\label{eq:whitenoise}
	\cW = \sum_{i \in \N} W_i e_i,
\end{align}
where $\lbrace W_i\rbrace_{i\in\N}$ is a sequence of independent real standard Gaussian random variables. We observe that even though this expansion does not converge in $L^2(\Omega;L^2(\cM))$, it does however converge in $L^2(\cM; H^s(\cM))$ for $s < -d/2$. 
Moreover, we have that for any $\phi\in L^2(\cM)$, the expansion $( \phi,\cW)_{L^2(\cM)}=  \sum_{i \in \N} W_i (\phi,e_i )_{L^2(\cM)}$ converges in $L^2(\Omega)$. 
Further, $(\cW, \phi)_{L^2(\cM)}$ defines a complex Gaussian variable with mean $0$, and for any $\phi'\in L^2(\cM)$,  $\Cov((\phi,\cW)_{L^2(\cM)}, (\phi',\cW)_{L^2(\cM)})=\E[(\phi,\cW)_{L^2(\cM)}\cj{(\phi',\cW)_{L^2(\cM)}}]=(\phi, \phi')_{L^2(\cM)}$. 
As such, the Gaussian white noise~\eqref{eq:whitenoise} can be interpreted as a \emph{generalized Gaussian random field} over~$L^2(\cM)$.

Circling back to the class of random fields defined in \Cref{sec:intro}, we can now make sense of \Cref{eq:fielddef} through \Cref{eq:def_func_of_op} and \Cref{eq:whitenoise}, thus yielding the definition
\begin{equation}\label{eq:defZ}
	\mathcal{Z}=\gamma(\cL) \cW = \sum_{i\in\N}\gamma(\lambda_{i})W_ie_i, 
\end{equation}
which results in $\mathcal{Z}\in L^2(\Omega;H^{s}(\cM))$ for any $s\ge0$ such that $4\alpha -d > 2s$. Note in particular that all summands in \Cref{eq:defZ} are real-valued functions, and that therefore $\cZ$ is real-valued.
In \Cref{fig:maternex} we illustrate the influence of the parameter choices on the resulting field on~$\IS^2$ for generalized \emph{non-stationary Whittle--Mat\'ern} fields on~$\IS^2$
\begin{align*}
	\cZ = (\cL)^{-\alpha}	\cW,
\end{align*}
where $\alpha >d/2$. 
By selecting $D_{ij}(x) = \delta_{ij}$ and $V(x) = \kappa^2$ with $\kappa>0$ one recovers the classical, stationary Whittle--Mat\'ern fields studied in for instance \cite{Cox2020,BKK20,Bonito2024,Lindgren2011,JKL22,W63}.
In Figures~\ref{fig:maternlowreg} and~\ref{fig:maternhighreg} we show this case with $\kappa^2 = 10$ for a rougher field with $\alpha = 0.55$ and a smoother one with $\alpha = 1.5$, respectively.

Two non-stationary fields are shown in \Cref{fig:maternvarpot} and \Cref{fig:maternvarmet}, obtained by varying the coefficient functions $D_{ij}$ and $V$ and setting $\alpha = 0.75$. 
In \Cref{fig:maternvarpot}, we keep $D_{ij} = \delta_{ij}$ but use
\begin{align*}
	V(x) =
	\begin{cases}
		10^5 & \text{for } x_2^6+x_1^3-x_3^2 \in  (0.1,0.5),	 \\
		10    & \text{else}
	\end{cases}
\end{align*}
resulting in the observed localized behavior, where the field is essentially turned off in the region with large~$V$. 
More specifically, $V$ describes the local correlation length, where a large $V(x)$ corresponds to a small correlation length around~$x$. 

Finally, setting $V = 10$ constant again, we show the influence of varying parameters $\cD$ in \Cref{fig:maternvarmet}.
To derive suitable coefficients, we select a smooth function~$f$ and compute its gradient~$\nabla_{\IS^2} f$ as well as its skew-gradient~$X_f$ given by $x \times \nabla f(x)$ at each point $x \in \IS^2$. 
We set for fixed $\rho_1, \rho_2 > 0$, $\cD(x)u=\rho_1(\nabla f(x)\cdot u)\nabla f(x)+\rho_2(X_f(x)\cdot u)X_f(x)$ for any $u\in T_x\IS^2$. 
Here, the inner product refers to the Riemannian inner product associated with the standard round metric on $\IS^2$.  
Since $\nabla_{\IS^2} f(x) $ and $X_f(x)$ both are in $T_x \IS^2$, $\cD(x)$ is a linear mapping from $T_x \IS^2$ into itself. 
Further, as $\nabla_{\IS^2} f$ is perpendicular to $X_f(x)$, by selecting $\rho_1$ and $\rho_2$, we obtain a field that is elongated either orthogonally to the level sets of~$f$
(large $\rho_1$, small $\rho_2$) or tangentially to the level sets (small $\rho_1$, and large $\rho_2$). 
To generate \Cref{fig:maternvarmet}, we selected  $f(x)= x_2$, i.e., the function returning the second coordinate in Cartesian coordinates, $\rho_1 = 1$ and $\rho_2 = 25$.
\begin{figure}[tb]
	\subfigure[Stationary Whittle--Mat\'ern field with low decay parameter.]{
		\includegraphics[width=0.21\textwidth, trim=150mm 5mm 150mm 5mm, clip]{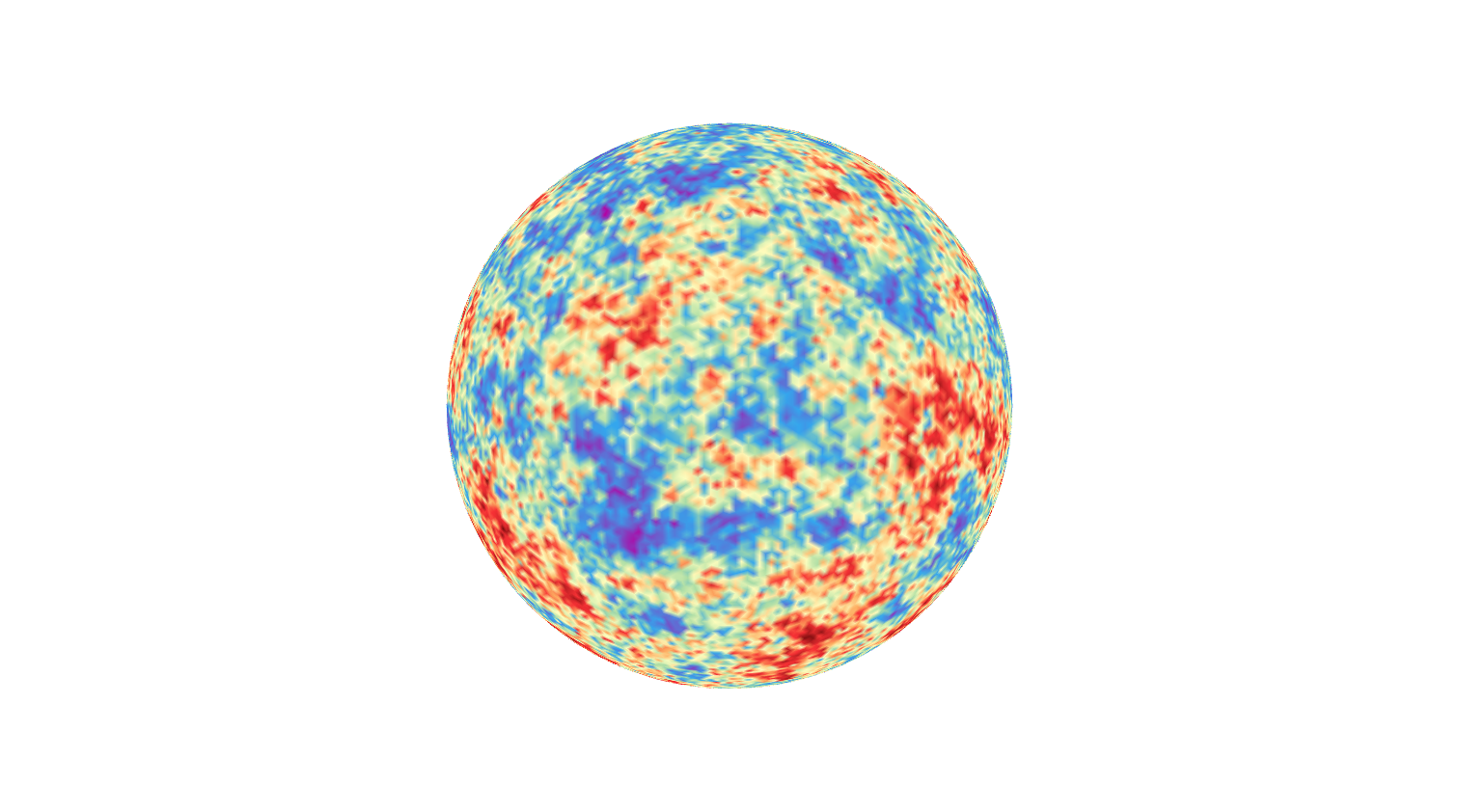}
		\label{fig:maternlowreg}
	}\hspace*{0.5em}
	\subfigure[Stationary Whittle--Mat\'ern field with high decay parameter.]{
		\includegraphics[width=0.21\textwidth, trim=150mm 5mm 150mm 5mm, clip]{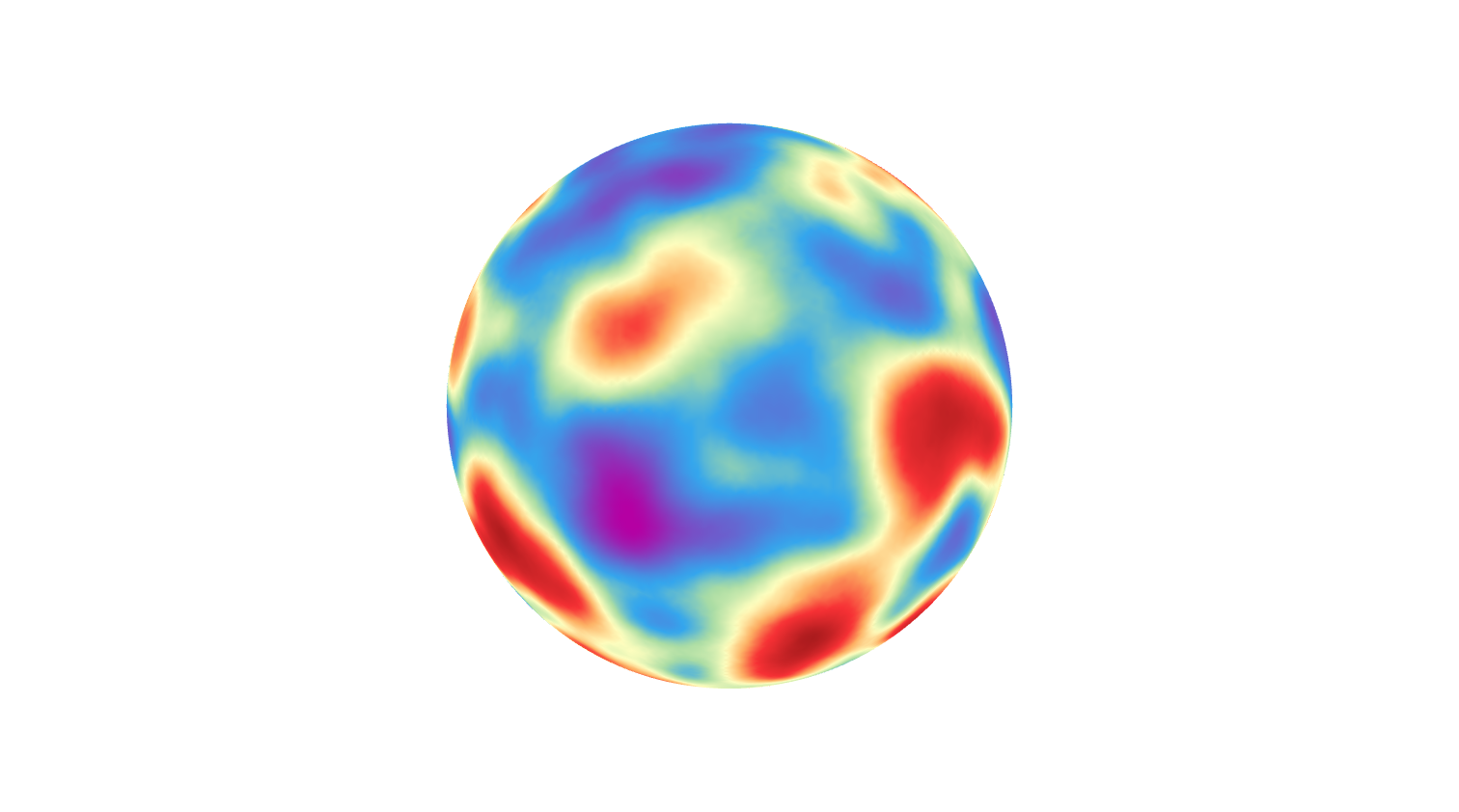}
		\label{fig:maternhighreg}
	}\hspace*{0.5em}
	\subfigure[Non-stationary Whittle--Mat\'ern field by varying potential.]{
		\includegraphics[width=0.21\textwidth, trim=150mm 5mm 150mm 5mm, clip]{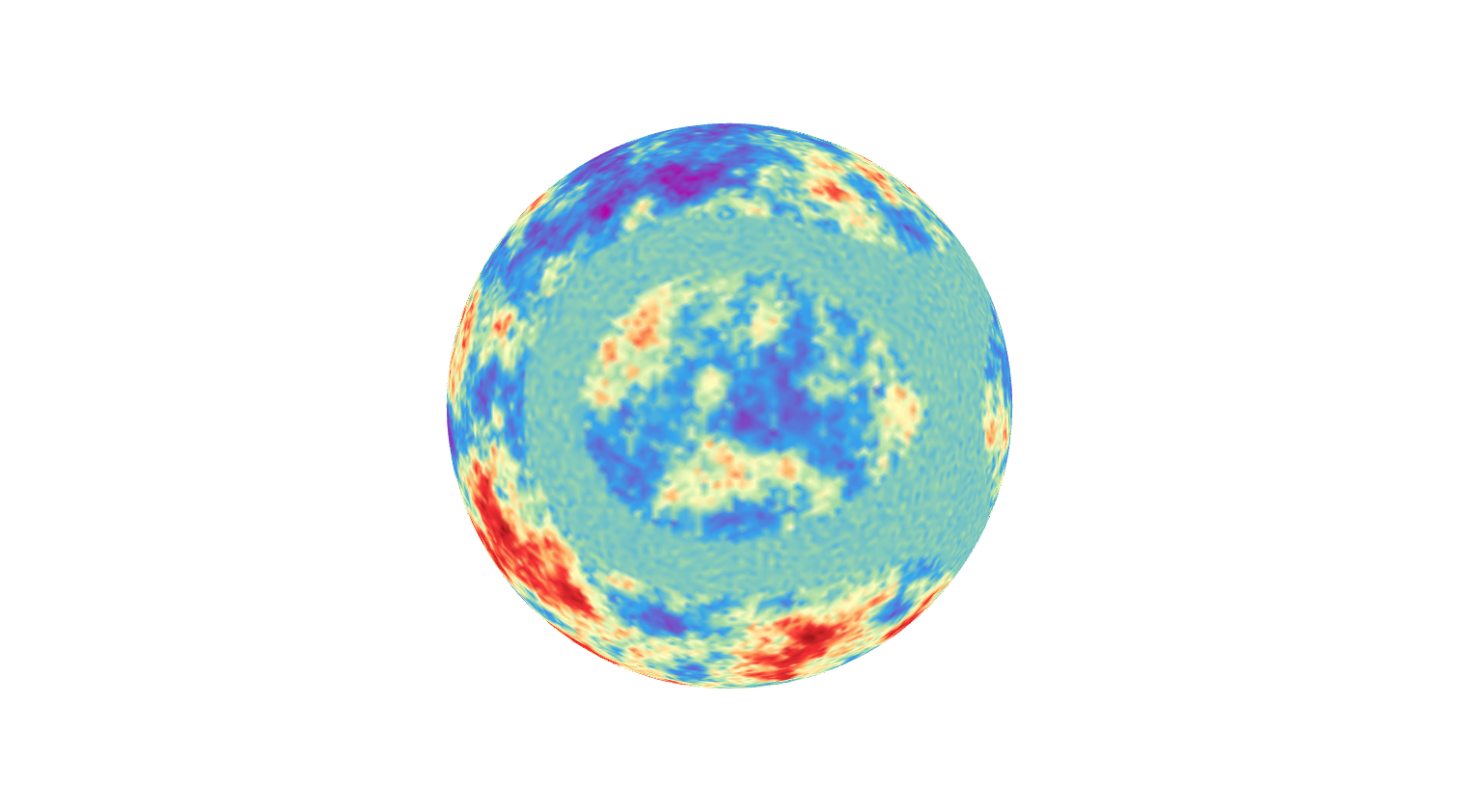}
		\label{fig:maternvarpot}
	}\hspace*{0.5em}
	\subfigure[Non-stationary Whittle--Mat\'ern field by varying diffusion matrix.]{
		\includegraphics[width=0.21\textwidth, trim=150mm 5mm 150mm 5mm, clip]{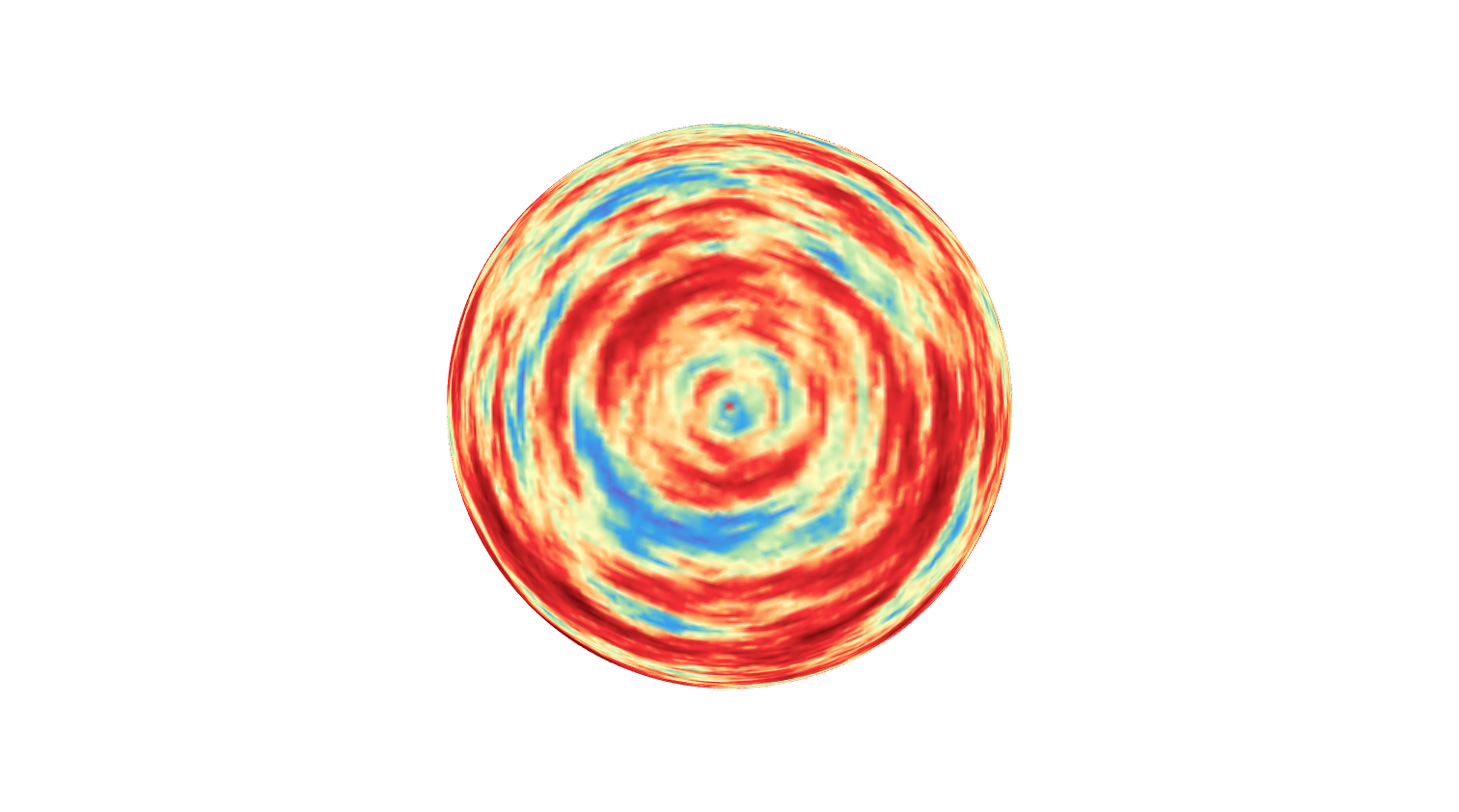}\label{fig:maternvarmet}}
	\caption{Some examples to highlight the influence of the different model parameters. \label{fig:maternex}
	}
\end{figure}
Finally, we remark that other types of \psds can be considered. For instance, to generate \Cref{fig:third_example}, we used an \psd of the form $\gamma(\lambda) = \sin(\lambda) \lambda^{-0.6}$.

\subsection{Approximation of random fields with surface finite elements}
\label{sec:method}

Let $\gamma$ be an $\alpha$-\psd with $\alpha>d/4$.
Following the approach presented by \cite{Lang2023}, the field $\cZ$ is approximated by an expansion similar to that of \Cref{eq:defZ}, but involving only quantities defined on the polyhedral surface $\cM_{h}$. 
More precisely, we define the SFEM--Galerkin approximation $\sZ_h$ of the field $\cZ$ by the relation 
\begin{equation}\label{eq:z_galer}
	\sZ_h=\sum_{i=1}^{N_h} \gamma(\Lambda_{i}^h)W_i^hE_{i}^h,
\end{equation}
where $\lbrace W_i^h \rbrace_{1\le i\le N_h}$ is a sequence of independent standard Gaussian random variables, whose precise definition is clarified later in this section, and $\lbrace (\Lambda_{i}^h, E_i^h)\rbrace_{1\le i\le N_h}$ are the eigenpairs of $\sL_h$ introduced in \Cref{sec:sfem}. 
Then, as proven in \cite[Theorem 3.4]{Lang2023}, $\sZ_h$ can be decomposed in the nodal basis $\lbrace \psi_i\rbrace_{1\le i\le N_h}$ of $S_h$ as 
\begin{equation*}
	\sZ_h =\sum_{i=1}^{N_h} Z_i\psi_i,
\end{equation*}
where the weights $(Z_1, \dots, Z_{N_h})$ form a centered Gaussian vector which covariance matrix $\bm\Sigma_{Z}$  can be expressed using the matrices $\bm C$ and $\bm S$ introduced in \Cref{eq:def_CG,eq:def_S} as follows:
\begin{equation}
	\bm\Sigma_{Z}=\big(\sqrt{\bm C}\big)^{-T}\,\gamma(\bm S)^2\,\big(\sqrt{\bm C}\big)^{-1}, 
	\label{eq:cov_weights}
\end{equation}
and, following \Cref{eq:ortho_S}, the function of matrix $\gamma(\bm S)$ is defined as 
\begin{equation}
	\gamma(\bm S)^2
	=
	\bm V
	\Diag\left(\gamma\big(\Lambda_{1}^h\big)^2 , \dots, \gamma\big(\Lambda_{1}^{N_h }\big)^2\right)
	\bm V^T.
\end{equation}

Note that sampling the weights $(Z_1, \dots, Z_{N_h})$, and therefore the field $\sZ_h$, using 
directly the expression of their covariance matrix requires  in practice to fully diagonalize  $\bm S$ (since \Cref{eq:cov_weights} involves a function of a matrix). Such an operation would result in a prohibitive computational cost (of order $\mathcal{O}(N_h^3)$ operations). 
To avoid this cost, we use the Chebyshev trick proposed by \cite[Section 4]{Lang2023}, and approximate $\sZ_h$ by the field $\sZ_{h,M}$ defined by 
\begin{equation}\label{eq:z_cheb}
	\sZ_{h,M}=\sum_{i=1}^{N_h} P_{\gamma,M}(\Lambda_{i}^h)W_i^h E_i^h,
\end{equation}
where $P_{\gamma, M}$ is a Chebyshev polynomial approximation of degree $M\in\N$ of $\gamma$ over an interval $[\lambda_{\min},\lambda_{\max}]$ containing all the eigenvalues $\lbrace\Lambda_{i}^h\rbrace_{1\le i\le N_h}$ of $\sL_h$ (which we recall, coincide with the eigenvalues of $\bm S$). 
Such an interval can be obtained as follows. On the one hand, one can take $\lambda_{\min}=V_-$. 
On the other hand, following \cite{Lang2023}, a candidate for $\lambda_{\max}$ is obtained by applying the Gershgorin circle theorem to $\bm S$. 

The field $\sZ_{h,M}$, called Galerkin--Chebyshev approximation of $\cZ$,  is in essence defined by just replacing the \psd $\gamma$  by the polynomial~ $P_{\gamma, M}$ in the definition of SFEM--Galerkin approximation $\sZ_{h}$. The expansion of $\sZ_{h,M}$ into the nodal basis,
\begin{equation}\label{eq:decomp_gc}
	\sZ_{h,M}=\sum_{i=1}^{N_h} Z_i^{(M)}\psi_i,
\end{equation}
is such that the weights $(Z_1^{(M)}, \dots, Z_{N_h}^{(M)})$ now form a centered Gaussian vector with covariance matrix $\bm\Sigma_{Z^{(M)}}$ given by
\begin{equation}
	\bm\Sigma_{Z^{(M)}}=\big(\sqrt{\bm C}\big)^{-T}\,P_{\gamma, M}^2(\bm S)\,\big(\sqrt{\bm C}\big)^{-1}.
\end{equation}
Hence, the matrix function in \Cref{eq:cov_weights} is now replaced by a matrix polynomial $P_{\gamma, M}^2(\bm S)$.
This eliminates the eigendecomposition need associated with matrix functions and therefore speeds up the computations. Indeed, the weights $(Z_1^{(M)}, \dots, Z_{N_h}^{(M)})$ can be sampled through
\begin{equation}\label{eq:compute_gc}
	\begin{psmallmatrix}
		Z_1^{(M)} \\
		\vdots\\
		Z_{N_h}^{(M)}
	\end{psmallmatrix}
	=\big(\sqrt{\bm C}\big)^{-T}P_{\gamma, M}(\bm S)\bm W
\end{equation}
where $\bm W = (w_1, \dots, w_{N_h})^T$ is a vector of independent standard Gaussian random variables, and the matrix-vector product by $P_{\gamma, M}(\bm S)$ can be computed iteratively while just requiring products between $\bm S$ and vectors. In the next two subsections, we provide error estimates quantifying the error between our target random field $\cZ$ and its successive SFEM and polynomial approximations by $\sZ_h$ and~$\sZ_{h,M}$.

\subsection{Error analysis of the SFEM discretization}
\label{sec:convergence}

We start with analyzing the error between the random field $\cZ$ defined on $\cM$ and its SFEM approximation $\sZ_h$, as stated in the next theorem.

\begin{theorem}
	Let $\gamma$ be an $\alpha$-\psd with $\alpha>d/4$.
	Then, for any $h\in(0,h_0)$,
	the strong approximation error of the random field $\mathcal{Z}$ by its discretization $\sZ_h^\ell$  satisfies the bound 
	\begin{equation}
		\label{eq:err_str}
		\|\cZ- \sZ_h^\ell\|_{L^2(\Omega;L^2(\cM))}   \leq  C_{\alpha}(h)h^{2\min\{\alpha-d/4;1\}},
	\end{equation}
	where  $C_\alpha(h)=\vert \log h\vert$ if $d/4<\alpha \le 1$, $C_\alpha(h)=\vert \log h\vert^{3/2}$ if $1<\alpha < 1+d/2$, and $C_\alpha(h)=\vert \log h\vert^{1/2}$ if $\alpha \ge 1+d/2$.
	\label{th:err_fem}
\end{theorem}

To prove the strong error estimate, we rely on the deterministic error bounds proven in the previous section, and on several intermediate approximations of $\cZ$ defined on the spaces $S_h^\ell$ (i.e., on~$\cM$) and $S_h$ (i.e., on~$\cM_h$). These intermediate approximations require in turn to define approximations of the Gaussian white noise $\cW$ on the spaces $S_h^\ell$ and $S_h$.

We first define on $S_h^\ell$ (i.e., on $\cM$), the projected white noise $\widetilde{\cW}$ as
\begin{align}\label{def:wtilde}
	\widetilde{\cW} = \sum_{j=1}^{N_h} \xi_j e_j^h, 
\end{align}
where we recall that $\lbrace e_j^h\rbrace_{1\le i\le N_h}$ denotes an orthonormal basis of eigenfunctions of $\cL_h$, and for any $j\in\lbrace 1,\dots,N_h\rbrace$, we take $\xi_j =\big(e_j^h,\cW\big)_{L^2(\cM)}$. In particular, this last relation implies (by definition of the white noise $\cW$) that  $\xi_1,\dots,\xi_{N_h}$ are independent standard Gaussian random variables.

\begin{remark}\label{rem:tildeW}
	By injecting the representation~\eqref{eq:whitenoise} of $\cW$ in the definition of $\xi_j$, we get that $\widetilde{\cW}$ can itself be formally represented by $\widetilde{\cW} = \sum_{k=1}^\infty W_k P_h e_k=P_h\cW$. This explains why we refer to it as a projected white noise. In particular, note that $\tilde{\cW} \in L^2(\Omega; L^2(\cM))$ and $\Vert \tilde{\cW} \Vert_{L^2(\Omega;L^2(\cM))}^2 = \Vert \sum_{j=1}^{N_h} \xi_j e_j^h \Vert_{L^2(\Omega;L^2(\cM))}^2 =N_h < \infty$.
\end{remark}

Based on $\widetilde{\cW}$, we can then introduce a first approximation, on the space $S_h^\ell$, of the field $\cZ$. We denote this approximation by $\tilde \cZ_h$ and define it in analogy to \Cref{eq:defZ} as
\begin{align}\label{def:ztildeh}
	\tilde \cZ_h & = \gamma(\cL_h) \widetilde{\cW}= \sum_{j=1}^{N_h}\gamma(\lambda_{j}^h)\xi_j e_j^h.
\end{align}

Now, on the space $S_h$, we define \emph{two} white noise approximations $\sWb$ and $\sW$ which are based on the projected white noise $\widetilde{\cW}$:
\begin{align*}
	\sWb = \sP_h(\sigma \widetilde{\cW}^{-\ell}) \quad \text{and} \quad \sW = \sP_h(\sigma^{1/2} \widetilde{\cW}^{-\ell}),
\end{align*}
where $\sigma$ is the ratio of area measures introduced in \Cref{sec:sfem}. 
On the one hand, we associate to $\sWb$ an approximation $\sZb_h$ of the field $\cZ$ on $S_h$, which we define in analogy to \Cref{eq:defZh} as
\begin{align*}
	\sZb_h = \gamma(\sL_h)\sWb.
\end{align*}
On the other hand, by expanding $\sWb$ and $\sW$ in the orthonormal basis $\lbrace E_j^h\rbrace_{1\le i\le N_h}$ of eigenfunctions of $\sL_h$, we obtain alternative representations of these fields, and we can draw a link between $\sW$ and the SFEM--Galerkin approximation $\sZ_{h}$, as stated in the next lemma.

\begin{lemma}
	\label{lem:noise-lemma}
	The noises $\sWb$ and  $\sW$ can be written as 
	\begin{align*}
		\sWb 
		= \sum_{j=1}^{N_h} \alpha_j E_j^h
		\quad \text{and} \quad 
		\sW 
		= \sum_{j=1}^{N_h} \beta_j E_j^h,
	\end{align*}
	where $(\alpha_1,\ldots,\alpha_{N_h})$ and $(\beta_1,\ldots,\beta_{N_h})$  are multivariate normal with mean $0$ 
	and respective covariance matrices $\bm A=[( \sigma E_i^h,E_j^h)_{L^2(\cM_h)}]_{1\le i,j\le N_h}$ and $\bm B = \bm I$. 
	In particular, it holds 
	that the SFEM--Galerkin approximation $\sZ_h$ defined in \Cref{eq:z_galer} satisfies
	\begin{align}\label{eq:defZh}
		\sZ_h = \gamma(\sL_h)\sW,
	\end{align}
	where we take for any $j\in\lbrace 1,\dots,N_h\rbrace$, $W_j^h=\beta_j$.
\end{lemma}

\begin{proof}
	As $\sWb \in S_h$, we can expand it in the orthonormal basis  $\lbrace E_j^h\rbrace_{1\le i\le N_h}$ to get
	\begin{align*}
		\sWb = \sum_{j=1}^{N_h} \alpha_j E_j^h,
	\end{align*}
	where $\alpha_j = (\sWb,E_j^h)_{L^2(\cM_h)} = (\sigma \widetilde{\cW}^{-\ell},E_j^h)_{L^2(\cM_h)}$. 
	Then, by definition of $\sigma$, of $\widetilde{\cW}$ and since $(E_j^h)^\ell \in S_h^\ell$, we further get for any $1\le j\le N_h$,
	\begin{align*}
		\alpha_j &= ( \widetilde{\cW},(E_j^h)^\ell)_{L^2(\cM)} 
		=\sum_{k=1}^{N_h} (\cW,e_k^h)_{L^2(\cM)}  (e_k^h,(E_j^h)^\ell)_{L^2(\cM)}\\
		&=\bigg( \cW, \sum_{k=1}^{N_h} (e_k^h,(E_j^h)^\ell)_{L^2(\cM)}e_k^h\bigg)_{L^2(\cM)}  
		= ( \cW,(E_j^h)^\ell)_{L^2(\cM)} .
	\end{align*}
	Therefore, by definition of the white noise $\cW$, we can conclude that for any $1\le i,j\le N_h$,  $\alpha_j$ is normally distributed with mean $0$, and that
	\begin{align*}
		\E[\alpha_i\alpha_j] =\Cov(\alpha_i,\alpha_j)
		= \left((E_i^h)^\ell ,(E_j^h)^\ell\right)_{L^2(\cM)} = \left(\sigma E_i^h ,E_j^h \right)_{L^2(\cM_h)} = A_{ij}.
	\end{align*}
	Hence, $(\alpha_1,\ldots,\alpha_{N_h})$ is indeed multivariate normal (any linear combination of the $\alpha_j$ being Gaussian by definition of $\cW$) with mean $0$ and covariance matrix $\bm A$. 
	
	Similarly, since $\sW \in S_h$, we can write again
	\begin{align*}
		\sW = \sum_{j=1}^{N_h} \beta_j E_j^h,	
	\end{align*}
	where $\beta_j = (\sW,E_j^h)_{L^2(\cM_h)} = (\sigma^{1/2} \widetilde{\cW}^{-\ell},E_j^h)_{L^2(\cM_h)}=( \sigma\widetilde{\cW}^{-\ell}, \sigma^{-1/2}E_j^h)_{L^2(\cM_h)}$,
	and the same computations as before yield that $\beta_j = \big( \cW,(\sigma^{-1/2}E_j^h)^\ell\big)_{L^2(\cM)}$.
	Hence, we conclude this time that for any $1\le i,j\le N_h$,  $\beta_j$ is also normally distributed with mean $0$, and that
	\begin{align*}
		\E[\beta_i\beta_j]
		& = \left(\sigma (\sigma^{-1/2}E_i^h) , (\sigma^{-1/2}E_j^h) \right)_{L^2(\cM_h)} 
		=  \left(\sigma (\sigma^{-1/2}E_i^h) , (\sigma^{-1/2}E_j^h) \right)_{L^2(\cM_h)} =B_{ij},
	\end{align*}
	by orthonormality of $\lbrace E_j^h\rbrace_{1\le i\le N_h}$.
	In conclusion, $(\beta_1,\ldots,\beta_{N_h})$ is indeed multivariate normal  with mean $0$ and covariance matrix $\bm B=\bm I$. In particular, this means that $\beta_1,\dots,\beta_{N_h}$ are independent standard Gaussian variables.
	
	Finally, note that we can write
	\begin{align*}
		\gamma(\sL_h)\sP_h\sW
		&=\sum_{j=1}^{N_h}\gamma(\Lambda_j^h)\big(\sP_h\sW,E_j^h\big)_{L^2(\cM_{h})}E_j^h
		=\sum_{j=1}^{N_h}\gamma(\Lambda_j^h)\big(\sW,E_j^h\big)_{L^2(\cM_{h})}E_j^h \\
		&=\sum_{j=1}^{N_h}\gamma(\Lambda_j^h)\beta_jE_j^h
		= \sZ_h
	\end{align*}
	where we take $W_j^h=\beta_j$ in \Cref{eq:z_galer}.
\end{proof}

We now circle back to proving \Cref{th:err_fem}. Using the intermediate approximations $\tilde \cZ_h$ and $\sZb_h$ and the equivalence of the $L^2$ norms on $\cM$ and $\cM_{h}$, we can upper-bound the error between the field $\cZ$ and its SFEM--Galerkin approximation $\sZ_h$ by
\begin{align*}
	\|&\cZ-  \sZ_h^\ell\|_{L^2(\Omega;L^2(\cM))} \\
	&\quad\lesssim \|\cZ-  \tilde \cZ_h\|_{L^2(\Omega;L^2(\cM))}  + \| \tilde \cZ_h-  \sZb_h^\ell\|_{L^2(\Omega;L^2(\cM))} + \|  \sZb_h^\ell-  \sZ_h^\ell\|_{L^2(\Omega;L^2(\cM))}  \\
	&\quad\lesssim \|\cZ-  \tilde \cZ_h\|_{L^2(\Omega;L^2(\cM))}  
	{+} \big\|\big(\tilde \cZ_h\big)^{-\ell}{-}\sZb_h\big\|_{L^2(\Omega;L^2(\cM_h))} 
	{+}\big\|\sZb_h -\sZ_h\big\|_{L^2(\Omega;L^2(\cM_h) )},
\end{align*}
\revision{where we used the fact that the area ratio $\sigma$ in \Cref{eq:sigma} is uniformly bounded in $h\in(0,h_0)$ to obtain the third term in the sum (cf. \cite[Lemma 4.1]{Dziuk2013}).}
We derive error estimates for each one of the three terms obtained in the last inequality. We start with the term $\|\cZ-  \tilde \cZ_h\|_{L^2(\Omega;L^2(\cM))}$.

\begin{lemma}
	Let $h\in(0,h_0)$. It holds that 
	\begin{align*}
		\|\cZ- \tilde \cZ_h\|_{L^2(\Omega;L^2(\cM))} \lesssim
		C_\alpha(h) h^{2\min\{\alpha-d/4;1\}},
	\end{align*}
	where  $C_\alpha(h)=\vert \log h\vert$ if $d/4<\alpha \le 1$, $C_\alpha(h)=\vert \log h\vert^{3/2}$ if $1<\alpha < 1+d/2$, and $C_\alpha(h)=\vert \log h\vert^{1/2}$ if $\alpha \ge 1+d/2$.
	\label{lem:finest}
\end{lemma}

\begin{proof}

	Our aim is to bound the error between $\tilde \cZ_h$ and $\cZ$, where we remark in particular that $\cZ = \gamma(\cL)\cW$ and $\cZ_h = \gamma(\cL_h)P_h \cW$ (cf. \Cref{rem:tildeW}). Recall that we consider $h\in(0,h_0)$ (meaning in particular that $\vert \log h\vert> 1$). 
	We distinguish two cases: if $\alpha\le 1$, and if $\alpha>1$.
	
	\medspace
	
	\paragraph{\underline{Case $\alpha\le 1$}} 
	The proof follows the same approach as \cite[Theorem 5.2]{Bonito2024}. 
	We define $\varepsilon_h =(\alpha - d/4)/|\log h|$, $\eta = 2(d/4 + \varepsilon_h)$, and $\zeta = 2(\alpha-d/4 -\varepsilon_h) >0$, so that $\alpha = (\eta+\zeta)/2$. 
	Note in particular that $\varepsilon_h >0$ since $\alpha >d/4$ and that $\varepsilon_h<\alpha-d/4$. 
	Finally, we have that $\eta > d/2$ and $\zeta >0$.

	The triangle inequality yields
	\begin{align*}
		\| \cZ- \tilde \cZ_h\|_{L^2(\Omega;L^2(\cM))} 
		= \| \gamma(\cL)\cW - \gamma(\cL_h)P_h \cW\|_{L^2(\Omega;L^2(\cM))} &\leq S_1 + S_2,
	\end{align*}
	where we take 
	\begin{align*}
		S_1 &
		= \| \gamma(\cL)\cW - \gamma(\cL)P_h \cW\|_{L^2(\Omega;L^2(\cM))} ,\quad
		S_2 
		= \| \gamma(\cL)P_h\cW - \gamma(\cL_h)P_h \cW\|_{L^2(\Omega;L^2(\cM))} .
	\end{align*}
	
	To bound the term $S_1$, we use 
	\begin{align*}
		S_1^2 
		&= \| \gamma(\cL)(I - P_h )\cW\|_{L^2(\Omega;L^2(\cM))}^2
		=\| \sum_{i \in \N} \gamma(\lambda_{i})W_i (I - P_h )e_i\|_{L^2(\Omega;L^2(\cM))}^2 \\
		&=\sum_{i \in \N} \gamma(\lambda_{i})^2\Vert (I - P_h )e_i\Vert_{L^2(\cM)}^2,
	\end{align*}
	and since $\gamma$ is an $\alpha$-\psd, we get
	\begin{equation}\label{eq:e1}
		S_1^2 \lesssim
		\sum_{i \in \N}\lambda_{i}^{-2\alpha}\Vert (I - P_h )e_i\Vert_{L^2(\cM)}^2,
	\end{equation}
	Applying the Bramble--Hilbert lemma~\eqref{eq:bh} with $t=2((\alpha-d/4)-\varepsilon_h)\in (0,2)$ together with \Cref{eq:e1} gives
	\begin{equation*}
		S_1^2 \lesssim
		\sum_{i \in \N}h^{2t}\lambda_{i}^{-2\alpha+t}
		=h^{4(\alpha-d/4-\varepsilon_h)}\sum_{i \in \N}\lambda_{i}^{-2(d/4+\varepsilon_h)}
		\lesssim
		h^{4(\alpha-d/4)}\sum_{i \in \N}\lambda_{i}^{-2(d/4+\varepsilon_h)},
	\end{equation*}
	since $h^{-4\varepsilon_h} = e^{4(\alpha-d/4)}\lesssim 1$.
	We can then write, using \Cref{eq:gbounds_ev},
	\begin{equation*}
		\begin{aligned}
			S_1^2 
			&\lesssim h^{4(\alpha-d/4)}\sum_{i \in \N}i^{-(1+4\varepsilon_h/d)}
			\lesssim h^{4(\alpha-d/4)} \int_{1}^{\infty}t^{-(1+4\varepsilon_h/d)}\dd t \\
			&\lesssim  \frac{1}{4\varepsilon_h/d}h^{4(\alpha-d/4)}
			\lesssim \vert \log h\vert h^{4(\alpha-d/4)},
		\end{aligned}
	\end{equation*}
where we used the definition of $\varepsilon_{h}$ for the last inequality. So, we conclude $S_1 \lesssim \vert \log h\vert^{1/2} h^{2(\alpha-d/4)}$.
	
	For the term $S_2$, we note that 
	\begin{equation*}
		S_2 = \| \gamma(\cL)P_h\cW - \gamma(\cL_h)P_h( P_h\cW)\|_{L^2(\Omega;L^2(\cM))},
	\end{equation*}
	where, following \Cref{def:wtilde}, $\Vert P_h\cW \Vert_{L^2(\Omega;L^2(\cM))}^2 = \Vert \sum_{j=1}^{N_h} \xi_j e_j^h \Vert_{L^2(\Omega;L^2(\cM))}^2 =N_h < \infty$ (cf. \Cref{rem:tildeW}).
	 Hence, we can use \Cref{prop:det} to deduce that since $\alpha \le 1$,
	\begin{equation*}
		S_2 \lesssim |\log h| h^{2\alpha} \Vert P_h\cW \Vert_{L^2(\Omega;L^2(\cM))}. 	
	\end{equation*}
	Note then that since $N_h \propto h^{-d}$ (uniform mesh assumption), we can write $\Vert P_h\cW \Vert_{L^2(\Omega;L^2(\cM))} = N_h^{1/2} \lesssim h^{-d/2}$. Therefore, we can conclude that
		\begin{equation*}
		S_2 \lesssim |\log h| h^{2\alpha -d/2}.
	\end{equation*}
	Putting together the bounds obtained for $S_1$ and $S_2$, we can then conclude that if $\alpha \le 1$, 
	\begin{align*}
		\| \cZ- \tilde \cZ_h\|_{L^2(\Omega;L^2(\cM))} 
		&\leq S_1 + S_2 \lesssim  |\log h| h^{2\alpha -d/2} =|\log h| h^{2\min\lbrace(\alpha-d/4); 1\rbrace},
	\end{align*}
	where
	 we use the fact that $\alpha-d/4<\alpha<1$ to derive the last inequality.

	\medspace
	
	\paragraph{\underline{Case $\alpha>1$}} 
	Let now $\varepsilon_h = \min\lbrace d/2;\alpha-d/2 \rbrace/(2\vert\log h\vert)$, $\eta = 2(d/4 + \varepsilon_h)$, and let $\zeta = 2(\alpha-d/4 -\varepsilon_h) >0$, so that $\alpha = (\eta+\zeta)/2$. 
	Note in particular that $\varepsilon_h >0$ since $d\in\lbrace 1, 2\rbrace$ and $\alpha>1\geq d/2$, and $\varepsilon_h<(\alpha-d/2)/2<\alpha-d/2<\alpha-d/4$. 
	Besides, $\eta > d/2$ and $\zeta >0$.
	We define the function $\tilde \gamma(x) = \gamma(x)x^{\zeta/2}$. 
	Note that $\gamma$ is an $\alpha$-\psd and $\tilde{\gamma}$ decays as $|\tilde \gamma(x)| \lesssim |x|^{-\left(\alpha-\zeta/2\right)} = |x|^{-\eta/2}$. Hence,
	$\tilde{\gamma}$ is a $(\eta/2)$-\psd.
	Now, by definition of $\tilde\gamma$,
	\begin{align*}
		\| \cZ- \tilde \cZ_h\|_{L^2(\Omega;L^2(\cM))} 
		 &= \| \tilde \gamma(\cL)\cL^{-\zeta/2}\cW- \tilde{\gamma}(\cL_h)\cL_h^{-\zeta/2}P_h \cW\|_{L^2(\Omega;L^2(\cM))},
	\end{align*}
	so that the triangle inequality yields $\| \cZ- \tilde \cZ_h\|_{L^2(\Omega;L^2(\cM))} \leq E_1 + E_2$
	where we take 
	\begin{align*}
		E_1 &
		=\| \big(\tilde \gamma(\cL)-\tilde\gamma(\cL_h)P_h\big) \cL^{-\zeta/2}\cW\|_{L^2(\Omega;L^2(\cM))},\\ 
		E_2 &
		=\| \tilde \gamma(\cL_h)\big( P_h\cL^{-\zeta/2}\cW-\cL_h^{-\zeta/2}P_h \cW\big)\|_{L^2(\Omega;L^2(\cM))}.
	\end{align*}

	For the term $E_1$, we first note that $\zeta =d/2+2(\alpha-d/2 -\varepsilon_h)>d/2$ since $\varepsilon_h <\alpha-d/2$. 
	Therefore, we can use \cite[Lemma 4.2]{Bonito2024} and the fact that $h$ (and therefore $\varepsilon_{h}$) is upper-bounded to deduce that
	\begin{align*}
		\| \cL^{-\zeta/2}\cW\|_{L^2(\Omega;L^2(\cM))}^2 \leq \frac{\zeta}{\zeta-d/2} \lesssim 1.
	\end{align*}
	In particular, $\cL^{-\zeta/2}\cW$ is in $L^2(\Omega;L^2(\cM))$. Let then $p=\min\{\alpha-d/2; 1\}-2\varepsilon_{h}$. In particular,  we have $p\in (0,1)$ since $p\le 1-2\varepsilon_{h}<1$ and $p > \min\{\alpha-d/2; 1\}-\min\{\alpha-d/2; d/2\}\ge 0$.
	Moreover, we have $\zeta/2-p=d/4+(\alpha-d/2)-\min\{\alpha-d/2; 1\}+\varepsilon_{h}\ge d/4 + \varepsilon_{h} >d/4$. Hence, using once again \cite[Lemma 4.2]{Bonito2024}, we can conclude that
	\begin{align*}
		\| \cL^p\cL^{-\zeta/2}\cW\|_{L^2(\Omega;L^2(\cM))}^2
		&=\| \cL^{-(\zeta/2-p)}\cW\|_{L^2(\Omega;L^2(\cM))}^2\\
		&\lesssim \frac{\zeta/2-p}{\zeta/2-p-d/4}\lesssim
		\begin{cases}
			\vert \log h\vert, &\text{if } \alpha-d/2\le 1 \\
			1, &\text{if } \alpha-d/2> 1
		\end{cases}
	\end{align*}
	where we used the definition of $\varepsilon_h$ to derive the last inequality. 
	  Hence, we can apply \Cref{prop:det} to get
	\begin{align*}
		E_1^2 &= \E[\| \big(\tilde \gamma(\cL)-\tilde\gamma(\cL_h)P_h\big) \cL^{-\zeta/2}\cW\|_{L^2(\cM)}^2] \\
		 &\lesssim C_{\eta/2+p}(h)^2 h^{4\min\{\eta/2+p; 1\}}\E[\|  \cL^p\cL^{-\zeta/2}\cW\|_{L^2(\cM)}^2] \\
		 &\lesssim C_{\eta/2+p}(h)^2 h^{4\min\{\eta/2+p; 1\}}K'_\alpha(h)^2,
	\end{align*}
	where we take on the one hand, $C_{\eta/2+p}(h)=\vert \log h\vert$ if $\eta/2+p\le1$, and $C_{\eta/2+p}(h)=1$ otherwise, and on the other hand  $K'_{\alpha}(h)= \vert \log h\vert^{1/2}$ if $\alpha \le 1+d/2$ and $K'_\alpha(h)=1$ otherwise.
	
	Finally, note that according to \Cref{lem:min}, $\eta/2+p=\min\{\alpha-d/2; 1\}+d/4-\varepsilon_{h}\ge \min\{\alpha-d/4; 1\}-\varepsilon_{h}$ which implies that $h^{4\min\{\eta/2+p; 1\}} \le h^{4(\min\{\alpha-d/4; 1\}-\varepsilon_{h})}=h^{4\min\{\alpha-d/4; 1\}}\exp(2\min\lbrace d/2;\alpha-d/2 \rbrace)$.
	Besides, we note that if $\alpha>1+d/2$, then $\eta/2+p=1+d/4-\varepsilon_h>1$ (since by construction $\varepsilon_{h}<d/4$), and therefore $C_{\eta/2+p}(h)=1$.
	 We can therefore conclude that
	\begin{equation*}
		E_1 \lesssim K_{1}(h)  h^{2\min\{\alpha-d/4; 1\}}
	\end{equation*}
where  $K_{1}(h)=\vert \log h\vert^{3/2}$ if $1<\alpha \le 1+d/4+\varepsilon_{h}$, and $K_{1}(h)=\vert \log h\vert^{1/2}$ if $1+d/4+\varepsilon_{h}< \alpha\le 1+d/2$ and $K_{1}(h)=1$ if $\alpha >1+d/2$. 
	
	For $E_2$, we first use the definition of the projection operator $P_h$ and the self-adjointness of $\cL$ and $\cL_h$ (which is a consequence of the definition of the associated bilinear form $\mathsf{A}_{\cM}$)
	\begin{equation*}
		\begin{aligned}
			E_2^2
			&=\bigg\| \sum_{i=1}^{N_h} \tilde{\gamma}(\lambda_i^h)\bigg(\big(P_h\cL^{-\zeta/2}\cW, e_i^h\big)_{L^2(\cM)}-\big(\cL_h^{-\zeta/2}P_h\cW, e_i^h\big)_{L^2(\cM)}\bigg)e_i^h\bigg\|_{L^2(\Omega;L^2(\cM))}^2\\
			&=\bigg\| \sum_{i=1}^{N_h} \tilde{\gamma}(\lambda_i^h)\bigg(\big(\cL^{-\zeta/2}\cW, e_i^h\big)_{L^2(\cM)}-\big(P_h\cW, \cL_h^{-\zeta/2}e_i^h\big)_{L^2(\cM)}\bigg)e_i^h\bigg\|_{L^2(\Omega;L^2(\cM))}^2\\
			&=\bigg\| \sum_{i=1}^{N_h} \tilde{\gamma}(\lambda_i^h)\bigg( \big(\cW,\cL^{-\zeta/2}e_i^h\big)_{L^2(\cM)}-\big(\cW, \cL_h^{-\zeta/2}e_i^h\big)_{L^2(\cM)}\bigg)e_i^h\bigg\|_{L^2(\Omega;L^2(\cM))}^2 \\
			&=\bigg\| \sum_{i=1}^{N_h} \tilde{\gamma}(\lambda_i^h) \big(\cW,\cL^{-\zeta/2}e_i^h- \cL_h^{-\zeta/2}e_i^h\big)_{L^2(\cM)}e_i^h\bigg\|_{L^2(\Omega;L^2(\cM))}^2.
		\end{aligned}
	\end{equation*}
	Using the orthonormality of the basis $\lbrace e_i^h\rbrace_{1\le i\le N_h}$ and the definition of the white noise $\cW$, we conclude that
	\begin{equation*}
	\begin{aligned}
		E_2^2
		&=\E\bigg[ \sum_{i=1}^{N_h} \tilde{\gamma}(\lambda_i^h)^2 \vert\big(\cW,\cL^{-\zeta/2}e_i^h- \cL_h^{-\zeta/2}e_i^h\big)_{L^2(\cM)}\vert^2 \bigg] \\
		&=\sum_{i=1}^{N_h} \tilde{\gamma}(\lambda_i^h)^2 \Vert\cL^{-\zeta/2}e_i^h- \cL_h^{-\zeta/2}e_i^h\Vert_{L^2(\cM)}^2.
	\end{aligned}
\end{equation*}

On the one hand, using \Cref{prop:det} with $f=e_i^h\in S_h^\ell$ ($ 1\le i\le N_h$), which has norm $\Vert e_i^h\Vert_{L^2(\cM)}=1$, we have
\begin{align*}
	\Vert\cL^{-\zeta/2}e_i^h- \cL_h^{-\zeta/2}e_i^h\Vert_{L^2(\cM)}
	\lesssim C_{\zeta}(h)h^{2\min\{\zeta/2; 1\}}, \quad\text{where}\quad C_\zeta(h) = 
	\begin{cases}
		1 , & \zeta > 2,	 \\
		|\log h|   , & \zeta \le 2.
	\end{cases}	
\end{align*}
Note that by definition of $\zeta$ and using \Cref{lem:min}, we have $\min\{\zeta/2; 1\}\ge \min\{\alpha -d/4; 1\} -\varepsilon_{h} >0$ (since $\varepsilon_{h} \in (0, \alpha-d/4)$), which in turn gives $h^{2\min\{\zeta/2; 1\}} \le h^{2(\min\{\alpha -d/4; 1\} -\varepsilon_{h})}\lesssim h^{2\min\{\alpha -d/4; 1\}}$. Hence, we can conclude that
\begin{align*}
	\Vert\cL^{-\zeta/2}e_i^h- \cL_h^{-\zeta/2}e_i^h\Vert_{L^2(\cM)}
	\lesssim C_{\zeta}(h)h^{2\min\{\alpha - d/4; 1\}}
\end{align*}
and therefore, $E_2^2
\lesssim C_{\zeta}(h)^2 h^{4\min\{\alpha - d/4; 1\}} \sum_{i=1}^{N_h} \tilde{\gamma}(\lambda_i^h)^2$.
On the other hand,  the term $\sum_{i=1}^{N_h} \tilde\gamma(\lambda_i^h)^2 $can be bounded using \Cref{eq:SF_growth} and \Cref{eq:gbounds_ev} which yield that 
\begin{align*}
	\sum_{i=1}^{N_h} \tilde\gamma(\lambda_i^h)^2  \lesssim \sum_{i=1}^{N_h} (\lambda_i^h)^{-\eta} \lesssim \sum_{i=1}^{N_h} i^{-2\eta/d} 
	\lesssim \int_{1}^{\infty}t^{-2\eta/d} = \frac{1}{2\eta/d-1}
	 \lesssim \vert\log h\vert,
\end{align*}
since $2\eta/d = 1 +4\varepsilon_{h}/d > 1$ and by definition of $\varepsilon_{h}$. Hence, we can  conclude that 
\begin{equation*}
	\begin{aligned}
		E_2
		&
		\lesssim  C_{\zeta}(h) \vert \log h\vert^{1/2} h^{2\min\{\alpha - d/4; 1\}}
		\lesssim  K_2(h) h^{2\min\{\alpha - d/4; 1\}},
	\end{aligned}
\end{equation*}
where $K_{2}(h)=C_\zeta(h) \vert \log h\vert^{1/2}$, i.e., $K_2(h) = \vert \log h\vert^{3/2}$ if $1<\alpha \le 1+d/4+\varepsilon_{h}$, and $K_{2}(h)= \vert \log h\vert^{1/2}$ if $\alpha > 1+d/4+\varepsilon_{h}$. 

Hence, by putting together the estimates for $E_1$ and $E_2$, and noting that $K_1(h)\le K_2(h)$, we can conclude that if $\alpha>1$, then
	\begin{align*}
	\| \cZ- \tilde \cZ_h\|_{L^2(\Omega;L^2(\cM))} 
	\leq E_1 + E_2 \lesssim K_2(h) h^{2\min\{\alpha - d/4; 1\}}.
\end{align*}
In particular note that for any $\alpha > 1$, we can clearly upper-bound $K_2(h)$ by $K_2(h) \lesssim  \vert\log h\vert^{3/2}$. But if moreover $\alpha\ge 1+d/2$, we have  \revision{by definition of $\varepsilon_{h}$, that} $1 + d/4 +\varepsilon_{h}< 1+d/4+ \min\lbrace d/2;\alpha-d/2 \rbrace/2 \le 1+d/2 \le \alpha$, and therefore $K_2(h)=\vert \log h\vert^{1/2}$. Hence, we can conclude that if $\alpha>1$, $\| \cZ- \tilde \cZ_h\|_{L^2(\Omega;L^2(\cM))}\lesssim  K'_2(h) h^{2\min\{\alpha - d/4; 1\}}$ where $K'_2(h)=\vert\log h\vert^{3/2}$ if $1<\alpha< 1+d/2$ and $K'_2(h)=\vert\log h\vert^{1/2}$ if $\alpha\ge 1+d/2$.
And finally gathering the estimates obtained in the two cases $\alpha\le 1$ and $\alpha>1$, we obtain the result stated in the lemma.
\end{proof}

For the error between $\big(\tilde \cZ_h\big)^{-\ell}$ and $\sZb_h$ we get the following estimate, inspired by \cite[Lemma 4.4]{Bonito2024}.

\begin{lemma}
	It holds that
	\begin{align*}
		\big\|\big(\tilde \cZ_h\big)^{-\ell} - \sZb_h\big\|_{L^2(\Omega;L^2(\cM_h))} \lesssim |\log(h)|^{(d-1)/2} ~h^2. 		
	\end{align*}
	\label{lem:lemma_bonito_error}
\end{lemma}

\begin{proof}
	Let 
	\begin{equation*}
		\cE_h = \big\|\left(\tilde \cZ_h\right)^{-\ell}- \sZb_h\big\|_{L^2(\Omega;L^2(\cM_h))}=\big\| \left(\gamma(\cL_h)\widetilde{\cW}\right)^{-\ell} - \gamma(\sL_h)\sP_h(\sigma\widetilde{\cW}^{-\ell})\big\Vert_{L^2(\Omega;L^2(\cM_h))}.
	\end{equation*}
	We note that $\widetilde{\cW}$ is an $S_h^\ell$-valued random variable. Therefore, we can apply \Cref{lem:error_polyhedral} to realizations of $\widetilde{\cW}$, and take the expectation on both sides to get
	\begin{align*}
		\cE_h &\lesssim h^2 \| \cL_h^{-\min\lbrace \alpha+d/4;1\rbrace/2}\widetilde{\cW}\Vert_{L^2(\Omega;L^2(\cM))}.
	\end{align*}
	We then distinguish two cases. First, if $d=1$, then we note that $\min\lbrace \alpha+d/4;1\rbrace \ge \alpha+d/4 > d/2$ since $\alpha >d/4$.
	Besides, \cite[Lemma 4.2]{Bonito2024} yields that for all $r \in (d/2,2)$, there is a constant $C>0$ such that
	\begin{align}
		\label{eq:bonitolem42}
		\|\cL_h^{-r/2} \widetilde{\cW}\|_{L^2(\Omega;L^2(\cM))}^2 \leq C\frac{r}{r-d/2}. 
	\end{align}
	Hence, using estimate~\eqref{eq:bonitolem42}, we conclude that
	\begin{align*}
		\cE_h &\lesssim h^2 \frac{\min\lbrace \alpha+d/4;1\rbrace}{\min\lbrace \alpha+d/4;1\rbrace-d/2}\lesssim h^2.
	\end{align*}
	If now $d=2$, since  $\alpha+d/4>d/2 =1$, we have $\min\lbrace \alpha+d/4;1\rbrace =1$. 
	We then note that by the proof of \cite[Lemma 4.2]{Bonito2024}, for any $\varepsilon >0$ 
	\begin{align*}
		\| \cL_h^{-1/2}\widetilde{\cW}\Vert_{L^2(\Omega;L^2(\cM))}^2 \lesssim \sum_{j=1}^{N_h} \lambda_j^{-1} \lesssim \lambda_{N_h}^\varepsilon \sum_{j=1}^{N_h} \lambda_j^{-1-\varepsilon}. 
	\end{align*}
	Apply now \Cref{eq:gbounds_ev} to see that 
	\begin{align*}
		\|\cL_h^{-1/2} \widetilde{\cW}\|_{L^2(\Omega;L^2(\cM))}^2 \lesssim N_h^{\varepsilon} \sum_{j=1}^{N_h} j^{-(1+\varepsilon)} \lesssim N_h^{\varepsilon}\frac{1+\varepsilon}{\varepsilon} \lesssim h^{-2\varepsilon}\frac{1+\varepsilon}{\varepsilon}, 
	\end{align*}
	meaning that for any $\varepsilon >0$, 
	\begin{align*}
		\cE_h	 \lesssim \bigg(\frac{1+\varepsilon}{\varepsilon} \bigg)^{1/2}h^{2-\varepsilon} .
	\end{align*}
	In particular, taking $\varepsilon = |\log(h)|^{-1}$ yields  
	\begin{align*}
		\cE_h	 \lesssim |\log(h)|^{1/2}h^{2},
	\end{align*}
	which concludes the proof.
\end{proof}

Finally, for the error between $\sZb_h$ and $\sZ_h$ we get the following estimate.

\begin{lemma}\label{corol:err_2sfem}
	Let $h\in(0,h_0)$. It holds that  there is a constant such that 
	\begin{align*}
		\|\sZ_h - \sZb_h\|_{L^2(\Omega;L^2(\cM_h) )} \lesssim h^2. 
	\end{align*}
\end{lemma}

\begin{proof}
	To prove this statement, we note that, using the same notations as the ones in the proof of \Cref{lem:noise-lemma},
	\begin{align*}
		\|\sZ_h - \sZb_h\|_{L^2(\Omega;L^2(\cM_h) )}^2
		&=\left\|\gamma(\sL_h)\left(\sWb-\sW\right)\right\|_{L^2(\Omega;L^2(\cM_h) )}^2 \\
		&= \E\left[\left\|\sum_{j=1}^{N_h} \gamma(\Lambda_j^h)(\alpha_j-\beta_j)E_j^h\right \|_{L^2(\cM_h)}^2\right] = \sum_{j=1}^{N_h} \gamma(\Lambda_j^h)^2 \E[\vert\alpha_j-\beta_j\vert^2],
	\end{align*}
	where we applied the orthogonality of the eigenfunctions in the last step. 
	Now, following the definition of $\alpha_j$ and $\beta_j$ given in the proof of \Cref{lem:noise-lemma},
	\begin{align*}
		\E[\vert\alpha_j-\beta_j\vert^2] 
		&=\E\bigg[\bigg\vert\big(\cW,\big((1-\sigma^{-1/2})E_j^h\big)^\ell\big)_{L^2(\cM)}\bigg\vert^2\bigg]
		= \left\|\left((1-\sigma^{-1/2})E_j^h\right)^\ell\right\|^2_{L^2(\cM)}\\
		& = \left\|(\sigma^{1/2}-1)E_j^h\right\|^2_{L^2(\cM_h)} \leq \left\|(\sigma^{1/2}-1)\right\|^2_{L^\infty(\cM_h)}.
	\end{align*}
	Besides, for all $x \in \cM_h$, it holds that 
	\begin{align*}
		\sigma^{1/2}(x)-1 {=} \sqrt{1+\sigma(x)-1} -1 \leq 	\sqrt{1+|\sigma(x)-1|} -1 \leq 1 + \frac{1}{2} |\sigma(x)-1| -1 {=} \frac{1}{2} |\sigma(x)-1|.
	\end{align*}
	Therefore, 
	\begin{align*}
		\left\|(\sigma^{1/2}-1)\right\|^2_{L^\infty(\cM_h)} \leq 	 \left\|(\sigma-1)\right\|^2_{L^\infty(\cM_h)} \leq Ch^2,
	\end{align*}
	where \cite[Lemma 4.1]{Dziuk2013} was applied in the final step. 
	We conclude that 
	\begin{align*}
		\|\sZ_h - \sZb_h\|_{L^2(\Omega;L^2(\cM_h) )}^2 \leq Ch^4\sum_{j=1}^{N_h} \gamma(\Lambda_j^h)^2,
	\end{align*}
	and it remains to show that $\sum_{j=1}^{N_h} \gamma(\Lambda_j^h)^2	$ is bounded by a constant. 
	To this end, we use \Cref{lem:ev_bound}, which implies that there exists some constant $C_\lambda>0$ such that 
	\begin{align*}
		|\Lambda_j^h/\lambda_j^h-1| \leq C_\lambda h^2. 
	\end{align*}
	Recall then that the mesh size $h$ satisfies $h<h_0$, where $h_0\in (0,1)$. Without loss of generality, let us further assume that $C_\lambda h_0 <1$. Now, by the growth assumption on $\gamma$, 
	\begin{align*}
		\sum_{j=1}^{N_h} \gamma(\Lambda_j^h)^2 
		& \lesssim \sum_{j=1}^{N_h} |\Lambda_j^h|^{-2\alpha} 
		= \sum_{j=1}^{N_h} |\lambda_j^h|^{-2\alpha} \left|1+\left(\frac{\Lambda_j^h}{\lambda_j^h}-1\right)\right|^{-2\alpha}  
		\le (1-Ch_0^2)^{-2\alpha}\sum_{j=1}^{N_h} |\lambda_j^h|^{-2\alpha}\\
		& \lesssim \sum_{j=1}^{N_h} |\lambda_j^h|^{-2\alpha}.
	\end{align*}
	Using \Cref{eq:SF_growth} and \Cref{eq:gbounds_ev}, we bound
	\begin{align*}
		\sum_{j=1}^{N_h} |\lambda_j^h|^{-2\alpha} 
		\leq\sum_{j=1}^{N_h} |\lambda_j|^{-2\alpha} 
		\lesssim  \sum_{j=1}^{\infty} j^{-4\alpha/d} = \zeta(4\alpha/d),
	\end{align*}
	where $\zeta$ denotes the Riemann zeta function. 
	Thus,  $\sum_{j=1}^{N_h} \gamma(\Lambda_j^h)^2 $ is bounded by a constant and
	\begin{align*}
		\|\sZ_h - \sZb_h\|_{L^2(\Omega;L^2(\cM_h) )}^2 \leq Ch^4,
	\end{align*}
	which proves the lemma. 
\end{proof}

Equipped with the estimates derived in \Cref{corol:err_2sfem,lem:lemma_bonito_error,lem:finest} we are ready to prove \Cref{th:err_fem}.

\begin{proof}[Proof of \Cref{th:err_fem}]
	By summing the three estimates derived in \Cref{corol:err_2sfem,lem:lemma_bonito_error,lem:finest}, we get
	\begin{align*}
		\|\cZ-  \sZ_h^\ell\|_{L^2(\Omega;L^2(\cM))} \lesssim
		C_\alpha(h) h^{2\min\{\alpha-d/4;1\}} +|\log(h)|^{(d-1)/2} ~h^2 + h^2.
	\end{align*}
	\revision{Note then} that $2\min\{\alpha-d/4;1\} \le 2$ and that $|\log(h)|^{(d-1)/2} \lesssim C_\alpha(h)$. Hence, we have the bound $|\log(h)|^{(d-1)/2} ~h^2 \lesssim C_\alpha(h) h^{2\min\{\alpha-d/4;1\}}$ which allows us to conclude that 
	\begin{align*}
		\|\cZ-  \sZ_h^\ell\|_{L^2(\Omega;L^2(\cM))} \lesssim
		C_\alpha(h) h^{2\min\{\alpha-d/4;1\}},
	\end{align*}
	and the proof is complete.
\end{proof}

Having bounded the SFEM error, we are now ready to derive the additional error associated with the Chebyshev approximation which we use to compute SFEM--Galerkin approximations in practice.

\subsection{Error analysis of the Galerkin--Chebyshev approximation}
\label{sec:conv_cheb_ap}

The error between the field $\cZ$ and its Galerkin--Chebyshev approximation $\sZ_{h,M}$ described in \Cref{sec:method} is derived by combining \Cref{th:err_fem} with an error bound between the SFEM--Galerkin approximation $\sZ_{h}$ and the Galerkin--Chebyshev approximation $\sZ_{h,M}$ obtained using \cite[Theorem 5.8]{Lang2023}. The latter bound is expressed in the next result.

\begin{lemma}	
	Let $[\lambda_{\min},\lambda_{\max}] \subset (0,\infty)$ be the interval on which the Chebyshev polynomial approximation $P_{\gamma, M}$ is computed, and let $\xi=\lambda_{\min}/(\lambda_{\max}-\lambda_{\min})$. Then, there exists a constant $C>0$ such that the error between the discretized field $\sZ_{h}$ and its approximation $\sZ_{h,M}$ is upper-bounded by
	\begin{equation}
		\begin{aligned}
			\| \sZ_{h,M}-\sZ_{h}\Vert_{L^2(\Omega;L^2(\cM_h))}
			\le   C 
			{h}^{-d/2} 
			\epsilon_\xi^{-1}(1+\epsilon_\xi)^{-M}
		\end{aligned}\label{eq:exprate}
	\end{equation}
	with $\epsilon_\xi= \xi+\sqrt{\xi(2+\xi)}$. In particular, setting $\lambda_{\min}=V_{-}$ and $\lambda_{\max}=\Lambda_{N_h}^h$, the error is bounded by
	\begin{equation}
		\begin{aligned}
			\| \sZ_{h,M}-\sZ_{h}\Vert_{L^2(\Omega;L^2(\cM_h))}
			\le   C_{V}^{-1}
			{h}^{-(d/2+1)} \exp(-C_V hM)
		\end{aligned}\label{eq:exprate_bis}
	\end{equation}
	for some constant $C_{V}>0$ proportional to $\sqrt{V_{-}}$.
	\label{prop:conv_cheb_ap}
\end{lemma}

\begin{proof}
	Let $E_\xi\subset \C$ be the ellipse centered at $z=(\lambda_{\min}+\lambda_{\max})/2$, with foci $z_1=\lambda_{\min}$ and $z_2=\lambda_{\max}$, and semi-major axis $a_\xi=\lambda_{\max}/2$. 
	In particular, note that $E_\xi \subset H_{\pi/2}$ and for any $z\in E_\xi$, $\Re(z)\ge \lambda_{\min}/2 >0$.  
	Hence, since $\gamma$ is an \psd, by definition $\gamma$ is holomorphic and bounded inside $E_\xi$.
	We can then adapt the same proof as in \cite[Theorem 5.8]{Lang2023} to obtain the stated proposition. 
\end{proof}

Hence, for a fixed mesh size $h$, the approximation error $\| \sZ_{h,M}-\sZ_{h}\Vert_{L^2(\Omega;L^2(\cM_h))}$ converges to $0$ as the order of the polynomial approximation $M$ goes to infinity.  
Choosing $M$ as a function of $h$ that grows fast enough then allows us to ensure the convergence of the approximation error as $n$ goes to infinity \cite[Section 5.2]{Lang2023}.
In particular, by taking $M=M_h=(d/2+3)C_V^{-1}h^{-1}\vert \log h\vert=(d/2+3)C_V^{-1}h^{-1} \log (h^{-1})$ the bound~\eqref{eq:exprate_bis} becomes
\begin{equation}
	\begin{aligned}
		\| \sZ_{h,M}-\sZ_{h}\Vert_{L^2(\Omega;L^2(\cM_h))}
		&\le   C_{V}^{-1} \exp((d/2+1)\log (h^{-1})-(d/2+3)\log (h^{-1}))\\
		&= C_{V}^{-1} h^2.
	\end{aligned}\label{eq:exprate_ter}
\end{equation}

Putting together \Cref{th:err_fem} and \Cref{eq:exprate_ter}, we can now state the following result which provides a bound between the field $\cZ$ and its Galerkin--Chebyshev approximation $\sZ_{h,M_h}$.

\begin{theorem}
	Let $\gamma$ be an $\alpha$-\psd with $\alpha>d/4$, let $h\in(0,h_0)$ and take $M_h=(d/2+3)C_V^{-1}h^{-1}\vert \log h\vert$. Then,
	The strong error between the random field $\mathcal{Z}$ by its Galerkin--Chebyshev approximation $\sZ_{h,M_h}^\ell$  satisfies the bound 
	\begin{equation}
		\label{eq:err_str_gc}
		\|\cZ-  \sZ_{h,M_h}^\ell\|_{L^2(\Omega;L^2(\cM))} \lesssim
		C_\alpha(h) h^{2\min\{\alpha-d/4;1\}},
	\end{equation}
	where  $C_\alpha(h)=\vert \log h\vert$ if $d/4<\alpha \le 1$, $C_\alpha(h)=\vert \log h\vert^{3/2}$ if $1<\alpha < 1+d/2$, and $C_\alpha(h)=\vert \log h\vert^{1/2}$ if $\alpha \ge 1+d/2$.
	\label{th:err_gc}
\end{theorem}

\begin{remark}\label{rem:cheb}
	In practice, it is common to \q{chop} a Chebyshev polynomial approximation, i.e. to truncate the polynomial approximation at an early order.  Indeed, let us assume that we have computed the coefficients $\lbrace c_k\rbrace_{0\le k\le M_h}$ of the Chebyshev series up to the order $M=M_h$ and let $c_{\max} = \max_{0\le k \le M_h} \vert c_k\vert$. Let $\varepsilon>0$ be fixed but arbitrary and assume that there exists  $m_\varepsilon \in \lbrace 0,\dots, M_h-1\rbrace$ such that for any $k>m_\varepsilon$, $\vert c_k \vert / c_{\max} < \varepsilon$. The error between the field $\sZ_{h,M_h}$ and its \q{chopped} counterpart $\sZ_{h,m_\varepsilon}$ is given by
	\begin{equation*}
		\begin{aligned}
			\| \sZ_{h,M_h}-\sZ_{h,m_\varepsilon}\Vert_{L^2(\Omega;L^2(\cM_h))}^2
			&=\| \sum_{i=1}^{N_h} (P_{\gamma,M_h}(\Lambda_{i}^h)-P_{\gamma,m_\varepsilon}(\Lambda_{i}^h))W_i^h E_i^h\Vert_{L^2(\Omega;L^2(\cM_h))}^2\\
			&=\sum_{i=1}^{N_h} (P_{\gamma,M_h}(\Lambda_{i}^h)-P_{\gamma,m_\varepsilon}(\Lambda_{i}^h))^2
		\end{aligned}
	\end{equation*}
	where $P_{\gamma,M_h}(\lambda)=\sum_{k=0}^{M_h}c_k T_k(\lambda)$ is the $M_h$-th order Chebyshev approximation of $\gamma$ and $T_k$ denotes the $k$-th (shifted) Chebyshev polynomial. In particular, we have $\vert P_{\gamma,M_h}(\Lambda_{i}^h)-P_{\gamma,m_\varepsilon}(\Lambda_{i}^h)\vert =\vert\sum_{k=m_\varepsilon+1}^{M_h}c_k T_k(\lambda)\vert \le \sum_{k=m_\varepsilon+1}^{M_h}\vert c_k \vert~\vert T_k(\lambda)\vert \le \sum_{k=m_\varepsilon+1}^{M_h}\vert c_k \vert \le (M_h-m_\varepsilon)c_{\max} \varepsilon \le M_hc_{\max}\varepsilon$. Hence, the error between $\sZ_{h,M_h}$ and its \q{chopped} counterpart $\sZ_{h,m_\varepsilon}$ satisfies
	\begin{equation*}
		\begin{aligned}
			\| \sZ_{h,M_h}-\sZ_{h,m_\varepsilon}\Vert_{L^2(\Omega;L^2(\cM_h))}
			&\le N_h^{1/2} M_h c_{\max}\varepsilon \lesssim h^{-(1+d/2)}\vert \log h\vert~\varepsilon
		\end{aligned}
	\end{equation*}
	Hence, if we have $\varepsilon \lesssim \vert \log h \vert^{-1} h^{(3+d/2)}$ then we can conclude that $	\| \sZ_{h,M_h}-\sZ_{h,m_\varepsilon}\Vert_{L^2(\Omega;L^2(\cM_h))} \lesssim h^2$, thus implying that once again we have $	\|\cZ-  \sZ_{h,m_\varepsilon}^\ell\|_{L^2(\Omega;L^2(\cM))} \lesssim
	C_\alpha(h) h^{2\min\{\alpha-d/4;1\}}$. This means that in practice, if we notice that the coefficients of the Chebyshev polynomial approximation $\lbrace c_k\rbrace_{0\le k\le M_h}$ decay fast enough that we can find some $m_\varepsilon \in \lbrace 0,\dots, M_h-1\rbrace$ such that for any $k>m_\varepsilon$, $\vert c_k \vert / c_{\max} < \varepsilon \lesssim \vert \log h \vert^{-1} h^{(3+d/2)}$, we can replace the field $\sZ_{h,M_h}$ by its chopped counterpart while still maintaining the same strong error bound with respect to the field $\cZ$.
\end{remark}

\section{Numerical experiments}
\label{sec:num}
In this section, we present numerical experiments confirming the strong error bound in \Cref{eq:err_str_gc}. We consider two cases: when the manifold $\cM$ is a circle (hence $d=1$), and when the manifold $\cM$ is a sphere (hence $d=2$). Indeed, in both cases we are able to easily define nested meshes of various sizes $h$, and define appropriate noise projections at the different levels.

In each case, the field $\cZ$ is approximated by a Galerkin--Chebyshev field $\sZ_{h_{\text{fine}},M_{h_{\text{fine}}}}$ computed at a very fine mesh size $h_{\text{fine}}$. More precisely, 25 samples of 
$\lbrace \sZ_{h_{\text{fine}},M_{h_{\text{fine}}}}^{(i)} \rbrace_{1\le i\le 25}$ of the field $\sZ_{h_{\text{fine}},M_{h_{\text{fine}}}}$ are computed.
In particular, let us denote by $\sW^{(i)}_{h_{\text{fine}}}$ the projected white noise used to compute the sample $\cZ^{(i)}=P_{\gamma,M_{h_{\text{fine}}}}(\sL_{h_{\text{fine}}})\sW^{(i)}_{h_{\text{fine}}}$. 
For a coarse mesh size $h>h_{\text{fine}}$, each sample $\cZ^{(i)}$ is compared to the sample of $\sZ_{h,M_h}^{(i)}$ computed from the white noise $\sW^{(i)}_{h}$ obtained  by projecting the noise $\sW^{(i)}_{h_{\text{fine}}}$ (defined in the fine mesh) down to the coarse mesh.
Note that, following \Cref{rem:cheb}, the Galerkin--Chebyshev fields are systematically computed by chopping the Chebyshev polynomial approximations at a small level $\epsilon=10^{-12}$.
Besides, when computing the representation of a Galerkin--Chebyshev field in the nodal basis through \eqref{eq:compute_gc}, the matrix $\sqrt{\bm C}$ is computed as a Cholesky factorization of the mass matrix $\bm C$ in the first experiment, and is approximate using a (diagonal) mass lumping approach in the second experiment. Finally, the strong error at a mesh size $h$ is then estimated from the samples using 
\begin{equation}
	\|\cZ-  \sZ_{h,M_h}^\ell\|_{L^2(\Omega;L^2(\cM))}
	\approx \bigg(\frac{1}{25} \sum_{i=1}^{25}  \|\sZ_{h_{\text{fine}},M_{h_{\text{fine}}}}^{(i)}-  \sZ_{h,M_h}^{(i)}\|_{L^2(\cM_h)}^2 \bigg)^{1/2}.\label{eq:rmse}
\end{equation}

In the first experiment, we consider a random field $\cZ$ defined on the unit circle $\cM$ with \psd given by the formula $\gamma(\lambda)=(V_{0})^{\alpha}(\lambda+\cos(0.9\pi)\sqrt{\lambda})^{-\alpha}$ where $V_0=10^4$ and for $\alpha \in\lbrace 0.5,1.05,1.5\rbrace$. The bilinear form $\cM$ is taken as follows. The diffusion matrix $\mathcal{D}$ is defined by $\mathcal{D}(x)= I -\nu(x)\nu(x)^T$, $x\in\cM$, where $\nu(x)=x/\Vert x\Vert$ is the normal vector to a point $x\in\cM$ of the unit circle.  The function $V$ is defined as $V(x)=3V_0$, $x\in \cM$, if $\theta(x)\in ]\pi/2, 3\pi/2[$ and $V(x)=V_0$ otherwise, and $\theta(x)\in [0,2\pi[$ denotes the circular coordinate of the point $x\in\cM$. 
The fine mesh size is taken to be $h_{\text{fine}}=2^{-19}\pi$ and the coarse mesh sizes are $h=\lbrace 2^{-13}\pi, 2^{-14}\pi, 2^{-15}\pi, 2^{-16}\pi\rbrace$. An example of sample (on the fine mesh) obtained with this choice of parameters is displayed in \Cref{fig:err_circle_sample}. The results of this first experiment are presented in \Cref{fig:err_circle}, where the log of the strong error~\eqref{eq:rmse} is plotted against the log of the mesh size $h$, for the three \psd obtained by taking $\alpha \in\lbrace 0.5,1.05,1.5\rbrace$. As shown in the figure, we retrieve the rates (respectively $0.5$,$1.6$ and $2$) predicted by \Cref{th:err_gc}.

\begin{figure}
	\centering
	\subfigure[Sample (in green) of the random field with $\alpha=1.75$ on the fine circle mesh (in black). The distance to the circle indicates the value of taken by the field.]{\includegraphics[width=0.35\textwidth]{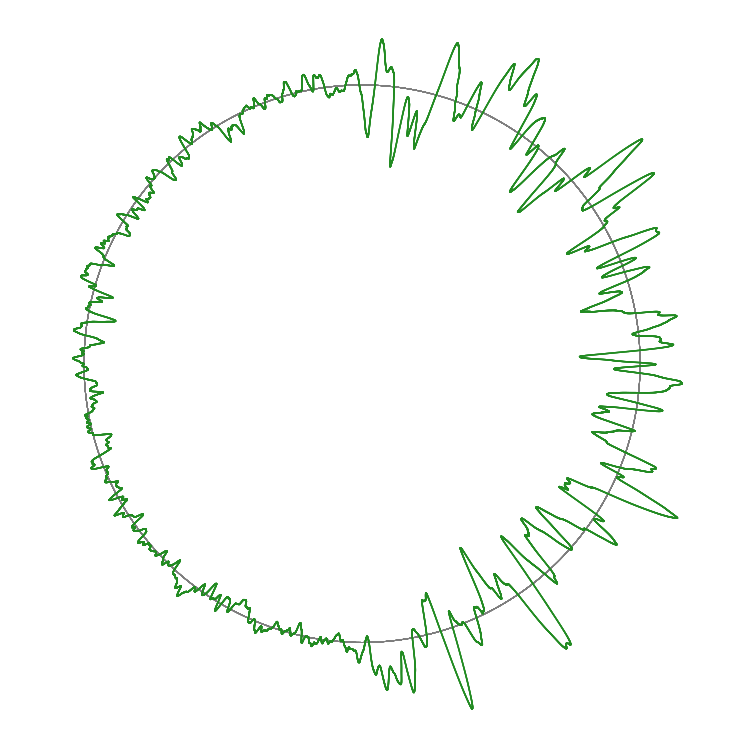}\label{fig:err_circle_sample}}
	\subfigure[Strong error as a function the mesh size.]{\includegraphics[width=0.64\textwidth]{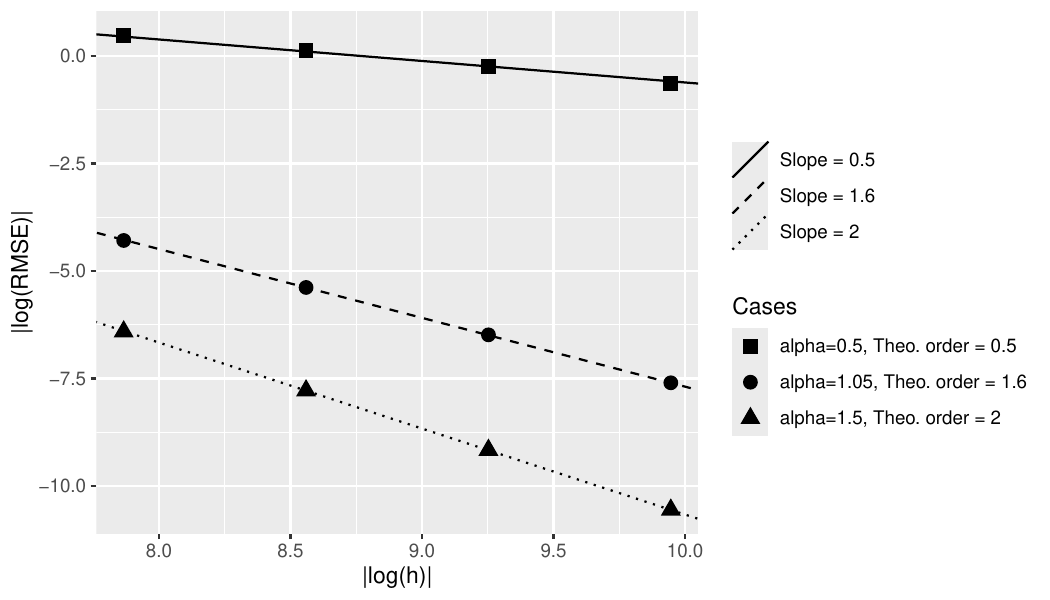}\label{fig:err_circle}}
	\caption{First numerical experiment (on the circle).}
\end{figure}

In the second experiment, we consider a random field $\cZ$ defined on the unit sphere $\cM$ with \psd given by the formula $\gamma(\lambda)=C_{0}\lambda^{-\alpha}$ where $C_0=500$ and for $\alpha \in\lbrace 0.75,1.25,2.25\rbrace$. The bilinear form $\cM$ is taken as follows. The diffusion matrix $\mathcal{D}$ is defined $\mathcal{D}(x)= \nabla_{\cM}f(x)(\nabla_{\cM}f(x))^{T} +\rho(x)\nabla_{\cM}^\perp f(x)(\nabla_{\cM}^\perp f(x))^{T}$, where  $f(x)=2\cos(\theta(x))\cos(\phi(x))\sin^2(\theta(x))$ and $\rho(x)=0.1+0.6/(1+e^{-4\cos(\theta(x))})$. Here $x\in\cM$ and  $(\theta(x),\phi(x))$ denotes spherical coordinates.
The function $V$ is defined as $V(x)=500(1+5\cos^2(\pi\theta(x)))$. 
The fine mesh size is taken to be $h_{\text{fine}}= 2.16\times 10^{-3}$ and the coarse mesh sizes are $h=\lbrace 
3.45\times 10^{-2}, 1.73\times 10^{-2},  8.63\times 10^{-3},  4.32\times 10^{-3} \rbrace$.  An example of sample (on the fine mesh) obtained with this choice of parameters is displayed in \Cref{fig:err_sphere_sample}.
The results of this second experiment are presented in \Cref{fig:err_sphere}, where once again the log of the strong error~\eqref{eq:rmse} is plotted against the log of the mesh size $h$, for the three \psd obtained by taking $\alpha \in\lbrace 0.75,1.25,2.25\rbrace$. As seen in \Cref{fig:err_sphere}, we retrieve the rates (respectively $0.5$,$1.5$ and $2$) predicted by \Cref{th:err_gc}.

\begin{figure}
	\centering
	\subfigure[Sample of the random field with $\alpha=1.75$ on the fine mesh. The colors indicate the value of taken by the field.]{\includegraphics[trim={0 0 2.1cm 0},clip,width=0.35\textwidth]{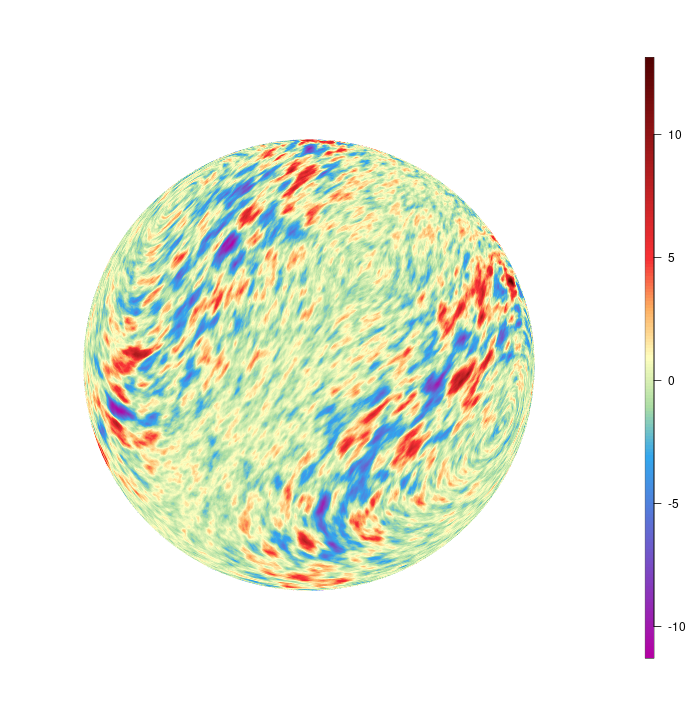}\label{fig:err_sphere_sample}}
	\subfigure[Strong error as a function the mesh size.]{\includegraphics[width=0.64\textwidth]{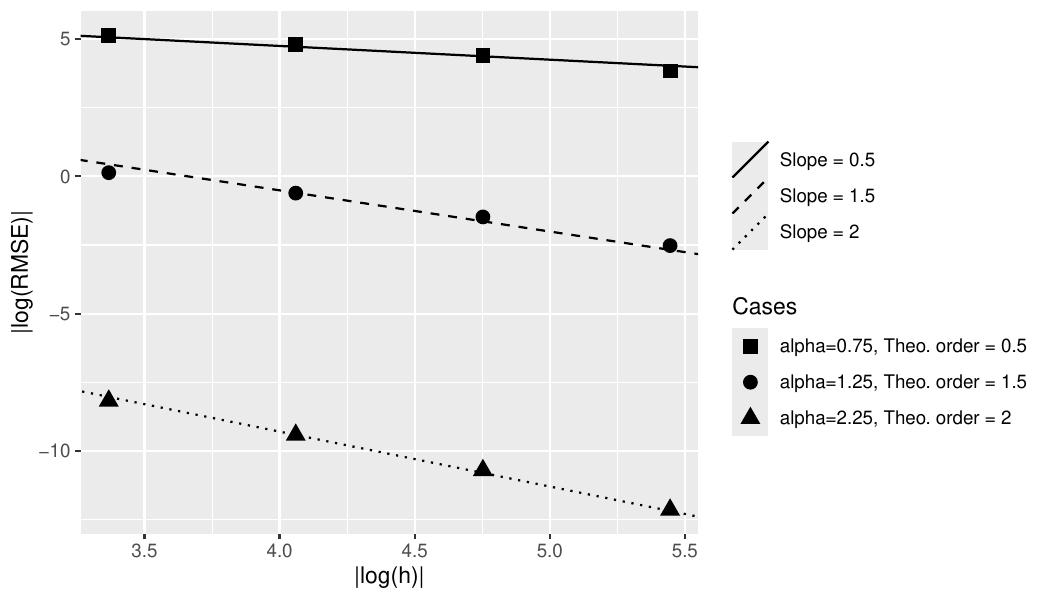}\label{fig:err_sphere}}
	\caption{Second numerical experiment (on the sphere).}
\end{figure}

\bibliographystyle{amsplainnat}
\bibliography{references.bib}

@misc{bonito2016,
	title = {Numerical {Approximation} of {Fractional} {Powers} of {Regularly} {Accretive} {Operators}},
	url = {http://arxiv.org/abs/1508.05869},
	language = {en},
	urldate = {2024-10-31},
	publisher = {arXiv},
	author = {Bonito, Andrea and Pasciak, Joseph E.},
	month = jul,
	year = {2016},
	note = {arXiv:1508.05869 [math]},
}

@TechReport{heinzel2002spectrum,
  author      = {Heinzel, Gerhard and R{\"u}diger, Albrecht and Schilling, Roland},
  title       = {Spectrum and spectral density estimation by the {D}iscrete {F}ourier transform ({DFT}), including a comprehensive list of window functions and some new at-top windows},
  institution = {{M}ax-{P}lanck-{I}nstitut f\"{u}r {G}ravitationsphysik},
  year        = {2002},
}

@inproceedings{BorovitskiyTMD20,
  author       = {Viacheslav Borovitskiy and
                  Alexander Terenin and
                  Peter Mostowsky and
                  Marc Peter Deisenroth},
  title        = {Mat{\'{e}}rn {G}aussian Processes on {R}iemannian Manifolds},
  booktitle    = {{A}dvances in {N}eural {I}nformation {P}rocessing {S}ystems 33},
  year         = {2020}
 }

@article{Boffi2010,
  title = {Finite element approximation of eigenvalue problems},
  volume = {19},
  ISSN = {1474-0508},
  journal = {Acta Numer.},
  author = {Boffi,  Daniele},
  year = {2010},
  month = may,
  pages = {1–120}
}

@incollection{Bonito2020,
  title={Finite element methods for the {L}aplace--{B}eltrami operator},
  author={Bonito, Andrea and Demlow, Alan and Nochetto, Ricardo H},
  booktitle={{H}andbook of {N}umerical {A}nalysis},
  volume={21},
  pages={1--103},
  year={2020},
  publisher={Elsevier}
}

@article{Bandara2021,
  title = {Eigenvalue asymptotics for weighted {L}aplace equations on rough {R}iemannian manifolds with boundary},
  ISSN = {0391-173X},
  journal = {Ann. Sc. Norm. Super. Pisa Cl. Sci.},
  publisher = {Scuola Normale Superiore - Edizioni della Normale},
  author = {Bandara,  Lashi and Nursultanov,  Medet and Rowlett,  Julie},
  year = {2021},
  month = dec,
  pages = {1843-1878}
}

@Article{BKK20,
  author    = {Bolin, David and Kirchner, Kristin and Kov\'acs, Mih\'aly},
  title     = {Numerical solution of fractional elliptic stochastic {PDE}s with spatial white noise},
  journal   = {IMA J. Numer. Anal.},
  year      = {2020},
  volume    = {40},
  number    = {2},
  pages     = {1051--1073},
  comment   = {arXiv:1705.06565 [math.NA]},
  doi       = {10.1093/imanum/dry091},
}

@article{Bonito2018,
	title={A Priori Error Estimates for Finite Element Approximations to Eigenvalues and Eigenfunctions of the {L}aplace--{B}eltrami Operator},
	author={Bonito, Andrea and Demlow, Alan and Owen, Justin},
	journal={SIAM J. Numer. Anal.},
	volume={56},
	number={5},
	pages={2963--2988},
	year={2018},
	publisher={SIAM}
}

@Article{Bonito2024,
  author    = {Bonito, Andrea and Guignard, Diane and Lei, Wenyu},
  title     = {Numerical Approximation of {G}aussian Random Fields on Closed Surfaces},
  journal   = {Comput. Appl. Math.},
  year      = {2024},
  volume    = {In press},
  number    = {0},
  month     = jan,
  issn      = {1609-9389},
  doi       = {10.1515/cmam-2022-0237},
  publisher = {Walter de Gruyter GmbH},
  url       = {http://dx.doi.org/10.1515/cmam-2022-0237},
}

@Article{Cox2020,
  author    = {Cox, Sonja G. and Kirchner, Kristin},
  title     = {Regularity and convergence analysis in {S}obolev and {H}\"{o}lder spaces for generalized {W}hittle–{M}atérn fields},
  journal   = {Numer. Math.},
  year      = {2020},
  volume    = {146},
  number    = {4},
  pages     = {819–873},
  month     = nov,
  issn      = {0945-3245},
  doi       = {10.1007/s00211-020-01151-x},
  publisher = {Springer Science and Business Media LLC},
  url       = {http://dx.doi.org/10.1007/s00211-020-01151-x},
}

@InCollection{Dziuk1988,
    AUTHOR = {Dziuk, Gerhard},
     TITLE = {Finite elements for the {B}eltrami operator on arbitrary
              surfaces},
 BOOKTITLE = {{P}artial {D}ifferential {E}quations and {C}alculus of {V}ariations},
    SERIES = {Lecture Notes in Math.},
    VOLUME = {1357},
     PAGES = {142--155},
 PUBLISHER = {Springer, Berlin},
      YEAR = {1988},
      ISBN = {3-540-50508-3},
   MRCLASS = {65N30},
  MRNUMBER = {976234},
MRREVIEWER = {Lubor\ Malina},
       DOI = {10.1007/BFb0082865},
       URL = {https://doi.org/10.1007/BFb0082865},
}

@article{Dziuk2013,
	author    = {Dziuk, Gerhard and Elliott, Charles M.},
	title     = {Finite element methods for surface {PDEs}},
	journal   = {Acta Numer.},
	year      = {2013},
	volume    = {22},
	pages     = {289--396},
	doi       = {10.1017/s0962492913000056},
	publisher = {Cambridge University Press ({CUP})},
}

@book{Farag2014, 
title={Biomedical Image Analysis: Statistical and Variational Methods},
 PUBLISHER = {Cambridge University Press, Cambridge},
 author={Farag, Aly A.},
  year={2014}}

@InCollection{Fujita1991,
	title = {Evolution problems},
	ISSN = {1570-8659},
	DOI = {10.1016/s1570-8659(05)80043-2},
	booktitle = {Handbook of Numerical Analysis},
	publisher = {Elsevier},
	author = {Fujita,  Hiroshi and Suzuki,  Takashi},
	year = {1991},
	pages = {789-928}
}

@book{Hall2013,
  title = {Quantum Theory for Mathematicians},
  ISBN = {9781461471165},
  ISSN = {2197-5612},
  url = {http://dx.doi.org/10.1007/978-1-4614-7116-5},
  DOI = {10.1007/978-1-4614-7116-5},
  journal = {Graduate Texts in Mathematics},
  publisher = {Springer, New York},
  author = {Hall,  Brian C.},
  year = {2013}
}

@Article{HLS18,
	author    = {Herrmann, Lukas and Lang, Annika and Schwab, {Ch}ristoph},
	title     = {Numerical analysis of lognormal diffusions on the sphere},
	journal   = {Stoch PDE: Anal. Comp.},
	year      = {2018},
	volume    = {6},
	number    = {1},
	pages     = {1--44},
	comment   = {(arXiv:1601.02500[math.PR], SAM report 2016-02, http://arxiv.org/abs/1601.02500},
	doi       = {10.1007/s40072-017-0101-x},
}

@Article{JKL22,
  author  = {Jansson, Erik and Kov\'acs, Mih\'aly and Lang, Annika},
  title   = {Surface finite element approximation of spherical {W}hittle--{M}at\'ern {G}aussian random fields},
  journal = {SIAM J. on Sci. Comput.},
  year    = {2022},
  volume  = {44},
  number  = {2},
  pages   = {A825--A842},
  comment = {arXiv:2102.08822},
  doi     = {10.1137/21M1400717},
}

@Article{Lang2023,
  author    = {Lang, Annika and Pereira, Mike},
  title     = {{G}alerkin--{C}hebyshev approximation of {G}aussian random fields on compact {R}iemannian manifolds},
  journal   = {BIT Numer. Math.},
  year      = {2023},
  volume    = {63},
  number    = {4},
  pages     = {51},
  issn      = {1572-9125},
  doi       = {10.1007/s10543-023-00986-8},
  publisher = {Springer Science and Business Media LLC},
  url       = {http://dx.doi.org/10.1007/s10543-023-00986-8},
}

@article{Lindgren2011,
  title = {An Explicit Link between {G}aussian Fields and {G}aussian {M}arkov Random Fields: The Stochastic Partial Differential Equation Approach},
  volume = {73},
  ISSN = {1467-9868},
  number = {4},
  journal = {J. R. Stat. Soc. Ser. B Methodol.},
  author = {Lindgren,  Finn and Rue,  H\r{a}vard and Lindstr\"{o}m,  Johan},
  year = {2011},
  month = aug,
  pages = {423-498}
}

@Book{MP11,
  title      = {Random Fields on the Sphere. Representation, Limit Theorems and Cosmological Applications},
  publisher  = {Cambridge University Press, Cambridge},
  year       = {2011},
  author     = {Marinucci, Domenico and Peccati, Giovanni},
  volume     = {389},
  series     = {London Mathematical Society Lecture Note Series},
  isbn       = {978-0-521-17561-6},
  doi        = {10.1017/CBO9780511751677},
  mrclass    = {60G60 (60B15 60D05 60H05 62M15 85A40)},
  mrnumber   = {2840154},
  mrreviewer = {Anatoliy\ Malyarenko},
  pages      = {xii+341},
  url        = {https://doi.org/10.1017/CBO9780511751677},
}

@Book{Shubin2001,
  title     = {Pseudodifferential Operators and Spectral Theory},
  publisher = {Springer, Berlin},
  year      = {2001},
  author    = {Shubin, M. A.},
  isbn      = {3-540-41195-X},
  doi       = {10.1007/978-3-642-56579-3},
  mrclass   = {47G30 (35Sxx 58J40)},
  mrnumber  = {1852334},
  pages     = {xii+288},
  url       = {https://doi.org/10.1007/978-3-642-56579-3},
}

@Book{Strang2008-pr,
  title     = {An Analysis of the Finite Element Method},
  publisher = {Wellesley-Cambridge Press, Wellesley, MA},
  year      = {2008},
  author    = {Strang, Gilbert and Fix, George},
  isbn      = {978-0-9802327-0-7; 0-9802327-0-8},
  mrclass   = {65-02 (65M60 65N30 74S05 76M10)},
  mrnumber  = {2743037},
  pages     = {x+402},
}

@Article{strichartz1983analysis,
	author    = {Strichartz, Robert S.},
	title     = {Analysis of the {L}aplacian on the complete {R}iemannian manifold},
	journal   = {J. Funct. Anal.},
	year      = {1983},
	volume    = {52},
	number    = {1},
	pages     = {48--79},
	publisher = {Academic Press},
}

@Article{Triebel1985,
  author    = {Triebel, Hans},
  title     = {Spaces of {B}esov--{H}ardy--{S}obolev type on complete {R}iemannian manifolds},
  journal   = {Ark. Mat.},
  year      = {1985},
  volume    = {24},
  number    = {1-2},
  pages     = {299-337},
  month     = dec,
  issn      = {0004-2080},
  doi       = {10.1007/bf02384402},
  publisher = {International Press of Boston},
  url       = {http://dx.doi.org/10.1007/BF02384402},
}

@Article{W63,
	author    = {Whittle, Peter},
	title     = {Stochastic processes in several dimensions},
	journal   = {Bull. Inst. Int. Stat.},
	year      = {1963},
	volume    = {40},
	pages     = {974--994},
	issn      = {0373-0441},
	fjournal  = {{Bulletin de l'Institut International de Statistique}},
	language  = {English},
	publisher = {International Statistical Institute (ISI), The Hague},
}

@Book{Yagi2010,
  title      = {Abstract Parabolic Evolution Equations and Their Applications},
  publisher  = {Springer, Berlin},
  year       = {2010},
  author     = {Yagi, Atsushi},
  series     = {Springer Monographs in Mathematics},
  isbn       = {978-3-642-04630-8},
  doi        = {10.1007/978-3-642-04631-5},
  mrclass    = {35-02 (34G10 35K90 35R20 47D06 47J35)},
  mrnumber   = {2573296},
  mrreviewer = {Sergey\ G.\ Pyatkov},
  pages      = {xviii+581},
  url        = {https://doi.org/10.1007/978-3-642-04631-5},
}

@Article{LBR22,
  author  = {Lindgren, Finn and Bolin, David and Rue, H\aa{}vard},
  title   = {The {SPDE} approach for {G}aussian and non-{G}aussian fields: 10 years and still running},
  journal = {Spat. Stat.},
  year    = {2022},
  volume  = {50},
  pages   = {100599},
  doi     = {10.1016/j.spasta.2022.100599},
}

@Article{HKS20,
  author   = {Herrmann, Lukas and Kirchner, Kristin and Schwab, Christoph},
  title    = {Multilevel approximation of {G}aussian random fields: Fast simulation},
  journal  = {Math. Models Methods Appl. Sci.},
  year     = {2020},
  volume   = {30},
  number   = {01},
  pages    = {181--223},
  doi      = {10.1142/S0218202520500050},
}

\appendix

\section{Deterministic proofs}
\label{app:detproofs}

\subsection{Proof of \Cref{prop:spectral}}

\begin{proof}
	A standard result in the spectral theory of elliptic operators on compact Riemannian manifolds (see e.g. \cite[Section 8]{Shubin2001}) ensures that there exists a set of eigenpairs $\lbrace (\lambda_i, f_i)\rbrace_{i\in\N}$ of $\cL$
	such that  $0\le\lambda_1 \le \lambda_2 \le \cdots $ with $\lambda_i \rightarrow +\infty$ as $i\rightarrow +\infty$, and  $\lbrace f_i\rbrace_{i\in\N}$ is an orthonormal basis of $L^2(\cM)$ composed of possibly complex-valued functions. 
	
	Hence, let us prove that $\lambda_{1}\ge \delta$ and that we can build an orthonormal basis $\lbrace e_i\rbrace_{i\in\N}$ of~$L^2(\cM)$ such that each $e_i$ is real-valued and an eigenfunction of $\cL$ with eigenvalue $\lambda_{i}$.
	On the one hand, by definition of $\cL$, $\lambda_{1}$ and $f_1$, we obtain
	\begin{equation*}
		\lambda_{1}=\lambda_{1}(f_1, f_1)_{L^2(\cM)}=\mathsf{A}_{\cM}(f_1,f_1) \ge \delta \Vert f_1\Vert_{H^1(\cM)}^2 \ge \delta \Vert f_1\Vert_{L^2(\cM)}^2 = \delta,
	\end{equation*}
	where for the last two inequalities we used the coercivity of $\mathsf{A}_{\cM}$ and the definition of the $H^1$-norm.
	
	On the other hand, let $\lambda >0$ be one of the eigenvalues of $\cL$, and $E_\lambda \subset L^2(\cM)$ the associated eigenspace. Following again the results from \cite[Section 8]{Shubin2001}, we get that $E_\lambda \subset C^\infty(\cM)$, $\dim E_{\lambda} < \infty$ and that if $\lambda' \neq \lambda$ is another eigenvalue of $\cL$, then $E_\lambda$ and $E_{\lambda'}$ are orthogonal. Besides,  $E_\lambda$ is in fact generated by the set $\lbrace f_j\rbrace_{j\in J_\lambda}$, where $J_\lambda =\lbrace i \in \N : \lambda_{i} = \lambda\rbrace$ is finite (since $\dim E_{\lambda} < \infty$).
	
	Take then $u\in E_\lambda$. Hence, for any $v\in H^1(\cM)$, $\mathsf{A}_{\cM}(u,v) = \lambda (u,v)_{L^2(\cM)}$. But also, by definition of $\mathsf{A}_{\cM}$,
	\begin{equation*}
		\mathsf{A}_{\cM}(\cj{u},v) = \mathsf{A}_{\cM}(\cj{v},u)=\cj{\mathsf{A}_{\cM}(u,\cj{v})}
		=\cj{\lambda (u,\cj{v})_{L^2(\cM)}}=\lambda (\cj{u},v)_{L^2(\cM)},
	\end{equation*}
	where used the fact that $\mathcal{D}$ is a real symmetric matrix and $V$ is real-valued for the first two equalities. Consequently, we also have $\cj{u}\in E_\lambda$, and so, the real-valued functions $\Re(u)=(u+\cj{u}))/2$ and $\Im(u)=(u-\cj{u}))/2i$ (corresponding to real and imaginary parts of $u$) are also in $E_\lambda$. 
	
	Circling back to the orthonormal basis $\lbrace f_j\rbrace_{j\in J_\lambda}$ of $E_\lambda$, we consider the set of real-valued functions $F_\lambda = \lbrace \Re(f_j)\rbrace_{j\in J_\lambda}\cup \lbrace \Im(f_j)\rbrace_{j\in J_\lambda} \subset E_\lambda$, and the subspace $V_\lambda \subset E_\lambda$ generated by $F_\lambda$. In particular $\dim V_\lambda \le \dim E_\lambda$. By applying the Gram--Schmidt orthogonalization process to $F_\lambda$, we get an orthonormal basis $\lbrace e_k\rbrace_{1\le k\le \dim V_\lambda}$ of $V_\lambda$ which by construction is composed of real-valued functions (since $F_\lambda$ is composed of real-valued functions). Let us show that $\lbrace e_k\rbrace_{1\le k\le \dim V_\lambda}$ is in fact a basis of $E_\lambda$, or equivalently that $\dim V_\lambda = \dim E_\lambda$.
	
	We proceed by contradiction. Assume that $\dim V_\lambda < \dim E_\lambda$. This means in particular that the orthogonal complement of $V_\lambda $ in $E_\lambda$, denoted by $V_\lambda^\perp$, is not reduced to $0$. Let then $0\neq w \in V_\lambda^\perp$. By linearity, we have, for any $j\in J_\lambda$, $(f_j,w)_{L^2(\cM)}=(\Re(f_j),w)_{L^2(\cM)} + i(\Im(f_j),w)_{L^2(\cM)} = 0$, since $\Re(f_j),\Im(f_j) \in F_\lambda \subset V_\lambda$. Hence, since $w\in E_\lambda$ and $\lbrace f_j\rbrace_{j\in J_\lambda}$ is an orthonormal basis of $E_\lambda$, it must hold that $w=0$, which contradicts our initial claim. Consequently, $\dim V_\lambda = \dim E_\lambda$, and therefore $\lbrace e_k\rbrace_{1\le k\le \dim V_\lambda}$ is an orthonormal basis of $E_\lambda$.
	
	Finally, by repeating the construction above to each eigenspace $E_\lambda$ associated with distinct eigenvalues, and concatenating the obtained bases, we obtain an orthonormal basis $L^2(\cM)$ (due to the fact that these eigenspaces are orthogonal to one another and span $L^2(\cM)$). Each element in this basis is an eigenfunction of $\cL$ since it is built from a given eigenspace, and is a real-valued function. This concludes our proof.
\end{proof}

\subsection{A useful inequality}
We end this section by introducing a lemma which will be used to obtain practical error bounds.

\begin{lemma}\label{lem:min}
	For any $a,b\in\R$ and any $\varepsilon>0$, $\min\{a-\varepsilon; b\} + \varepsilon \ge \min\{a; b\} $. 
\end{lemma}
\begin{proof}
	Indeed, if $a<b$, we have $\min\{a-\varepsilon; b\} + \varepsilon = a-\varepsilon+\varepsilon=a = \min\{a; b\}$. If $b\le a <b+\varepsilon$,  $\min\{a-\varepsilon; b\} + \varepsilon = a-\varepsilon+\varepsilon=a \ge b=\min\{a; b\}$. And finally if $a\ge b+\varepsilon$, $\min\{a-\varepsilon; b\} + \varepsilon = b+\varepsilon>  b=\min\{a; b\}$. 
\end{proof}

\section{Error estimates}\label{app:error_discr_op}

\subsection{Geometric consistency estimate}\label{app_sub:geom_consist}

The following geometric consistency estimate quantifies the error between the bilinear forms $\mathsf{A}_{\cM}$ and $\mathsf{A}_{\cM_h}$.
Its proof is an adaptation of the proof of \cite[Lemma 4.7]{Dziuk2013} to account for the diffusion matrix $\cD$.
\begin{lemma}
	\label{lem:geometric-consistency-A}
	There is a constant $C>0$ such that for all $h < h_0$ and $u_h,v_h \in S_h^\ell$, 
	\begin{align}
		\label{eq:perterror}
		\left|\mathsf{A}_\cM(u_h^\ell,v_h^\ell)-\mathsf{A}_{\cM_h}(u_h,v_h)\right| \leq C h^2  \|u_h^\ell\|_{H^1(\cM)}\|v_h^{\ell}\|_{H^1(\cM)}.
	\end{align}
\end{lemma}

\begin{proof}
	We first note that, by definition of $\cD$, for any $x_0\in\cM$, and any $w,w'\in T_{x_0}\cM$, $(\cD(x_0) w)\cdot \cj{w'}$ defines an inner product on $T_{x_0}\cM$.
	We denote by $\|\cdot\|$ the usual Euclidean norm of vectors of $T_{x_0}\cM\subset \C^{d+1}$ and by $\|\cdot\|_{\cD(x_0)}$ the norm defined by  $\|w\|_{\cD(x_0)}^2 = (\cD(x_0) w) \cdot \cj{w}$, $w\in T_{x_0}\cM$.
	
	Let $\Pi=I-\nu\nu^T$ (resp. $\Pi_h = I - \nu_h\nu_h^T$) \revision{be} the orthogonal projection onto the tangent planes of $\cM$ (resp. $\cM_h$), and let $\cH : \cM \rightarrow \R^{(d+1)\times (d+1)}$ be the extended Weingarten map of $\cM$ (cf. \cite[Definition 2.5]{Dziuk2013}). Recall in particular that $\cH(x) \nu(x) =0$ for any $x\in \cM$, meaning in particular that $\cH \Pi = \Pi\cH = \cH$. Finally, we introduce the map $\cQ_h : \cM \rightarrow \R^{(d+1)\times (d+1)}$ defined as
	\begin{equation*}
		\cQ_h = \frac{1}{\sigma^\ell} \Pi(I-d_s^\ell\cH)\Pi_h^\ell \cD \Pi_h^\ell (I-d_s^\ell\cH)\Pi,
	\end{equation*}
	where $d_s$ is the oriented distance function restricted to~$\cM_h$ and introduced in \Cref{sec:sfem}. On the one hand, note that for any $u_h, v_h \in S_h$, 
	\begin{align*}
		&(\cD^{-\ell} \nabla_{\cM_h} {u_h}) {\cdot} \nabla_{\cM_h} {\cj{v}_h}\\
		&\qquad = \big(\cD^{-\ell} \Pi_h(I-d_s\cH^{-\ell})\Pi^{-\ell}(\nabla_{\cM} u_h^\ell)^{-\ell}\big) {\cdot} \big( \Pi_h(I-d_s\cH^{-\ell})\Pi^{-\ell}(\nabla_{\cM} \cj{v}_h^\ell)^{-\ell}\big)\\
		&\qquad= \sigma \big(\cQ_h^{-\ell}(\nabla_{\cM} u_h^\ell)^{-\ell}\big)\cdot(\nabla_{\cM} \cj{v}_h^\ell)^{-\ell},
	\end{align*}
	which gives, after integrating both sides over $\cM_h$ and using \Cref{eq:sigma}, 
	\begin{equation}\label{eq:intmh}
		\int_{\cM_h} (\cD^{-\ell} \nabla_{\cM_h} u_h)\cdot (\nabla_{\cM_h} \cj{v}_h) \dd A_h =\int_{\cM} \big(\cQ_h\nabla_{\cM} u_h^\ell\big)\cdot(\nabla_{\cM} \cj{v}_h^\ell) \dd A .
	\end{equation}
	Let then $\mathsf{A}_{\cM_h}^\ell\colon S_h^\ell \times S_h^\ell \rightarrow \R$  be the Hermitian form defined for any $u_h^\ell,v_h^\ell \in S_h^\ell$ by
	\begin{align} 
		\label{eq:biformGh_lifted}
		\begin{split}
			&\mathsf{A}_{\cM_h}^\ell(u_h^\ell,v_h^\ell)
			=\int_{\cM} \big(\cQ_h\nabla_{\cM} u_h^\ell\big)\cdot(\nabla_{\cM} \cj{v}_h^\ell)  \dd A  
			+ \int_{\cM} \big(\sigma^{\ell}\big)^{-1}V  u_h^\ell \cj{v}_h^\ell  \dd A.
		\end{split}
	\end{align}
	Note that following \Cref{eq:intmh,eq:sigma}, $\mathsf{A}_{\cM_h}^\ell$ satisfies for any $u_h,v_h \in S_h$ the equality \\ $\mathsf{A}_{\cM_h}^\ell(u_h^\ell,v_h^\ell)=\mathsf{A}_{\cM_h}(u_h,v_h)$.
	Therefore, for any $u_h,v_h \in S_h$, we bound
	\begin{align}\label{eq:eqAmAml_2}
		\begin{split}
			\big|\mathsf{A}_\cM(u_h^\ell,v_h^\ell)&-\mathsf{A}_{\cM_h}(u_h,v_h)\big|=\big| \mathsf{A}_{\cM}(u_h^\ell,v_h^\ell)-  \mathsf{A}_{\cM_h}^\ell (u_h^\ell,v_h^\ell)\big| \\
			& \leq \left|\int_\cM ((\cQ_h -\cD) \nabla_\cM u^\ell_h)\cdot (\nabla_\cM \cj{v}_h^\ell)~\dd A\right| 
			+ \left| \int_\cM (1-\big(\sigma^{\ell}\big)^{-1})V u_h^\ell\cj{v}_h^\ell~\dd A\right|.
		\end{split}
	\end{align}
	
	We now bound these two terms. Recall that \cite[Lemma 4.1]{Dziuk2013} shows
	\begin{equation}\label{eq:ineq_sigma}
		\Vert \sigma \Vert_{L^{\infty}(\cM_{h})} \lesssim 1, \quad
		\Vert \sigma^{-1} \Vert_{L^{\infty}(\cM_{h})} \lesssim 1, \quad
		\Vert \sigma -1 \Vert_{L^{\infty}(\cM_{h})} \lesssim h^2, \quad \Vert \sigma^{-1} -1 \Vert_{L^{\infty}(\cM_{h})} \lesssim h^2.
	\end{equation}
	Hence, since $V$ takes positive values,
	\begin{align*}
		\big| \int_\cM (1-\big(\sigma^{\ell}\big)^{-1})V u_h^\ell\cj{v}_h^\ell~\dd A\big|
		&\le  \int_\cM \vert 1-\big(\sigma^{\ell}\big)^{-1}\vert~V~ \vert u_h^\ell\vert ~\vert {v_h}^\ell\vert \dd A \\
		&\le   \Vert 1-\sigma^{-1}\Vert_{L^{\infty}(\cM_h)}  \int_\cM V~ \vert u_h^\ell\vert ~\vert {v_h}^\ell\vert \dd A,
	\end{align*}
	which in turn gives (using the Cauchy--Schwartz inequality and \Cref{eq:ineq_sigma}), 
	\begin{align}\label{eq:interma}
		\big| \int_\cM (1-\big(\sigma^{\ell}\big)^{-1})V u_h^\ell\cj{v}_h^\ell~\dd A\big|
		&\lesssim h^2 \bigg(\int_\cM V \vert u_h^\ell\vert^2 ~\dd A \bigg)^{1/2} \bigg(\int_\cM V~\vert {v_h}^\ell\vert^2  ~\dd A\bigg)^{1/2}.
	\end{align}
	
	To bound the other term, we first introduce for any $ B \in \R^{(d+1)\times(d+1)}$ the notation $\Vert B\Vert=\sup_{\Vert x\Vert =1} \Vert B x\Vert$.
	Then we have
	\begin{align*}
		\Vert \cQ_h - \cD\Vert
		&=\big\Vert (\sigma^\ell)^{-1} \big(\Pi(I-d_s^\ell\cH)\Pi_h^\ell \cD  \Pi_h^\ell (I-d_s^\ell\cH)\Pi - \cD\big) +\big((\sigma^\ell)^{-1}-1\big)\cD\Vert\\
		&\le \big\Vert (\sigma^\ell)^{-1}\big\Vert_{L^{\infty}(\cM)}~\big\Vert \Pi(I-d_s^\ell\cH)\Pi_h^\ell \cD  \Pi_h^\ell (I-d_s^\ell\cH)\Pi - \cD\big\Vert \\ &\quad\quad+\big\Vert(\sigma^\ell)^{-1}-1\big\Vert_{L^{\infty}(\cM)}~\big\Vert\cD\big\Vert.
	\end{align*}
	By \Cref{eq:ineq_sigma} and since $\cD$ has bounded eigenvalues over $\cM$, we obtain
	\begin{align}\label{eq:interm1}
		\Vert \cQ_h - \cD\Vert
		&\lesssim \big\Vert \Pi(I-d_s^\ell\cH)\Pi_h^\ell \cD  \Pi_h^\ell (I-d_s^\ell\cH)\Pi - \cD\big\Vert + h^2,
	\end{align}
	where the constant in the inequality is independent of the location on~$\cM$.
	We split the first term on the right into 
	\begin{align*}
		\big\Vert \Pi(I-d_s^\ell\cH)\Pi_h^\ell &\cD  \Pi_h^\ell (I-d_s^\ell\cH)\Pi - \cD\big\Vert \\
		&= \big\Vert \Pi \Pi_h^\ell \cD  \Pi_h^\ell \Pi - \cD - \Pi \Pi_h^\ell \cD  \Pi_h^\ell d_s^\ell\cH \Pi
		-d_s^\ell\cH \Pi_h^\ell \cD  \Pi_h^\ell (I-d_s^\ell\cH)\Pi
		\big\Vert \\
		&\le \big\Vert \Pi \Pi_h^\ell \cD  \Pi_h^\ell \Pi - \cD \big\Vert + \big\Vert \Pi \Pi_h^\ell \cD  \Pi_h^\ell d_s^\ell\cH \Pi \big\Vert +
		\big\Vert d_s^\ell\cH \Pi_h^\ell \cD  \Pi_h^\ell (I-d_s^\ell\cH)\Pi
		\big\Vert.
	\end{align*}
	Since $\Vert d_s \Vert_{L^{\infty}(\cM_h)} \lesssim h^2$ by \cite[Lemma 4.1]{Dziuk2013} and $\cH$ is defined independently of~$h$, we conclude that 
	\begin{align}\label{eq:interm2}
		\big\Vert \Pi(I-d_s^\ell\cH)\Pi_h^\ell \cD  \Pi_h^\ell (I-d_s^\ell\cH)\Pi - \cD\big\Vert 
		&\lesssim \big\Vert \Pi\Pi_h^\ell \cD  \Pi_h^\ell \Pi - \cD \big\Vert + h^2.
	\end{align}
	We notice that $\cD = \Pi\cD \Pi$, since by definition of $\cD$, $\cD\nu =0$, which implies 
	\begin{align*}
		\big\Vert \Pi\Pi_h^\ell \cD  \Pi_h^\ell \Pi - \cD \big\Vert
		&= \big\Vert \Pi\Pi_h^\ell \Pi\cD \Pi  \Pi_h^\ell \Pi - \Pi\cD \Pi \big\Vert
		= \big\Vert (\Pi\Pi_h^\ell \Pi-\Pi)\cD \Pi  \Pi_h^\ell \Pi + \Pi\cD(\Pi\Pi_h^\ell \Pi{-}\Pi) \big\Vert \\
		&\le  \big\Vert \Pi\Pi_h^\ell \Pi-\Pi\big\Vert \big\Vert\cD \Pi  \Pi_h^\ell \Pi\big\Vert + \big\Vert \Pi\cD \big\Vert \big\Vert \Pi\Pi_h^\ell \Pi-\Pi \big\Vert.
	\end{align*}
	Using that $\big\Vert \Pi\Pi_h^\ell \Pi-\Pi\big\Vert \lesssim h^2$ by the proof of \cite[Lemma~4.1]{Dziuk2013}, we deduce that $\big\Vert \Pi\Pi_h^\ell \cD  \Pi_h^\ell \Pi - \cD \big\Vert \lesssim h^2$. 
	Injecting this inequality into \Cref{eq:interm2}, and the resulting inequality into \Cref{eq:interm1}, we conclude that
	\begin{equation*}
		\Vert \cQ_h - \cD\Vert \lesssim h^2.
	\end{equation*}
	This allows us to write
	\begin{align*}
		&\left|\int_\cM ((\cQ_h -\cD) \nabla_\cM u^\ell_h)\cdot (\nabla_\cM \cj{v}_h^\ell) \dd A\right| 
		\le  \int_\cM 	\Vert(\cQ_h -\cD) \nabla_\cM u^\ell_h\Vert ~\Vert \nabla_\cM {v_h}^\ell\Vert \dd A \\
		&\qquad \lesssim  \int_\cM 	h^2\Vert \nabla_\cM u^\ell_h\Vert ~\Vert \nabla_\cM {v_h}^\ell\Vert \dd A
		\le h^2\int_\cM 	(\mu_{\min})^{-1} \Vert \nabla_\cM u^\ell_h\Vert_{\cD} ~\Vert \nabla_\cM {v_h}^\ell\Vert_{\cD} \dd A,
	\end{align*}
	where $\mu_{\min} : \cM \rightarrow \R_+$ maps any $x\in\cM$ to the smallest eigenvalue of $\cD(x)$ associated with an eigenvector in $\nu^\perp$. This last inequality is a consequence of the fact that by construction $\nabla_\cM u^\ell_h, \nabla_\cM v^\ell_h \in \nu^\perp$ and using the characterization of eigenvalues through Rayleigh quotients. Since the non-zero eigenvalues of $\cD$ are uniformly bounded above and below by positive constants, we conclude that 
	\begin{align*}
		\big|\int_\cM ((\cQ_h -\cD) \nabla_\cM u^\ell_h)\cdot (\nabla_\cM \cj{v}_h^\ell)~\dd A\big| 
		&\lesssim   h^2\int_\cM  \Vert \nabla_\cM u^\ell_h\Vert_{\cD} ~\Vert \nabla_\cM {v_h}^\ell\Vert_{\cD} \dd A.
	\end{align*}
	Then, using the Cauchy--Schwartz inequality yields
	\begin{align}\label{eq:intermb}
		\begin{split}
			&\left|\int_\cM ((\cQ_h -\cD)\nabla_\cM u^\ell_h)\cdot (\nabla_\cM \cj{v}_h^\ell)~\dd A\right| \\
			&\qquad \lesssim  
			h^2\bigg(\int_\cM 	\Vert \nabla_\cM u^\ell_h\Vert_{\cD}^2 ~\dd A\bigg)^{1/2}
			\bigg(\int_\cM 	\Vert \nabla_\cM v^\ell_h\Vert_{\cD}^2 ~\dd A\bigg)^{1/2}.
		\end{split}
	\end{align}
	Inserting the derived bounds \Cref{eq:interma} and \Cref{eq:intermb} into \Cref{eq:eqAmAml_2}, we derive
	\begin{align*}
		\big|\mathsf{A}_\cM(u_h^\ell,v_h^\ell)-\mathsf{A}_{\cM_h}(u_h,v_h)\big| 
		\lesssim~ &h^2\bigg(\int_\cM 	\Vert \nabla_\cM u^\ell_h\Vert_{\cD}^2 ~\dd A\bigg)^{1/2}
		\bigg(\int_\cM 	\Vert \nabla_\cM v^\ell_h\Vert_{\cD}^2 ~\dd A\bigg)^{1/2} \\
		&+ h^2 \bigg(\int_\cM V \vert u_h^\ell\vert^2 ~\dd A \bigg)^{1/2} \bigg(\int_\cM V~\vert {v_h}^\ell\vert^2  ~\dd A\bigg)^{1/2}.
	\end{align*}
	Note that for any $u \in H^1(\cM)$, 
	\begin{equation*}
		\mathsf{A}_{\cM}(u,u) \geq \int_\cM V \vert u\vert^2 \dd A , \quad \text{and}\quad \mathsf{A}_{\cM}(u,u) \ge \int_\cM  \cD\nabla_{\cM} u \cdot \nabla_{\cM}\cj{ u} \dd A =\int_\cM  \|\nabla_\cM u \|_{\cD}^2 \dd A, 
	\end{equation*}
	so we obtain
	\begin{align*}
		\left| \mathsf{A}_{\cM}(u_h^\ell,v_h^\ell)-  \mathsf{A}_{\cM_h}^\ell (u_h^\ell,v_h^\ell)\right| \lesssim h^2\sqrt{\mathsf{A}_{\cM}(u_h^\ell,u_h^\ell)}\sqrt{\mathsf{A}_{\cM}(v_h^\ell,v_h^\ell)}.
	\end{align*}
	Finally, due to  \Cref{eq:cont},
	\begin{align*}
		\sqrt{\mathsf{A}_{\cM}(u_h^\ell,u_h^\ell)}\sqrt{\mathsf{A}_{\cM}(v_h^\ell,v_h^\ell)} \lesssim \|u_h^\ell\|_{H^1(\cM)}\|v_h^\ell\|_{H^1(\cM)},
	\end{align*}
	and the result follows. 
\end{proof}

\subsection{Norm estimates}

We start by introducing a few norm estimates involving the resolvents of $\cL$ and $\cL_h$.
 \revision{These results can be seen as extensions} of the results introduced in \cite[Lemmma 6.3]{bonito2016}, but adapted to the contour $\Gamma$ used to define functions of operators in this paper. We also recall a finite element error estimate which derives from a result in   \cite[Lemma 6.1]{bonito2016}.

\begin{lemma}\label{lem:sol}
	For any $z\in\Gamma$, $s\in [0,1]$,  $v\in L^2(\cM)$, and $v_h\in S_h^\ell$, it holds
	\begin{align}
		\|\cL^{s}(z-\cL)^{-1}v\|_{L^{2}(\cM)} &\lesssim |z|^{-(1-s)} \| v\|_{L^2(\cM)},\label{eq:Ls} \\
		\|\cL_h^{s}(z-\cL_h)^{-1}v_h\|_{L^{2}(\cM)} &\lesssim |z|^{-(1-s)} \| v_h\|_{L^2(\cM)}.\label{eq:Lhs}
	\end{align}
	Besides, for any $\beta\in[0,1]$ and for any $\varphi\in L^{2}(\cM)$, 
	\begin{equation}\label{eq:err_fem}
		\begin{aligned}
			\Vert \cL_h^{(1-\beta)/2}P_h(\cL^{-1}-\cL_h^{-1}P_h)\varphi\Vert_{L^2(\cM)}
			\lesssim h^{2\beta} \Vert \cL^{-(1-\beta)/2}\varphi\Vert_{L^2(\cM)}
		\end{aligned}
	\end{equation}
\end{lemma}

\begin{proof}
	
	We start with the proof of \Cref{eq:Ls}. 
	To this end, let $z \in \Gamma$ and $s \in [0,1]$. Let then $v\in L^2(\cM)$ and $w=(z-\cL)^{-1}v\in\cD(\cL)$, where $\cD(\cdot)$ denotes the domain of an operator. Since $\cL$ is self-adjoint, we have $\cD(\cL^s)=[L^2(\cM),\cD(\cL)]_{s}$ with isometry, where $[L^2(\cM),\cD(\cL)]_{s}$  the intermediate space between $L^2(\cM)$ and $\cD(\cL)$ obtained by the complex interpolation method \cite[Theorem 16.1]{Yagi2010}. In particular, we have 
	\begin{equation}\label{eq:Ls_temp}
		\Vert \cL^{s} w\Vert_{L^2(\cM)} 
		= \Vert w\Vert_{[L^2(\cM),\cD(\cL)]_{s}}
		\le \Vert \cL w\Vert_{L^2(\cM)}^s \Vert w\Vert_{L^2(\cM)}^{1-s}
	\end{equation}
	where the inequality derives from a classical result on interpolation spaces (see e.g., \cite[Section 5.1]{Yagi2010}).
	On the one hand, we obtain by \Cref{eq:A4} that 
	\begin{align*}
		\|w\|_{L^2(\cM)}=\|(z-\cL)^{-1}v\|_{L^2(\cM)} \lesssim |z|^{-1} \|v\|_{L^2(\cM)}. 
	\end{align*}
On the other hand, adding and subtracting $z(z-\cL)^{-1}v$ yields
\begin{align*}
	\|\cL w\|_{L^2(\cM)} 
	& 
	=\|z(z{-}\cL)^{-1}v- v\|_{L^2(\cM)} 
	\leq |z|\|(z{-}\cL)^{-1}v\|_{L^2(\cM)}+ \| v\|_{L^2(\cM)} 
	\lesssim 2\|v\|_{L^2(\cM)},
\end{align*}
where we once again use \Cref{eq:A4} to derive the last inequality.
Combining these two estimates with \Cref{eq:Ls_temp}, we get 
\begin{align*}
	\|\cL^{s} (z{-}\cL)^{-1}v\|_{L^2(\cM)}{=}\|\cL^{s} w\|_{L^2(\cM)}
	\lesssim \|v\|^s_{L^2(\cM)} |z|^{-(1-s)}\|v\|_{L^2(\cM)}^{1-s } 
	= |z|^{-(1-s)} \|v\|_{L^2(\cM)},
\end{align*}
hence proving \Cref{eq:Ls}. As for \Cref{eq:Lhs}, it is proven using the same steps as \Cref{eq:Ls} but substituting $\cL$ by $\cL_h$, \Cref{eq:A4} by \Cref{eq:A5}, and $\cD(\cL)$ by $S_h^\ell$.  

Finally, let us prove \Cref{eq:err_fem}. Let $\beta\in(0,1]$ and $\varphi\in L^{2}(\cM)$. Then, we have
\begin{align*}
	\Vert \cL_h^{(1-\beta)/2}P_h(\cL^{-1}-\cL_h^{-1}P_h)\varphi\Vert_{L^2(\cM)}
	&\lesssim \Vert \cL^{(1-\beta)/2}P_h(\cL^{-1}-\cL_h^{-1}P_h)\varphi\Vert_{L^2(\cM)}\\
	&\lesssim \Vert \cL^{(1-\beta)/2}(\cL^{-1}-\cL_h^{-1}P_h)\varphi\Vert_{L^2(\cM)}
\end{align*}
where we use \cite[Lemma 5.2]{bonito2016} to derive the first inequality, and \cite[Lemma 5.1]{bonito2016} to derive the second inequality. Finally, noting that $\cL$ satisfies elliptic regularity for indices $\alpha\in (0,1]$ (cf. \cite[Assumption 1]{bonito2016}),  we can apply  \cite[Lemma 6.1]{bonito2016} with $\alpha=\beta$ and $s=(1-\beta)/2$ to conclude that $\Vert\cL^{(1-\beta)/2}(\cL^{-1}-\cL_h^{-1}P_h)\varphi\Vert_{L^2(\cM)}\lesssim h^{2\beta}\Vert\cL^{-(1-\beta)/2}\varphi\Vert_{L^2(\cM)}$ and therefore 
\begin{align*}
	\Vert \cL_h^{(1-\beta)/2}P_h(\cL^{-1}-\cL_h^{-1}P_h)\varphi\Vert_{L^2(\cM)}
	&\lesssim h^{2\beta}\Vert\cL^{-(1-\beta)/2}\varphi\Vert_{L^2(\cM)}.
\end{align*}
This concludes the proof of \Cref{lem:sol}.
\end{proof}

We now prove some estimates for the norm of shifted inverses of the operators $\cL_h$ and~$\sL_h$, and for the error between inverses of these two operators. These results can be seen as extensions of the ones stated in \cite[Lemma A.1]{Bonito2024}.

\begin{lemma}
\label{lem:lemmaa1}
Let $v\in L^2(\cM)$, $v_h \in S_h^\ell$ and $V_h \in S_h$ be arbitrary. Then, for all $z \in \Gamma$, for any  $q \in [-1,1]$, and any $r \in (0,2)$, and any $s\in[0,1]$, 
\begin{align}
	\|\sL_h(z-\sL_h)^{-1}V_h\|_{L^2(\cM_h)} &\lesssim |z|^{-1/2} \| V_h\|_{H^1(\cM_h)}, \label{eq:A2}\\
	\|(z-\cL_h)^{-1}v_h\|_{L^2(\cM)} & \lesssim |z|^{r/2-1} \|\cL_h^{-r/2} v_h\|_{L^2(\cM)}, \label{eq:A3} \\
	\left\|(\cL_h^{-1} v_h)^{-\ell}-\sL_h^{-1}\sP_h(\sigma v_h^{-\ell})\right\|_{H^1(\cM_h)} & \lesssim h^2 \|\cL_h^{-1/2} v_h\|_{L^2(\cM)}. 	\label{eq:B}
\end{align}
\end{lemma}

\begin{proof}

To prove \Cref{eq:A2}, we first observe that the estimate in \Cref{eq:Lhs} carry over to the case when $\sL_h$ is used instead of $\cL_h$, and $L^2(\cM)$ (resp. $S_h^\ell$) is replaced by its counterpart $L^2(\cM_{h})$ (resp. $S_h$) on the polyhedral surface. Hence, for any $z\in\Gamma$, $V_h\in L^2(\cM_h)$, we have
\begin{equation*}
	\|\sL_h^{s}(z-\sL_h)^{-1}V_h\|_{L^{2}(\cM_h)} \lesssim |z|^{-(1-s)} \| V_h\|_{L^2(\cM_h)}
\end{equation*}
In particular, we retrieve (by taking $s=1/2$)
\begin{equation}
	\begin{split} 
		\|\sL_h(z-\sL_h)^{-1}V_h\|_{L^2(\cM_h)}
		& =\| \sL_h^{1/2}(z-\sL_h)^{-1}\sL_h^{1/2}V_h\|_{L^2(\cM_h)}\\
		& \lesssim |z|^{-1/2} \|\sL_h^{1/2}V_h\|_{L^2(\cM_h)}\lesssim  |z|^{-1/2} \|V_h\|_{H^1(\cM_h)}\label{eq:E'3},
	\end{split}
\end{equation}
where we used the equivalence of norms \eqref{eq:norm_equiv} in the last inequality.

To bound \Cref{eq:A3}, we apply \Cref{eq:Lhs} with $s = r/2$ to obtain 
\begin{align*}
	\|(z-\cL_h)^{-1}v_h\|_{L^2(\cM)} = 	\|\cL_h^{r/2}(z-\cL_h)^{-1}\cL_h^{-r/2}v_h\|_{L^2(\cM)} \lesssim |z|^{-(1-r/2)}\|\cL_h^{-r/2}v_h\|_{L^2(\cM)}. 
\end{align*}

Finally, to prove the bound in \Cref{eq:B}, we rely on the geometric consistency estimate of \Cref{lem:geometric-consistency-A}. 
Let $u_h = \cL_h^{-1} v_h$ and let $U_h = \sL_h^{-1}\sP_h(\sigma v_h^{-\ell})$. 
Note that by definition of $\cL_h^{-1}$,
\begin{align*}
	(\cL_h u_h,w_h)_{L^2(\cM)} = \mathsf{A}_\cM(u_h,w_h) = 	(v_h,w_h)_{L^2(\cM)},
\end{align*}
for all  $w_h \in S_h^\ell$. 
Likewise, for $\sL_h^{-1}$ we obtain 
\begin{align*}
	(\sL_h U_h,W_h)_{L^2(\cM_h)} 
	&= \mathsf{A}_{\cM_h}(U_h,W_h)
	=	(\sP_h(\sigma v_h^{-\ell}),W_h)_{L^2(\cM_h)}\\
	&= 	(\sigma v_h^{-\ell},W_h)_{L^2(\cM_h)} = (v_h,W_h^{\ell})_{L^2(\cM)}, 
\end{align*}
for all $W_h \in S_h$, where we used the definition of $\sigma$ in the last step.
Let us now select a fixed, but arbitrary, $\Xi_h \in S_h$.
Then, by combining the last two equations,
\begin{align*}
	|\mathsf{A}_{\cM_h}(u_h^{-\ell}-U_h,\Xi_h)| &=|\mathsf{A}_{\cM_h}(u_h^{-\ell},\Xi_h)-\mathsf{A}_{\cM_h}(U_h,\Xi_h)|\\
	&=|\mathsf{A}_{\cM_h}(u_h^{-\ell},\Xi_h)-(v_h,\Xi_h^{\ell})_{L^2(\cM)}|
	=|\mathsf{A}_{\cM_h}(u_h^{-\ell},\Xi_h)-\mathsf{A}_\cM(u_h,\Xi_h^{\ell})|,
\end{align*}
meaning that an application of \Cref{lem:geometric-consistency-A} results in the bound 
\begin{align}
	|\mathsf{A}_{\cM_h}(u_h^{-\ell}-U_h,\Xi_h)|  \lesssim h^2 \|u_h\|_{H^1(\cM)} \|\Xi_h^\ell\|_{H^1(\cM)}. \label{eq:C1} 
\end{align}

Further, note that for any $\Xi_h^\ell \in S_h^\ell$, the equivalence of norms~\eqref{eq:norm_equiv} gives
\begin{align}\label{eq:equiv_m}
	\|\Xi_h^\ell\|_{H^1(\cM)}^2 \sim \|\cL_h^{1/2} \Xi_h^\ell\|_{L^2(\cM)}^2 
	= ({\cL_h} \Xi_h^\ell,\Xi_h^\ell)_{L^2(\cM)}
	=\mathsf{A}_{\cM}(\Xi_h^\ell,\Xi_h^\ell),
\end{align}
where the last equality comes from the definition of $\cL_h$.  And similarly, for any $\Xi_h \in S_h$, we have
\begin{align}\label{eq:equiv_mh}
	\|\Xi_h\|_{H^1(\cM_h)}^2 \sim
	\mathsf{A}_{\cM_h}(\Xi_h,\Xi_h).
\end{align}
Then, applying the triangle inequality to the (last) right-hand side of \Cref{eq:equiv_m} gives
\begin{align*}
	\|\Xi_h^\ell\|_{H^1(\cM)}^2 &\lesssim \mathsf{A}_{\cM_h}(\Xi_h,\Xi_h)
	+
	\big\vert  \mathsf{A}_{\cM}(\Xi_h^\ell,\Xi_h^\ell)-\mathsf{A}_{\cM_h}(\Xi_h,\Xi_h)\big\vert \\
	&\lesssim \|\Xi_h\|_{H^1(\cM_h)}^2 + h^2 \|\Xi_h^\ell\|_{H^1(\cM)}^2,
\end{align*}
where we used \Cref{eq:equiv_mh} and \Cref{lem:geometric-consistency-A} to derive the second inequality.
This means in particular that there exists $C>0$ independent of $h$ such that $ \|\Xi_h^\ell\|_{H^1(\cM)}^2 \le C(\|\Xi_h\|_{H^1(\cM_h)}^2 + h^2 \|\Xi_h^\ell\|_{H^1(\cM)}^2)$. Recall that $h \in (0, h_0)$ for some $h_0 \in (0,1)$ small enough. Assuming that especially $1-Ch_0^2 >0$ yields $ \|\Xi_h^\ell\|_{H^1(\cM)}^2 \le C(1-Ch^2)^{-1}\|\Xi_h\|_{H^1(\cM_h)}^2\le C(1-Ch_0^2)^{-1}\|\Xi_h\|_{H^1(\cM_h)}^2$, which allows us to conclude that
\begin{align}\label{eq:equiv_m_mh}
	\|\Xi_h^\ell\|_{H^1(\cM)}^2 &
	\lesssim \|\Xi_h\|_{H^1(\cM_h)}^2.
\end{align}

Now, applying successively \Cref{eq:equiv_mh} and \Cref{eq:C1} with $\Xi_h=u_h^{-\ell}-U_h$, we obtain
\begin{align*}
	\|u_h^{-\ell}-U_h\|_{H^1(\cM_h)}^2 \lesssim h^2  \|u_h\|_{H^1(\cM)} \|u_h-U_h^\ell\|_{H^1(\cM)}
	\lesssim h^2  \|u_h\|_{H^1(\cM)} \|u_h^{-\ell}-U_h\|_{H^1(\cM_h)},
\end{align*}
where the last inequality is derived from applying \Cref{eq:equiv_m_mh}.
Therefore, we end up with
\begin{align*}
	\|u_h-U_h^\ell\|_{H^1(\cM)} \lesssim h^2 \|u_h\|_{H^1(\cM)}\lesssim h^2 \|\cL_h^{-1/2}v_h\|_{L^2(\cM)},
\end{align*}
where the equivalence of norms~\eqref{eq:norm_equiv} together with the definition of~$u_h$ are used in the final step. This concludes the proof of \Cref{eq:B}.
\end{proof}

Finally, we recall the Bramble-–Hilbert lemma (cf. \cite[Equation (4.8)]{Bonito2024}), which is used in several proofs in this paper.

\begin{lemma}[Bramble-–Hilbert lemma]\label{lem:bh}
For any $t\in [0,2]$ and any $\varphi \in H^t(\cM)$,
\begin{equation}\label{eq:bh}
	\Vert (I - P_h )\varphi\Vert_{L^2(\cM)} \lesssim h^t\Vert \varphi\Vert_{H^t(\cM)}
	\lesssim  h^t\Vert \cL^{t/2} \varphi\Vert_{L^2(\cM)},
\end{equation}
where we used the equivalence of Sobolev norms dot-spaces norms in the last inequality.
\end{lemma}

\subsection{Proof of \Cref{lem:lemma_error1}}\label{app:detproofs_err}

\begin{proof}
Let $h\in(0,h_0)$, $z \in \Gamma$. Let  $f \in L^2(\cM)$ and let $p\in [0,1]$ such that $\Vert \cL^pf\Vert_{L^2(\cM)}<\infty$. Finally, let $\beta\in[0,1]$ such that $p\in [0,(1+\beta)/2]$. We then define $\cF_h(z)$  by $\cF_h(z)=(z-\cL_h)^{-1}P_h -P_h(z-\cL)^{-1}$. Note then 
\begin{equation*}
	\begin{aligned}
		\cF_h(z)
		&=(z-\cL_h)^{-1}\cL_h(\cL_h^{-1}P_h(z-\cL)\cL^{-1}-\cL_h^{-1}(z-\cL_h)P_h\cL^ {-1})\cL(z-\cL)^{-1}\\
		&= (z-\cL_h)^{-1}\cL_h(P_h\cL^{-1}-\cL_h^{-1}P_h)\cL(z-\cL)^{-1}
	\end{aligned}
\end{equation*}
Hence, 
\begin{equation*}
	\begin{aligned}
		\Vert\cF_h(z)f\Vert_{L^2(\cM)}
		&= \Vert(z-\cL_h)^{-1}\cL_h(P_h\cL^{-1}-\cL_h^{-1}P_h)\cL(z-\cL)^{-1}f\Vert_{L^2(\cM)}\\
		&= \Vert \cL_h^{1-(1-\beta)/2}(z-\cL_h)^{-1}\cL_h^{(1-\beta)/2}(P_h\cL^{-1}-\cL_h^{-1}P_h)\cL(z-\cL)^{-1}f\Vert_{L^2(\cM)}\\
		&\lesssim  \vert z\vert^{-(1-\beta)/2}\Vert \cL_h^{(1-\beta)/2}(P_h\cL^{-1}-\cL_h^{-1}P_h)\cL(z-\cL)^{-1}f\Vert_{L^2(\cM)}
	\end{aligned}
\end{equation*}
where we used  \Cref{eq:Lhs} with $s=(1-\beta)/2\in[0,1/2]$ to derive the  inequality. Note then that, using \Cref{eq:err_fem}, we have
\begin{equation*}
	\begin{aligned}
		\Vert\cF_h(z)f\Vert_{L^2(\cM)}
		&\lesssim\vert z\vert^{-(1-\beta)/2} h^{2\beta}\Vert \cL^{-(1-\beta)/2}\cL(z-\cL)^{-1}f\Vert_{L^2(\cM)}
		\\
		& =h^{2\beta}\vert z\vert^{-(1-\beta)/2}\Vert \cL^{(1+\beta)/2}(z-\cL)^{-1}f\Vert_{L^2(\cM)}\\
		& =h^{2\beta}\vert z\vert^{-(1-\beta)/2}\Vert \cL^{(1+\beta)/2-p}(z-\cL)^{-1}\cL^{p}f\Vert_{L^2(\cM)}
	\end{aligned}
\end{equation*}
Using then \Cref{eq:Ls} with $s=(1+\beta)/2-p\in [0,1]$, we retrieve
\begin{equation*}
	\begin{aligned}
		\Vert\cF_h(z)f\Vert_{L^2(\cM)}
		&\lesssim h^{2\beta}\vert z\vert^{-(1-\beta)/2} \vert z\vert^{-(1-(1+\beta)/2+p)}
		\Vert\cL^{p}f\Vert_{L^2(\cM)} \\
		& =h^{2\beta}\vert z\vert^{-(1-\beta+p)} 
		\Vert\cL^{p}f\Vert_{L^2(\cM)}.
	\end{aligned}
\end{equation*}
This concludes the proof.		
\end{proof}

\subsection{Proof of \Cref{lem:error_polyhedral}}

Based on the results in the previous subsections, we can now move on to the proof of \Cref{lem:error_polyhedral}.

\begin{proof}

Let $\tilde{f}\in S_h^\ell$. We introduce the inverse lift operator $\mathcal{P}_\ell : L^2(\cM) \rightarrow L^2(\cM_h)$  which maps any $F\in L^2(\cM) $ to $\mathcal{P}_\ell F = F^{-\ell}$. Let then $\cE_h=\big\| \big(\gamma(\cL_h)\tilde{f}\big)^{-\ell} - \gamma(\sL_h)\sP_h(\sigma\tilde{f}^{-\ell})\big\Vert_{L^2(\cM_{h})}
=\big\| \mathcal{P}_\ell\gamma(\cL_h)\tilde{f} - \gamma(\sL_h)\sP_h(\sigma\mathcal{P}_\ell\tilde{f})\big\Vert_{L^2(\cM_{h})}$. 
Note that by the integral representations of the operators~\eqref{eq:int_rep_operator} 
\begin{align*}
	\cE_h 
	=\big\|\frac{1}{2\pi i} \int_{\Gamma} \gamma(z)\cF(z)\tilde{f}\dd z \big\|_{L^2(\cM_h)},
\end{align*}
where we take for any $z\in\Gamma$, $\cF(z)=
\mathcal{P}_\ell(z-\cL_h)^{-1}-(z-\sL_h)^{-1}\sP_h\sigma\mathcal{P}_\ell$. Similarly, as in the proof of \Cref{prop:det}, we use the splitting~\eqref{def:Gamma} of $\Gamma$ and the triangle inequality to deduce that
\begin{equation}\label{eq:cEh}
	\begin{aligned}
		\cE_h 
		&\le
		\frac{1}{2\pi } \int_{\Omega_+} |\gamma(g_+(-t))| \left\|\cF(g_+(-t))\tilde{f}\right\|_{L^2(\cM_h)} \dd t
		+ \frac{\delta_0}{2\pi} \int_{\Omega_0} |\gamma(g_0(t))| \|\cF(g_0(t))\tilde{f}\|_{L^2(\cM_h)} \dd t\\
		& \qquad + \frac{1}{2\pi}\int_{\Omega_-} |\gamma(g_-(t))| \left\|\cF(g_-(t))\tilde{f}\right\|_{L^2(\cM_h)} \dd t,
	\end{aligned}
\end{equation}
where we take $\Omega_+=\Omega_-=[\delta_0,+\infty)$ and $\Omega_0=[-\theta,\theta]$. 
For $*\in\lbrace +, 0, -\rbrace$, let us then introduce the quantity
\begin{equation*}
	\cE_h^{*}=\int_{\Omega_*} |\gamma(g_*(t\varepsilon_*))| \left\|\cF(g_*(t\varepsilon_*))\right\|_{L^2(\cM_h)} \dd t,
\end{equation*}
where $\varepsilon_*=-1$ if $*=+$, and $\varepsilon_*=1$ otherwise.
Hence, \Cref{eq:cEh} may be rewritten as $\cE_h \lesssim  \cE_h^{+}+\cE_h^{0}+\cE_h^{-}$ and in particular, $g_*(t\varepsilon_*)\in \Gamma$ for any $t\in\Omega_*$. 

We now fix $*\in\lbrace +, 0, -\rbrace$ and bound the term $\cE_h^{*}$. First, for any $z\in\Gamma$, we rewrite $\cF(z)$ and split
\begin{align*}
	\cF(z)
	&=
	(z-\sL_h)^{-1}\sL_h\left((z\sL_h^{-1}-I)\mathcal{P}_\ell\cL_h^{-1}-\sL_h^{-1}\sP_h\sigma\mathcal{P}_\ell(z\cL_h^{-1}-I)\right)\cL_h(z-\cL_h)^{-1}\\
	&=(z-\sL_h)^{-1}\sL_h\left(z\sL_h^{-1}(I-\sP_h\sigma)\mathcal{P}_\ell\cL_h^{-1}+\sL_h^{-1}\sP_h\sigma\mathcal{P}_\ell-\mathcal{P}_\ell\cL_h^{-1}
	\right)\cL_h(z-\cL_h)^{-1}\\
	& = \cF_1(z) + \cF_2(z),
\end{align*}
where we take $\cF_1(z)=(z-\sL_h)^{-1}\sL_h\left(z\sL_h^{-1}(I-\sP_h\sigma)\mathcal{P}_\ell\cL_h^{-1}
\right)\cL_h(z-\cL_h)^{-1}=z(z-\sL_h)^{-1}(I-\sP_h\sigma)\mathcal{P}_\ell(z-\cL_h)^{-1}$ and $\cF_2(z)=(z~-~\sL_h)^{-1}\sL_h\left(\sL_h^{-1}\sP_h\sigma\mathcal{P}_\ell-\mathcal{P}_\ell\cL_h^{-1}
\right)\cL_h(z-\cL_h)^{-1}$. Hence, by the triangle inequality,
\begin{align}\label{eq:cEh2}
	\cE_h^* 
	&
	\lesssim \int_{\Gamma_t} |\gamma(g_*(t))|\big(\left\|\cF_1(g_*(t)) \tilde f\big\|_{L^2(\cM_h)}+ \big\|\cF_2(g_*(t)) \tilde f\big\|_{L^2(\cM_h)}\right) \dd t.
\end{align}

We first bound $\left\|\cF_1(z) \tilde{f}\right\|_{L^2(\cM_h)}$. 
Using successively \Cref{eq:A5bis} and the geometric estimates in  \cite[Corollary 2.2]{Bonito2018} results in
\begin{align*}
	\left\|\cF_1(z) \tilde{f}\right\|_{L^2(\cM_h)} &= |z| \left\|(z-\sL_h)^{-1}(I-\sP_h\sigma)\mathcal{P}_\ell(z-\cL_h)^{-1} \tilde{f} \right\|_{L^2(\cM_h)} \\
	&\lesssim  \left\|(I-\sP_h\sigma)\mathcal{P}_\ell(z-\cL_h)^{-1} \tilde{f}\right\|_{L^2(\cM_h)}
	\lesssim h^2 \left\|(z-\cL_h)^{-1} \tilde{f}\right\|_{L^2(\cM)}.
\end{align*}
Using then \Cref{eq:A3}, we conclude that, for any $p\in (0,2)$, 
\begin{equation}
	\left\|\cF_1(z) \tilde{f}\right\|_{L^2(\cM_h)} 
	\lesssim h^2 | z|^{-(1-p/2)}\left\|\cL_h^{-p/2} \tilde{f}\right\|_{L^2(\cM)}^2.\label{eq:f2}
\end{equation}

To bound $\big\|\cF_2(z) \tilde{f}\big\|_{L^2(\cM_h)}$,  we apply \Cref{eq:A2,eq:B} to obtain
\begin{align*}
	\left\|\cF_2(z) \tilde{f}\right\|_{L^2(\cM_h)} 
	&\lesssim |z|^{-1/2} \left\|\left( \sL_h^{-1}\sP_h\sigma\mathcal{P}_\ell-\mathcal{P}_\ell\cL_h^{-1}\right)\cL_h (z-\cL_h)^{-1} \tilde{f}\right\|_{H^1(\cM_h)}\\
	&\lesssim |z|^{-1/2} h^2 \left\|\cL_h^{1/2} (z-\cL_h)^{-1} \tilde{f}\right\|_{L^2(\cM)}\\
	&=|z|^{-1/2} h^2 \left\|\cL_h^{(1+p)/2} (z-\cL_h)^{-1}\cL_h^{-p/2} \tilde{f}\right\|_{L^2(\cM)}
\end{align*}
and with \Cref{eq:Lhs} (applied with $s=(1+p)/2$)
\begin{equation}
	\left\|\cF_2(z) \tilde{f}\right\|_{L^2(\cM_h)} 
	\lesssim h^2|z|^{-(1-p/2)}  \left\|\cL_h^{-p/2} \tilde{f}\right\|_{L^2(\cM)} \label{eq:f1}.
\end{equation}

Using \Cref{eq:f1,eq:f2}  with $p=\min\lbrace \alpha + d/4; 1\rbrace$ together with \Cref{eq:cEh2} gives
\begin{align*}
	\cE_h^* 
	&
	\lesssim \int_{\Omega_*} |\gamma(g_*(t\varepsilon_*))| h^2| g_*(t\varepsilon_*)|^{-(1-p/2)} \| \cL_h^{-p/2}\tilde{f}\Vert_{L^2(\cM)}\dd t,
\end{align*}
which yields in turn (since $\gamma$ is an $\alpha$-\psd)
\begin{align*}
	\cE_h^* 
	\lesssim  h^2\|  \cL_h^{-p/2} \tilde{f}\Vert_{L^2(\cM)}\int_{\Omega_*}| g_*(t\varepsilon_*)|^{-(1+\alpha-p/2)}  \dd t 
	\lesssim  h^2\|  \cL_h^{-\min\lbrace \alpha + d/4; 1\rbrace/2}\tilde{f}\Vert_{L^2(\cM)},
\end{align*}
since $\alpha-p/2=\max\lbrace \alpha - (\alpha+d/4)/2; \alpha -1/2\rbrace=\max\lbrace (\alpha-d/4)/2; \alpha -1/2\rbrace\ge (\alpha-d/4)/2>0$. Finally, since this inequality holds for any $*\in\lbrace +, 0, -\rbrace$, we retrieve the claim~\eqref{eq:err_pol_ld} using \Cref{eq:cEh}.
\end{proof}

\end{document}